\documentclass[10pt]{amsart}

\usepackage{amsthm,amsmath}%natbib
\usepackage{graphicx}

\usepackage{amsmath}
\usepackage{amsthm}
\usepackage{amsfonts}
\usepackage{amssymb}
\usepackage{color}
\usepackage{verbatim}
\usepackage{amscd}
\usepackage[latin1]{inputenc}
\usepackage{latexsym}
\usepackage{url}
\usepackage[colorlinks=true]{hyperref}
\usepackage{cleveref}
\usepackage{enumitem}
\usepackage{amsmath}
\usepackage{amssymb}
\usepackage{fancyhdr}
\usepackage{tikz}
\usepackage{amsthm}
\usepackage{relsize}
\usepackage{empheq}
\usepackage{mathrsfs}
\usepackage{bm}

\newtheorem{theorem}{Theorem}[section]
\newtheorem{corollary}{Corollary}[theorem]
\newtheorem{lemma}[theorem]{Lemma}
\newtheorem{proposition}[theorem]{Proposition}

\newtheorem{assumption}{Assumption}[section]
\theoremstyle{definition}
\newtheorem{definition}{Definition}[section]
\newtheorem{remark}{Remark}[section]
\newtheorem{example}{Example}[section]
\numberwithin{equation}{section}

\newcommand{\one}{\mathbf{1}}
\newcommand{\abs}[1]{\left\lvert#1\right\rvert}
\newcommand{\norm}[1]{\left\lVert#1\right\rVert}
\newcommand{\real}{\mathbb{R}}
\newcommand{\nat}{\mathbb{N}}

\newcommand{\rat}{\mathbb{Q}}

\newcommand{\inprod}[2]{\left\langle#1,\,#2\right\rangle}
\newcommand{\set}[1]{\left\{#1\right\}}

\newcommand{\paran}[1]{\left(#1\right)}

\newcommand{\convergeto}{\longrightarrow}
\newcommand{\prob}[1]{\mathbb{P}\left(#1\right)}
\newcommand{\expect}[1]{\mathbb{E}\left[#1\right]}
\newcommand{\cX}{\mathcal{X}}
\newcommand{\xmap}{\Phi}
\newcommand{\cY}{\mathcal{Y}}
\newcommand{\ymap}{\Psi}
\newcommand{\Z}{\mathcal{Z}}

\newcommand{\U}{\mathcal{U}}

\newcommand{\T}{\mathbb{T}}
\newcommand{\K}{\mathcal{K}}
\newcommand{\R}{\mathbb{R}}
\renewcommand{\S}{\mathcal{S}}
\newcommand{\F}{\mathcal{F}}
\newcommand{\M}{\mathcal{M}}
\newcommand{\I}{\mathcal{I}}
\newcommand{\J}{\mathcal{J}}
\newcommand{\cPP}{\mathcal{P}}

\newcommand{\tF}{\widetilde{\F}}

\newcommand{\ttheta}{\tilde{\theta}}

\newcommand{\relent}[2]{K(#1\,\|\,#2)}
\newcommand{\dynrelent}[2]{h(#1\,\|\,#2)}
\newcommand{\dist}[2]{\text{dist}(#1, #2)}
\newcommand{\cL}{\mathcal{L}}

\newcommand{\cG}{\mathcal{G}}

\newcommand{\supp}{\text{supp}}

\newcommand{\partition}{\mathscr{Z}}

\begin{document}

% \begin{frontmatter}



\title[Large deviations and posterior consistency]{Large Deviation Asymptotics and Bayesian Posterior Consistency on Stochastic Processes and Dynamical Systems}

 %\author{Langxuan Su and Sayan Mukherjee}
%\affiliation{organization={Department of Mathematics},
\author{Langxuan Su}
\author{Sayan Mukherjee}
% \footnote{%
%     $^1$Department of Mathematics, Duke University, NC\\%
%     $^2$ Departments of Statistical Science, Computer Science, Biostatistics \& Bioinformatics, Duke University, NC\\%
%     $^3$ Center for Scalable Data Analytics and Artificial Intelligence, Universit\"at Leipzig\\%
%     $^4$ Max Planck Institute for Mathematics in the Sciences, Leipzig \\%
%     ~\\
% %    \today
% }

%\author{Langxuan Su}
%\affiliation{{Duke University},
%\author{Sayan Mukherjee}

\maketitle

\begin{abstract}
We consider generalized Bayesian inference on stochastic processes and dynamical systems with potentially long-range dependency. Given a sequence of observations, a class of parametrized model processes with a prior distribution and a loss function, we specify the generalized posterior distribution.
The problem of frequentist posterior consistency
is concerned with whether as more and more samples are observed,
the posterior distribution on parameters will asymptotically concentrate on the ``right" parameters.
We show that posterior consistency can be derived using a combination of classical large deviation techniques, such as Varadhan's lemma, conditional/quenched large deviations, annealed large deviations, and exponential approximations.
We show that the posterior distribution will asymptotically concentrate on parameters that minimize the expected loss and a divergence term, and we identify the divergence term as the Donsker-Varadhan relative entropy rate from process-level large deviations. As an application, we prove new quenched and annealed large deviation asymptotics and new Bayesian posterior consistency results for a class of mixing stochastic processes. 
In the case of Markov processes, one can obtain explicit
conditions for posterior consistency, whenever estimates
for log-Sobolev constants are available, which
makes our framework essentially a black box.
We also recover state-of-the-art posterior consistency on classical dynamical systems with a simple proof. Our approach has the potential of proving posterior consistency for a wide range of Bayesian procedures in a unified way.

\end{abstract}


% \begin{highlights}
% \end{highlights}

% \begin{keyword}
% Bayesian posterior consistency, large deviations


% \end{keyword}

% \end{frontmatter}



\begin{center}
    \vspace{0.5cm}
    \textit{LS dedicates this to his father Weibo Su for his endurance.}
\end{center}

\section{Introduction} \label{Sect:Intro}

In this paper, we provide asymptotic results concerning (generalized) Bayesian inference for certain stochastic processes and dynamical systems based on a large deviation approach. The elements of generalized Bayesian inference include a sequence of observations $y$, a class of model processes parameterized by $\theta \in \Theta$ which can be characterized as a stochastic process $X^\theta$ or a measure $\mu_\theta$, a prior distribution $\pi_0$ on $\Theta$, and a loss function $L$ which measures the error between $y$ and a realization of $X^\theta$. These elements together allow us to specify the generalized posterior distribution $\pi_t(\theta \mid y)$ on the parameters, where $t$ denotes number of samples or the length of the observed trajectory based on $y$. The goal of this paper is to study the asymptotic behavior of $\pi_t(\theta \mid  y)$ as $t \to \infty$, such as under what conditions it will asymptotically concentrate on the ``right" set of parameters so that the Bayesian procedure is justified. This is known as the problem of posterior consistency. In particular, we state conditions on the model family $\{\mu_\theta\}_{\theta \in \Theta}$ and the loss function $L$ such that the posterior distribution converges, and we characterizes the set of parameters the posterior distribution asymptotically concentrates on as those minimize the expected loss and a divergence term. 

The study of posterior consistency is known to be technically involved in the statistics community. Most of the existing literature on posterior consistency focuses on very specific cases of the underlying data-generating processes, and Bayesian procedures.
Each case is treated with many ad hoc arguments that do not generalize,
and this limits the applicability. 
In addition, very few posterior consistency results exist for dynamical systems that
involves potentially long-range dependency. 
It was unclear how posterior consistency connects with other well-studied properties of dynamical systems, such as mixing conditions.

% In generalized Bayesian inference, we only require the existence of a loss function that can be used to assess how well a sequence generated by a model process fits the observed data sequence. In contrast, Bayesian inference requires knowledge of the likelihood or the data generation process, and it is generally assumed that the data generation process belongs to the family of candidate model processes.  We consider generalized Bayes procedures for two reasons: (1) the large deviation approach to Bayesian inference is more natural for considering general loss functions, and (2) when studying dynamical or other complex systems, it may be unrealistic to assume knowledge of the data generating system and difficult to verify that the data generating process is in the model class specified. 

% The two conditions we require are: (1) a conditional large deviation behavior for a single
% $X^\theta$, and (2) an  exponential  continuity condition over the model family for the map from the parameter $\theta$ to the loss incurred between $X^\theta$ and the observation sequence $y$. 
The central contribution of this paper is to provide a flexible proof framework that can be applied to a variety of stochastic processes and dynamical systems including both discrete and continuous state spaces and discrete-time and continuous-time dynamics. 
We simplify proving posterior consistency to checking for properties that have been well-studied in the large deviation literature. As the proposed framework can be applied to a wide class of processes, we are providing new proof techniques that can be used to analyze posterior consistency for novel stochastic models. We provide some evidence of the generality of our procedure via applications to continuous-time hypermixing processes and Gibbs processes on shifts of finite type.
In particular, posterior consistency of hypermixing processes has 
not been covered by the current literature.
The point is that the same procedure for proving posterior consistency can be used for two very different dynamical systems, and this procedure potentially connects posterior consistency with well-studied properties of dynamical systems. 

To obtain posterior consistency of Bayesian inference of hypermixing processes, we prove new quenched and annealed large deviation asymptotics. In the special case of Markov processes,
we make explicit connections with the log-Sobolev inequality
and hypercontractivity, where our sufficient conditions
for posterior consistency can be seamlessly identified.
As a result, whenever there is new progress in estimating log-Sobolev constants, a new posterior consistency result will follow,
which makes our framework essentially a black box in 
the case of Markov processes.

A key result in our framework is the identification of the 
divergence term more suitable for dynamical systems, which governs the large deviation asymptotics of the posterior distributions. 
This divergence term coincides with the classical
Donsker-Varadhan relative entropy rate in
process-level large deviation principles,
which can be defined on a variety of 
stochastic processes and dynamical systems. It also
connects to other classical quantities in ergodic 
theory, such as the Kolmogorov-Sinai entropy. 
This divergence allows us to formulate and prove posterior consistency
for a wide range of stochastic processes and dynamical
systems. 

\subsection{Connections to previous work}\label{Sect:Previous}

Large deviations have been applied to Bayesian procedures including posterior distributions of i.i.d. random variables \cite{ganesh1999inverse}, Markov chains \cite{papangelou1996large, eichelsbacher2002bayesian},
Dirichlet processes \cite{ganesh2000large}, and empirical likelihoods \cite{grendar2007bayesian}. These papers
do not specifically consider posterior consistency and do not provide a general methodology for proving
posterior consistency, rather they focus on deriving the large deviation rate function for a particular Bayesian
procedure.

Posterior consistency for i.i.d. processes is well-studied, going back to initial results by 
Doob \cite{Doob49CNRS} and Schwartz \cite{Schwartz1965}. Counterexamples and challenges to proving
posterior consistency for nonparametric models were highlighted by Diaconis and Freedman in \cite{diaconis1998}.
Posterior consistency for Bayesian nonparametric models is an active area of research, for a detailed review see \cite{Ghosal2017}.
There are far fewer results for posterior consistency for dependent processes or dynamical systems. Posterior 
consistency for hidden Markov models was established in \cite{Chopinetal2015,douc2020posterior,Gassiat2014,Vernet2015}. For dynamical systems, \cite{mcgoff2019gibbs} established
posterior consistency of (hidden) Gibbs processes on mixing subshifts of finite type using properties of Gibbs measures, and \cite{lopes2020bayes} directly used large deviation results to prove a similar result.

The setting of this paper is similar to \cite{shalizi2009dynamics}, where the author establishes posterior consistency by proving the large deviation principle for the posterior distribution of some dependent processes with likelihood ratios instead of loss functions. There are technical difficulties in verifying the main assumptions in \cite{shalizi2009dynamics}, in particular when the state space is neither finite nor compact. In addition, \cite{shalizi2009dynamics} does not provide any procedures to check the required assumptions. Our work can be seen as a general framework for verifying those assumptions. Additionally, the methods used in both this paper and \cite{shalizi2009dynamics} are
similar to the entropy-based argument in \cite{mcgoff2015consistency} based on \cite{barron1985strong}, which originates from \cite{douc2011consistency}.

From a large deviation perspective, we will make extensive use of two techniques in large deviation theory.
We will require large deviation behavior on each model process conditioning on a given observation $y$. In the large deviation literature, this is closely related to the conditional large deviation principle or the quenched large deviation principle which has been studied for a variety of stochastic models \cite{seppalainen1994large,comets2000quenched,chi2002conditional}. We will also require a regularity condition on the map from parameters to model processes called exponential continuity, 
which is adapted from the work on the large deviation of mixtures or
the annealed large deviation principle. This formalism has been widely studied for a variety of
processes \cite{comets2000quenched,biggins2004large,dinwoodie1992large, wu2004large}, with
exchangeable sequences the most notable from the literature \cite{dinwoodie1992large, wu2004large}.
The idea of exponential continuity we use comes from \cite{dinwoodie1992large},
and it is also implicit in the study of uniform large deviation principles
of Markov processes with different initial conditions
in \cite{dupuis2011weak, budhiraja2019analysis}.

We will see later the variational perspective of Bayesian inference is critical. The idea of a variational formulation of Bayesian inference was developed by Zellner \cite{Zellner1988}, and the link between statistical mechanics and information theory with Bayesian inference was at the heart of the inference framework advocated by Edwin T. Jaynes 
\cite{Jaynes1973-JAYTWP}. In the variational perspective Bayesian inference is an optimization procedure over distributions that minimizes an error term and a regularization term that enforces the posterior to be close to the prior. Generalized Bayes inference refers to procedures for updating prior beliefs where the loss may not necessarily be the log-likelihood and the regularization term 
need not be the relative entropy to the prior. The idea of using loss functions to update beliefs goes back at least to Vovk \cite{Vovk}. It has played a central role in the PAC-Bayesian  approach to
statistical learning \cite{McAllester,catoni2007pac} and has been adopted by the mainstream Bayesian community \cite{Bissiri2016,MillerDunson}. 
In \cite{Grunwald2016}
% \cite{Grunwald2016}
consistency and rates 
of convergence are obtained for
generalized Bayesian methods. 

\subsection{Overview}

%\textcolor{red}{Placeholder. Divided remainder of this section into three sections.  
%Need to update section and subsection headings and labels.}

In Section \ref{Sect:ObservedSystem}, we provide notations, the setup of our problem and the main results of the paper. In Section \ref{sec:LDA}, we outline our large deviation framework to prove posterior consistency and provide both details and examples for the steps required and the various mathematical quantities involved. 
In Section \ref{sec:ex_inverse}, as a sanity check, we apply our framework to Bayesian inverse problems. 
In Section \ref{sec:ex_hypermixing} and \ref{sec:ex_Gibbs}, we apply our large deviation framework to two very different dynamic models: continuous-time hypermixing stochastic processes, including Markov processes, and Gibbs processes on shifts of finite type. We close with a discussion.
We also provide appendices for deferred proofs, background material and auxiliary technical results.

\section{The setting} 
\label{Sect:ObservedSystem}

Our inference framework consists of two components. The first component is the dynamical system that generated the observations. The second component is a parameterized family of  stochastic processes for which (generalized) Bayesian updating provides a posterior distribution conditioning on the observations.
We will denote $\T$ as the time index. In this chapter, we consider both discrete $\T = \nat$ time indices, equipped with discrete topology, and continuous  $\T = [0, \infty)$ time indicies, equipped with Euclidean topology. In both cases, $\T$ is equipped with its Borel $\sigma$-algebra.
%We focus on the continuous time horizon $\T = [0, \infty)$, but the same results hold for the discrete time horizon $\T = \nat$ by replacing integrals in time with summations and making other similar straightforward changes.

\subsection{Observed system}

% We will denote the observation process as $Y$. 
Let $\cY$ be a Polish space equipped with the Borel $\sigma$-algebra. We denote
 an one-parameter family of measurable transformations acting on $\cY$ by $(\ymap^t)_{t \in \T}$, i.e.,  for each $t \in \T$, $\ymap^t: \cY \to \cY$ is a measurable map, 
the map $\T \times \cY \to \cY, \,(t,y) \mapsto \ymap^t y$ is jointly measurable, 
%  $\ymap^0 = \text{Id}$, 
 and $\ymap^{t + s} = \ymap^t \circ \ymap^s.$ We recall a few definitions from ergodic theory.
A Borel probability measure $\nu$ on $\cY$ is said to be $\ymap$-invariant if $\nu((\ymap^t)^{-1} E) = \nu(E)$ for
any $t \in \T$ and any measurable set $E \subset \cY$. A set $E \subset \cY$ is said to be $\ymap$-invariant
if $(\ymap^{t})^{-1}(E) = E$ for any $t \in \T$. A $\ymap$-invariant Borel probability measure $\nu$ is
called $\ymap$-ergodic if $\nu(E) \in \set{0,1}$ for any $\ymap$-invariant set $E \subset \cY$.
A function $f: \cY \to [-\infty, \infty]$ is called $\ymap$-invariant if 
for $\nu$-almost every
$y \in \cY$, $f(y) = f(\ymap^t y)$ for any $t \in \T$.
In the following, we assume that observations come from a dynamical system $(\cY, \ymap, \nu)$
where $\nu$ is $\ymap$-ergodic. 
% We denote the observed sequence in the discrete case as $y_{0;t} := \{y, \ymap y, \ymap^2 y,  \ldots, \ymap^{t-1} y\}$ and in the continuous case
% $y_{0:t} := \{\ymap^s y\}_{s \in [0,t]}$. 
% is indexed by time $0$ to $t$. 

In the continuous time case $\T = [0, \infty)$, we make an additional assumption that 
$(\ymap^t)_{t \in \T}$ is an one-parameter family of \emph{continuous} transformations so that
the map $(t,y) \mapsto \ymap^t(y)$ is continuous. This ensures the integrals in time in
this paper are all well-defined.

\subsection{Parameterized model processes}

As we consider Bayesian inference in this paper, our objective is to obtain a posterior distribution over a parameterized family of processes which quantifies the evidence that each of the model processes in the family could have generated the observed data. We specify
the family of processes as $(X^\theta_t)_{t \in \T}$ indexed by parameter $\theta \in \Theta$ and our inferential goal is to obtain a posterior distribution $\pi_t(\theta \mid y)$ given observations $y$ up to time $t$. Before we rigorously state the form of the posterior, we formally specify the general form of the model processes we consider in this paper.

Let $\S$ be a Polish space with the Borel $\sigma$-algebra. Let $\cX = C(\T, \S)$ be the space of continuous paths $x: t \mapsto x_t$ on $\S$ equipped with the topology of uniform convergence on compact intervals of $\T$, so $\cX$ is also a Polish space under this topology with a compatible metric $d_\cX$. The natural one-parameter family of continuous transformations acting on $\cX$ is the family of time shift  maps $(\xmap^t)_{t \in \T}$ given by $\xmap^t: \cX \to \cX$,
$(\xmap^t(x))_s = x_{t+s}$. Thus, $\cX$ can be viewed as the canonical probability space of continuous $\S$-valued stochastic processes. 
% For the discrete setting $\T = \nat$ we will also denote as $\cX$ the discrete process analogous to the continuous process above, basically the 
% for the time shift map $s,t \in \nat$.
Let $\Theta$ be a compact metric space of parameters with the metric denoted by 
$d_\Theta$. One can view the family of model processes as a family of $\xmap$-invariant measures $\{\mu_\theta\}_{\theta \in \Theta}$ on $\cX$, perhaps a more dynamical perspective. A more probabilistic view considers  a parameterized family of $\S$-valued  stationary stochastic processes 
$\{X^\theta = (X^\theta_t)_{t \in \T}\}_{\theta \in \Theta}$ 
with corresponding laws $\{\mu_\theta\}_{\theta \in \Theta}$ on $\cX$. Both perspectives will be considered in this paper. From the probabilistic view point, 
we have the freedom to choose the underlying probability space  $(\Omega, \F, \mathbb{P})$ of the stochastic process $X^\theta$. 
% A realization from the process $X^\theta$ from time zero to $t$ will be denoted as $x^\theta_{0:t}$.
\begin{remark}[Parametric and non-parametric inference]
    It may seem that we are working in a parametric setting,
    since each model process depends on a parameter $\theta$.
    However, non-parametric inference is basically
    parametric inference when the parameter space is
    infinite-dimensional, such as functions. We do not assume
    $\Theta$ to be finite-dimensional, so our setting
    includes non-parametric inference.
\end{remark}

\begin{remark}[Compactness of $\Theta$]
    It is quite common to assume the parameter space to be compact in the posterior consistency literature.
    In the discussion section at the end, we briefly
    provide a reason why compactness is necessary from the 
    large deviation perspective. Even so, compactness of $\Theta$
    is only used in the last step of large deviation framework.
    Also, the large deviation asymptotics is strictly stronger than posterior consistency. 
    We believe that the current framework can be extended to non-compact
    $\Theta$, for example, by assuming some exponential decay of 
    the prior distribution $\pi_0$.
\end{remark}

\subsection{Posterior inference}
In Bayesian inference, the quantity of interest is the posterior distribution on the parameters given our observation sequence
$y$ up to time $t$,
\begin{eqnarray*}
 \pi_t(\theta \mid y) 
 &=& 
\frac{\mbox{Lik}^t(y \mid \theta) \times \pi_0(\theta)}
        {\int_{\theta' \in \Theta}  \mbox{Lik}^t(y \mid \theta') \times \pi_0(\theta') \, d\theta' },
\end{eqnarray*}
where $\pi_0(\theta)$ is the prior distribution or the belief over the parameters $\theta \in \Theta$, $\mbox{Lik}^t(y \mid \theta)$ is the likelihood of observing the data $y$ given
parameter $\theta$ up to time $t$, and for the purpose of this paper, the denominator is a normalizing constant (in Bayesian inference it is the marginal probability of the data $y$). Another formulation of Bayesian inference focuses on the loss function $l^t(y;\theta) = -\log \mbox{Lik}^t(y \mid \theta)$ which measures the incurred error or loss on the data $y$ given a parameter $\theta$ up to time $t$. The posterior in this setting is
\begin{eqnarray*}
  \pi_t(\theta \mid y) 
  &=& 
  \frac{ \exp(-l^t(y;\theta)) \times \pi_0(\theta)}
         {\int_{\theta' \in \Theta} \exp(-l^t(y;\theta')) \times \pi_0(\theta') \, d\theta' }.
  \end{eqnarray*}
We will focus on the above loss-based formulation of Bayesian inference
because it makes more transparent the connections between large deviation theory and posterior consistency, and it allows us to provide results for the wider class of generalized Bayesian updating procedures. 

We consider a general loss function 
$L: \Theta \times \cX \times \cY \to \real, \, (\theta, x, y) \mapsto L_\theta(x,y)$,
where the path $x \in \cX$ can be interpreted as the underlying state of the observed 
data.
We define the (integrated) loss up to time $t$ for the continuous and discrete cases as 
$$L^t_\theta(x,y) := \int_0^t L_\theta(\xmap^s x, \ymap^s y)  \, ds, \quad L^t_\theta(x,y) := \sum_{s=0}^{t-1} L_\theta(\xmap^s x, \ymap^s y).$$
% The two properties of the loss function that will directly impact our large deviation perspective is that the loss effects the posterior at an exponential scale and we will
% require continuity of the map $(\theta, x, y) \mapsto L(x,y;\theta)$.
We now formally state the posterior distribution of interest in this paper. We assume the space of parameters $\Theta$ is equipped with its Borel $\sigma$-algebra and a fully supported prior measure $\pi_0$. Given an observation $y$, the posterior distribution is
\begin{equation}
    \pi_t(E  \mid  y)
     = \frac{1}{\partition^{\pi_0}_t(y)} \int_{E \times \cX} 
    \exp\left(-L^t_\theta(x, y)\right)
    d\mu_\theta(x)\,d\pi_0(\theta),
\end{equation}
where $E \subseteq \Theta$ and the normalization constant or partition function is
\begin{equation}
\label{partition}
\partition^{\pi}_t(y) =  \int_{\Theta \times \cX} 
    \exp\left(-L^t_{\theta'}(x, y)\right)
    d\mu_{\theta'}(x)\,d\pi(\theta'),  
\end{equation}
when $\pi$ is any Borel probability measure on $\Theta$.
Again, when $L_\theta(x,y)$ is the negative log-likelihood of $y$ given $x$ and $\theta$,
one recovers Bayesian inference.

The goal of this paper is to study the asymptotic behavior of the 
posterior distribution $\pi_t(\cdot \mid  y)$ as $t \to \infty.$
In particular, we want to understand under what conditions on
the model family $\{\mu_\theta\}_{\theta \in \Theta}$ and the loss function $L$, 
the posterior distribution $\pi_t(\cdot \mid y)$ converges, and around what set of parameters
it concentrates asymptotically.
%On what sets of parameters does $\pi_t(\cdot \,|\, y)$ concentrate as $t \to \infty$?
\begin{remark}[Support of $\pi_0$]
    It is not necessary to assume that $\pi_0$ is fully supported.
    In that case, one can prove the main results of the paper on
    $\supp(\pi_0)$ instead of the whole $\Theta$ in the same way.
\end{remark}

\subsection{Main results}
We develop a large deviation framework to study the asymptotic
behavior of the posterior distribution $\pi_t(\cdot \mid  y)$. 
We introduce the notion of exponentially continuous
families for describing the regularity of the parametrization 
map $\theta \mapsto \mu_\theta$
with respect to the loss $L$ and prove the following
variational characterization of exponential asymptotics of the partition
function.
\begin{theorem}
    Suppose $\set{\mu_\theta}_{\theta \in \Theta}$ is an exponentially continuous
    family with respect to the loss function $L$. Then
    there exists a continuous function
    $V: \Theta \to \real$ such that for any Borel probability measure
    $\pi$ on $\Theta$, for
    $\nu$-almost every $y \in \cY$,
    \begin{equation*}
        \lim_{t \to \infty} -\frac{1}{t}\log \partition_t^{\pi}(y) 
        = \inf_{\theta \in \supp(\pi)} V(\theta).
    \end{equation*}
\end{theorem}

As a corollary from \cite[Theorem 2]{mcgoff2019gibbs}, we obtain the convergence of
the posterior distribution $\pi_t(\cdot \mid  y)$ 
as $t \to \infty$ and characterize
the set it concentrates on asympotically.
\begin{corollary}
    Let $\Theta_{\min}$ be the set of minimizers of $V$ in $\Theta$.
    Then if $U$ is an open neighborhood of $\Theta_{\min}$, 
    for $\nu$-almost every $y \in \cY$, we have 
    \[
        \lim_{t \to \infty} \pi_t(\Theta \setminus U \,|\, y) = 0.
    \]
\end{corollary}
To generalize previous results on posterior consistency 
and demonstrate the flexibility of our framework,
we provide natural sufficient conditions 
of exponential continuity for two classes of very different 
model processes: (1) a class of continuous-time, dependent processes
on general state spaces called hypermixing processes; (2)
a class of discrete-time, discrete-valued dynamical systems
called Gibbs processes on shifts of finite type. 
For these two classes of examples, we also characterize 
$V$ explicitly so that the posterior distribution
concentrates asymptotically to those parameters
that minimize the sum of the expected loss and
a divergence term, thus proving posterior consistency.

In the case of hypermixing processes,
we prove the existence of the relative entropy rate
and obtain new quench and annealed large deviation
asymptotics, which may be of independent interests.
For the special case of Markov processes,
we derive explicit and checkable conditions for
exponential continuity based on the connections 
between hypermixinig, hypercontractivity, 
and the log-Sobolev inequality, which makes
our framework essentially a black box.

\subsection{Notations}

Suppose $\Z$ is a Polish space equipped with the Borel $\sigma$-algebra and $(R^t)_{t \in \T}$ is an one-parameter family of (continuous if $\T = [0,\infty)$) transformations
on $\Z$. The following objects appear throughout the paper.\\

\noindent {\em Measures}: Let $\M(\Z)$ denote the vector space of finite Borel measures on $\Z$,  $\M_1(\Z)$ the subset of Borel probability measures on $\Z$, and $\M_1(\Z, R)$ the subset of $R$-invariant Borel probability measures.  The topology we consider is the weak convergence of measures.
In particular, $\M_1(\Z)$ is a Polish space under this topology. 
Let $\M_\nu(\cX \times \cY)$ denote the set of Borel probability measures on $\cX \times \cY$ with $\cY$-marginal $\nu$.
For $z \in \Z$, we write the empirical process of $z$ on $\Z$ in continuous and
discrete time as
    \begin{equation*} 
         M_t(z) := \frac{1}{t}\int_0^t \delta_{R^s z}\,ds, \quad
          M_t(z) := \frac{1}{t}\sum_{s=0}^{t - 1} \delta_{R^s z},
    \end{equation*}
    where $\delta_z$ is the Dirac Delta measure at $z$. \\
    
\noindent {\em Sigma algebras}: Consider $\cX$ and $\cY$ as canonical probability spaces for the stochastic process $X = (X_t)_{t \in \T}$ and the random variable $Y$ (defined by $X_t: x \mapsto x_t$ and $Y: y \mapsto y$).  We define the $\sigma$-algebra $\F_t = \sigma(X_s : 0 \le s \le t)$ on $\cX$ as generated by $X$ up to time $t$ and
$\tF_t = \sigma(X_s, Y :  0 \le s \le t)$ on $\cX \times \cY$ as generated by $X$ up to time $t$ with full information of $Y$. \\
    
\noindent {\em Functions}: Let $C_{b}(\Z)$ denote the set of bounded continuous functions $f: \Z \to \real$, $C_{b,loc}(\cX)$ the set of 
functions $f \in C_{b}(\cX)$ such that $f$ is $\F_r$-measurable for some $r \ge 0$, and $C_{b,loc}(\cX \times \cY)$ the set of 
functions $f \in C_{b}(\cX \times \cY)$ such that $f$ is $\tF_r$-measurable for some $r \ge 0$. 

For a continuous function $f: \Z \to \real$ and interval $I \subset \T$, we write in continuous time and
discrete time, respectively, 
$$ f^I(z) := \int_I f(R^s z)\,ds, \quad f^I(z) := \sum_{s \in I} f(R^s z) $$ and we write $f^t := f^{[0,t]}$, $f^t := f^{\set{0, 1, \ldots, t-1}}$, respectively. 
We denote the sup norm of $f$ by $\|f\|$. \\

\noindent {\em Joint processes}: Let $\J(\xmap : \nu)$ denote the set of $(\xmap \times \ymap)$-invariant probability measures with $\cY$-marginal $\nu$, and $\J_e(\xmap : \nu)$ the $(\xmap \times \ymap)$-ergodic elements of $\J(\xmap : \nu)$.

\section{A large deviation framework for posterior consistency}\label{sec:LDA}
\noindent
In this section, we present a framework of proving posterior convergence using a large deviation approach. We start with the observation
in \cite{mcgoff2019gibbs} that there is a variational characterization that implies posterior consistency. Starting with this observation, the large deviation approach to posterior consistency is based on showing the following:
\begin{enumerate}
    \item a conditional large deviation behavior of some empirical process on $\cX \times \cY$;  this allows us to prove the variational characterization for a single model process;
    \item an exponential continuity condition over the model family, adapted from the large deviation theory of mixtures; this allows us to prove the variational characterization over the entire model family.
\end{enumerate}

%We start with a variational characterization that implies posterior
%consistency, as shown in \cite{mcgoff2019gibbs}. Next, we observe that
%in the case of a single parameter, this variational characterization is
%a consequence of the conditional large deviation behavior of some empirical process
%on $\X \times \Y$. Finally, we introduce the exponential continuity condition
%on the model family, adapted from the large deviation theory of mixtures, 
%which allows us to prove
%the variational characterization on the whole parameter space.  

\subsection{A variational formulation for posterior convergence}

A central idea in \cite{mcgoff2019gibbs} was that proving posterior consistency can be reduced to proving a variational
characterization of the asymptotics of the normalizing constant $\partition_t^\pi$ at an exponential scale. The results in
\cite{mcgoff2019gibbs} were for a specific model family, Gibbs processes on mixing subshifts of finite type. A key point in this paper
is that the variational characterization of the asymptotics of normalizing constant at an exponential scale will hold for a wide class
of processes, and this characterization can be used to prove posterior consistency. 

We start by restating Theorem 1 and 2 in \cite{mcgoff2019gibbs} using the notation for model processes in our paper.
\begin{theorem}[\cite{mcgoff2019gibbs}] \label{thm:posterior}
    Suppose there exists a lower semicontinuous function
    $V: \Theta \to \real$ such that for any Borel probability measure
    $\pi$ on $\Theta$, for
    $\nu$-almost every $y \in \cY$,
    \begin{equation}\label{eq:variational}
        \lim_{t \to \infty} -\frac{1}{t}\log \partition_t^{\pi}(y) 
        = \inf_{\theta \in \supp(\pi)} V(\theta).
    \end{equation}
    Let $\Theta_{\min}$ be the set of minimizers of $V$ in $\Theta$.
    Then if $U$ is an open neighborhood of $\Theta_{\min}$, 
    for $\nu$-almost every $y \in \cY$, we have 
    \[
        \lim_{t \to \infty} \pi_t(\Theta \setminus U \,|\, y) = 0.
    \]
\end{theorem}
In other words, to show posterior consistency, it suffices to 
prove \eqref{eq:variational} for any probability measure $\pi$ on $\Theta$.
Moreover, Theorem \ref{thm:posterior} reveals that 
the generalized posterior distribution concentrates asymptotically on
minimizers of the function $V$. A crucial
observation in \cite{mcgoff2019gibbs} was that $V$ can be explicitly stated
when the model processes are Gibbs processes on mixing subshifts of finite type:
\begin{equation} \label{eq:v}
    V(\theta) = \inf_{\lambda \in \J(\xmap : \nu)} \set{\int L_\theta \,d\lambda + h(\lambda\,\|\,\mu_\theta \otimes \nu)}, 
\end{equation}
where $L_\theta$ is the map $(x,y) \mapsto L_\theta(x,y)$, 
$\J(\xmap : \nu)$ are the joint processes defined before, and $h(\lambda\,\|\,\mu_\theta \otimes \nu)$ measures the divergence
of $\lambda$ from $\mu_\theta \otimes \nu$, defined explicitly later. 
Intuitively, when a parameter $\theta$ is given
without further information of the dynamics on $\cX$ and $\cY$, the 
natural guess for the joint distribution on $\cX$ and $\cY$
is the independent coupling $\mu_\theta \otimes \nu$.
From this perspective, we may interpret the parameters in $\Theta_{\min}$
as those minimizing the expected loss and the divergence from the prior belief. 
We provide a short proof of Theorem \ref{thm:posterior}
in Section \ref{append:var}. The proof is very similar to the proof of in Theorem 2 in \cite{mcgoff2019gibbs}, 
but we added it to provide a proof that uses our notation and to be self-contained.

\subsection{The process-level LDP and the relative entropy rate}
In this subsection, 
we provide a variational characterization of the partition function 
for a single parameter $\theta$ as a direct application
of large deviation theory.
More specifically, we state the partition function specified in \eqref{partition} 
for a fixed parameter $\theta \in \Theta$ (consider $\pi = \delta_\theta$ on $\Theta$):
\begin{equation}
\label{single_part}
\partition_t^\theta(y) := \partition_t^{\delta_\theta}(y) 
=  \int_{\cX} \exp\left(-L_{\theta}^t(x, y)\right) d\mu_{\theta}(x).
\end{equation}
Then the variational characterization \eqref{eq:variational}
for a fixed $\theta$ reduces to
\begin{equation}
\label{eq:var-single}
\lim_{t \to \infty} \frac{1}{t} \log \partition_t^\theta(y) = - V(\theta),
\end{equation}
where $V(\theta)$ has the form of \eqref{eq:v}.
We will show that \eqref{eq:var-single} is implied by 
some large deviation principle.
 
The large deviation principle quantifies asymptotic behavior of
probability measures on an
exponential scale through the rate function via the following definition.
\begin{definition}[Large Deviation Principle (LDP)]
    Let $\Z$ be a Polish space and $\I: \Z \to [0, \infty]$ be a
    lower semicontinuous function.
    A family $(\eta_t)_{t \in \T}$ of probability measures on $\Z$
    is said to satisfy the large deviation principle with 
    rate function $\I$ if for every closed set $E \subset \Z$,
    \begin{equation} \label{eq:ldpub}
        \limsup_{t \to \infty} \frac{1}{t} \log \eta_t(E) \le - \inf_{z \in E} \I(z),
    \end{equation}
    and for every open set $U \subset \Z$,
    \begin{equation} \label{eq:ldplb}
        \liminf_{t \to \infty} \frac{1}{t} \log \eta_t(U) \ge - \inf_{z \in U} \I(z).
    \end{equation}
    In particular, we say that a stochastic process $(Z_t)_{t \in \T}$ on $\Z$
     satisfies the large deviation principle with 
    rate function $\I$ if the law of $(Z_t)_{t \in \T}$ on $\Z$ does.
\end{definition}
Note that the variational characterization of the scaled log partition function in equation \eqref{eq:variational} takes a similar form as the
above LDP. The connection is even more explicit when we consider Varadhan's Lemma \cite[Theorem 1.5]{budhiraja2019analysis} which
states the following equivalent definition of the large deviation principle.
\begin{theorem}[Varadhan's Lemma] \label{thm:laplace}
    Let $\Z$ be a Polish space and $\I: \Z \to [0, \infty]$ be a
    lower semicontinuous function.
    A stochastic process $(Z_t)_{t \in \T}$ on $\Z$
     satisfies the large deviation principle with 
    rate function $\I$, then it satisfies the Laplace principle, i.e.,
    for all  $F \in C_b(\Z)$, we have 
    \begin{equation} \label{eq:laplace}
        \lim_{t \to \infty} \frac{1}{t} \log \expect{\exp(-t F(Z_t))}
        = - \inf_{z \in \Z}(F(z) + \I(z)).
    \end{equation}
    % In fact, each of \eqref{eq:ldpub} and \eqref{eq:ldplb} implies 
    % one side of inequality.
    In fact, the large deviation upper bound \eqref{eq:ldpub} implies the $``\le"$ part of \eqref{eq:laplace},
    while the lower bound \eqref{eq:ldplb} implies the $``\ge"$ part of \eqref{eq:laplace}.
\end{theorem}
Varadhan's Lemma states that the large deviation principle implies the Laplace principle.
Actually, the large deviation principle and the Laplace principle are equivalent if the rate function $\I$ has compact sublevel sets, see \cite[Theorem 1.8]{budhiraja2019analysis}.

We now specify a collection of stochastic processes 
for which the variational characterization \eqref{eq:var-single} 
of the single process partition function 
can be stated as a direct application of Varadhan's Lemma.
% Based on Theorem \ref{thm:laplace}, 
We set $\Z = \M_1(\cX \times \cY)$ 
and the stochastic process $(Z_t)_{t \in \T}$ of probability 
measures on $(\cX \times \cY, \xmap \times \ymap)$ as
$$Z_t = M_t(X^\theta, y) := \frac{1}{t}\int_0^t \delta_{(\xmap^s X^\theta, \ymap^s y)}\,ds
\quad \text{ or } \quad \frac{1}{t}\sum_{s=0}^{t-1} \delta_{(\xmap^s X^\theta, \ymap^s y)}$$
for a fixed observation $y \in \cY$ and the model process
$(X^\theta_t)_{t \in \T} \sim \mu_\theta$  
corresponding to a fixed parameter $\theta \in \Theta$.
Next, we pick a specific function $F$ that is defined by 
\begin{equation} \label{eq:ldploss}
    F: \M_1(\cX \times \cY) \to \real, \quad \eta 
    \mapsto \int_{\cX \times \cY} L_\theta(x,y)\, d\eta(x,y),
\end{equation}
where $\theta$ is the same fixed parameter. By plugging in the $Z_t$ and $F$ into the left-hand side of equation \eqref{eq:laplace}, we 
recover the partition function
\[
    \expect{\exp(-t F(Z_t))} = 
    \int_{\cX} \exp\left(-L_{\theta}^t(x, y)\right) d\mu_{\theta}(x) 
    = \partition_t^\theta(y).
\]
Hence, if $(M_t(X^\theta, y))_{t \in \T}$ 
satisfies the large deviation principle on $\M_1(\cX \times \cY)$
with rate function $\I_\theta: \M_1(\cX \times \cY) \to [0,\infty]$
given by
\begin{equation} \label{eq:rate}
    \I_\theta(\lambda) = 
    \begin{cases}
        h(\lambda \,\|\, \mu_\theta \otimes \nu), & \text{ if } \lambda \in \J(\xmap : \nu), \\ 
        \infty, & \text{ otherwise,}
	\end{cases}
\end{equation}
then by Varadhan's Lemma, we obtain the variational characterization
for a fixed $\theta$:
\[
    \lim_{t \to \infty} \frac{1}{t} \log \partition_t^\theta(y) 
    = -\inf_{\lambda \in \J(\xmap : \nu)}
    \set{\int L_\theta \,d\lambda + h(\lambda\,\|\,\mu_\theta \otimes \nu)} = - V(\theta).
\]
In the remainder of this subsection, we define rigorously the divergence term $h(\cdot\,\|\,\cdot)$ with examples for some stochastic processes and dynamical systems.

% Next, we want to identify $V(\theta)$ as in \eqref{eq:v}  
% with the right-hand side of \eqref{eq:laplace}.
% Now it is clear that $V(\theta)$ coincide with the right-hand side of \eqref{eq:laplace},
% if $(M_t(X^\theta, y))_{t \in \T}$ satisfies the large deviation principle
% on $\M_1(\X \times \Y)$
% with rate function $\I(\eta) = h(\eta \, \|\, \mu_\theta \otimes \nu)$, 
% a divergence term of $\eta$ from $\mu_\theta \otimes \nu$ to be defined later.

% \begin{remark}
%     The careful reader may notice that the domain of the variational expression
%     of $V$ in \eqref{eq:v} is $\J(\phi : \nu)$ instead of the
%     larger space $\M_1(\X \times \Y)$. Later, we will see that
%     $h(\cdot \, \|\, \mu_\theta \otimes \nu) = \infty$ on
%     $\M_1(\X \times \Y) \setminus \J(\phi : \nu)$, so
%     the results are consistent.
% \end{remark}

\begin{remark}
    The large deviation principle is stronger than \eqref{eq:var-single}, 
    since the latter only requires \eqref{eq:laplace} to hold for the one selected function \eqref{eq:ldploss}
    rather than all functions in $C_b(\Z)$.
    We will exploit this fact to prove 
    part of the large deviation principle that
    is sufficient to ensure \eqref{eq:var-single}
    instead of the full large deviation principle.
\end{remark}

We now outline the role of the divergence term as the rate function. We need to
make sense of the divergence term $h(\cdot \,\|\, \mu \otimes \nu)$, where we set $\mu = \mu_\theta$ for a fixed parameter $\theta$.
We start with the definition of relative entropy,
which quantifies the divergence between measures.
\begin{definition}
    Let $\Z$ be a Polish space. 
    For $\lambda, \eta \in \M(\Z)$, the relative entropy,
    also called the Kullback-Leibler divergence, 
	of $\lambda$ with respect to $\eta$ is defined as
	$$\relent{\lambda}{\eta} = 
	    \begin{cases}
	        \int_\Z \log \frac{d\lambda}{d\eta}\,d\lambda, & \text{ if } \lambda \ll \eta, \\ 
	        \infty, & \text{ otherwise,}
	    \end{cases}$$ 
	where $\lambda \ll \eta$ means that $\lambda$ is absolutely continuous with
	respect to $\eta$, and $\frac{d\lambda}{d\eta}$ is the corresponding 
	Radon-Nikodym derivative.
\end{definition}

The relative entropy plays a universal role in large deviation theory.
See \cite{dupuis2011weak, budhiraja2019analysis} for derivations
of many well-known LDPs using the relative entropy.
For discrete time case, Sanov's Theorem \cite{ellis2007entropy} establishes
the LDP for empirical measures 
$(\frac{1}{t}\sum_{s=0}^{t-1} \delta_{X_s})_{t \in \nat}$ 
of independent and identically distributed random
variables $X_0, X_1, \ldots, X_n \sim \mu_0$ on $\S$ with rate function
function $\I(\eta) = K(\eta\,\|\,\mu_0)$.
% $$\I: \M_1(\S) \to [0,\infty], \quad \I(\eta) = K(\eta\,\|\,\mu_0).$$ 
% In the case of discrete-time Markov chains \cite{donsker1975asymptotic, donsker1976asymptotic, donsker1983asymptotics}, 
% because of an extra order of dependency,
% it is more natural to prove the LDP for the two-step empirical
% measures $(\frac{1}{t}\sum_{s=0}^{t-1} \delta_{(X_s, X_{s+1})})_{t \in \nat}$.
% Then the rate function will be 
% $$\I: \M_1(\S^2) \to [0,\infty], \quad 
% \I(\eta) = \begin{cases}
%     K(\eta \, \| \, \mu_0 \otimes p),  & \text{if $\eta$ has same marginals,} \\ 
%     \infty, & \text{otherwise,}
% \end{cases}$$ 
% where $\mu_0$ is the initial stationary distribution, $p$
% the transition kernel and $\mu_0 \otimes p$ is the probability measure on $\S \times \S$
% defined by $$(\mu_0 \otimes p)(A \times B) = \int_A p(x, B)\,d\mu_0(x).$$ 
In the case of empirical measures of discrete-time Markov chains \cite{donsker1975asymptotic, donsker1976asymptotic, donsker1983asymptotics}, the rate function will be based on a
 relative entropy term involving both the Markov transition kernel as well as the initial
stationary distribution. 
% The LDP for one-step empirical measures
% can be shown by a classical contraction argument.
% As the order of dependencies increases, one can repeat this strategy
% by proving the LDP of empirical measures of more steps.
As the processes we are interested in can have arbitrary dependencies, and the dynamics can be deterministic, we will study the large deviation 
behavior of the whole empirical processes 
$$M_t(X) = \paran{\frac{1}{t}\sum_{s=0}^{t-1}\delta_{\xmap^s X}}_{t \in \nat} = \paran{\frac{1}{t}\sum_{s=0}^{t-1}\delta_{(X_s, X_{s + 1}, X_{s+2}, \ldots)}}_{t \in \nat},
$$
so we need to consider the role of the relative entropy on the whole processes as 
 dynamical systems.

As discussed in \cite{mcgoff2019gibbs}, the relative entropy cannot be directly applied to measures on dynamical systems, 
as any two different ergodic measures are mutually singular, the relative entropy will be infinite.  We are in need of a more useful divergence
for measures on the path space and on dynamical systems. Thus, we introduce the relative entropy rate on $\cX$, which 
usually appears as the rate function of the process-level LDPs,
such as the LDP of the empirical process $M_t(X)$ on $\cX$
for well-behaved stochastic processes $(X_t)_{t \in \T}$. The relative entropy rate quantifies the divergence
between measures on the path space in a meaningful way.
\begin{definition}[The relative entropy rate on $\cX$]
	For $\lambda, \eta \in \M(\cX)$, we define the relative entropy rate 
	of $\lambda$ with respect to $\eta$ as
	$$\dynrelent{\lambda}{\eta} := \lim_{t \to \infty} \frac{1}{t} 
		\relent{\lambda|_{\F_{t}}}{\eta|_{\F_{t}}}$$ 
	if the limit exists, where $\lambda|_{\F_{t}}$ and $\eta|_{\F_{t}}$ are
    $\lambda$ and $\eta$ restricted on $\F_t$.
\end{definition}
For many well-behaved stochastic processes $(X_t)_{t \in \T}$ with some 
law $\mu$ on $\cX$, the relative entropy rate $h(\cdot\, \|\, \mu)$
is a well-defined, nontrivial function on $\M_1(\cX)$, and the empirical process
$M_t(X)$ satisfy an LDP with rate function 
\[
    \I(\eta) = 
    \begin{cases}
        h(\eta \,\|\, \mu), & \text{ if } \eta \in \M_1(\cX, \xmap), \\ 
        \infty, & \text{ otherwise.}
	\end{cases}
\]
See, for instance, \cite{ellis2007entropy, seppalainen1994large, 
chiyonobu1988large, rassoul2015course}.

We introduce the 
relative entropy rate on $\cX \times \cY$, and we adapt
 \cite[equation (3.7)]{seppalainen1994large} which states the process-level conditional LDP of Markov chains on a random
environment which is analgous to our observation $y$.
\begin{definition}[The relative entropy rate on $\cX \times \cY$] \label{def:entropyXY}
	For $\lambda, \eta \in \M(\cX \times \cY)$, we define the relative entropy rate 
	of $\lambda$ with respect to $\eta$ as
	$$\dynrelent{\lambda}{\eta} := \lim_{t \to \infty} \frac{1}{t} 
		\relent{\lambda|_{\tF_{t}}}{\eta|_{\tF_{t}}}$$ 
	if the limit exists. 
\end{definition}
Now we are in a good position to restate the first step of our strategy:
to show that for $X^\theta \sim \mu_\theta$
with any fixed $\theta$, it holds for $\nu$-almost every $y \in \cY$, 
$(M_t(X^\theta, y))_{t \in \T}$ satisfies the large deviation principle
on $\M_1(\cX \times \cY)$ with well-defined rate function 
\eqref{eq:rate},
or a weaker result that is sufficient to guarantee \eqref{eq:var-single}.
The previous statement can be interpreted in the following way: we have hope of
proving posterior consistency as long as our model family $\set{\mu_\theta}_{\theta \in \Theta}$ 
corresponds to processes with good large deviation behavior.

\subsubsection{Examples of the relative entropy rate}
For concreteness, we include a few examples, where the relative entropy rate on $\cX$
is a closed form expression. 
The second and third examples below come from \cite[Appendix A]{dupuis2016path},
where computations are shown in full details.
The fourth example comes from \cite{chazottes1998relative}.
The fifth example comes from \cite{jiang2004mathematical}.

\begin{example}[I.I.D. processes]
Consider two discrete-time i.i.d. processes 
$(X_t)_{t \in \nat}$, $(Z_t)_{t \in \nat}$ on $\S$,
where for each $t \in \nat$, $X_t \sim \mu_0$ 
and $Z_t \sim \eta_0$ in $\S$, respectively, and 
the whole process $(X_t)_{t \in \nat}$, $(Z_t)_{t \in \nat}$ have 
law $\mu = \mu_0^{\otimes \nat}$, 
$\eta = \eta_0^{\otimes \nat}$ on the path space $\cX = \S^\nat$, respectively.
Then the relative entropy rate of $\eta$ with respect to $\mu$ is
$h(\eta \, \| \, \mu) = K(\eta_0 \, \| \, \mu_0).$
\end{example}

\begin{example}[Discrete-time Markov chains]
Let $(X_t)_{t \in \nat}$, $(Z_t)_{t \in \nat}$ be two Markov chains on $\S$
with transition kernels $p$ and $q$, and 
 stationary initial measures $\mu_0$ and $\eta_0$ on $\S$,
respectively. Let $\mu$, $\eta$ be the law 
of $(X_t)_{t \in \nat}$, $(Z_t)_{t \in \nat}$
on the path space $\cX = \S^\nat$, respectively.
Then the relative entropy rate of $\eta$ with respect to $\mu$ is
$h(\eta \, \| \, \mu) = K(\eta_0 \otimes q \, \| \, \mu_0 \otimes p)$,
where $\mu_0 \otimes p$ is the probability measure on $\S \times \S$
defined by $(\mu_0 \otimes p)(A \times B) = \int_A p(x, B)\,d\mu_0(x).$
\end{example}

\begin{example}[Stochastic differential equations]
Consider two It\^o diffusion processes on $\real^d$:
\begin{align*}
    dX_t & = a(X_t)dt + \sigma(X_t)dW_t, \\
    dZ_t & = b(Z_t)dt + \sigma(Z_t)dW_t,
\end{align*}    
where $a$, $b$ are vector fields on $\real^d$ and $\sigma(x) \in \real^{d \times d}$ 
is non-singular so that each of the stochastic differential equations above
have unique global weak solutions, given
stationary initial conditions $X_0 \sim \mu_0$ and $Z_0 \sim \eta_0$.
Let $\mu$, $\eta$ be the law 
of $(X_t)_{t \in [0, \infty)}$, $(Z_t)_{t \in [0,\infty)}$
on the path space $\cX = C([0,\infty), \real^d)$, respectively.
Under some assumptions for Novikov's condition, one can 
show the relative entropy rate of $\eta$ with respect to $\mu$ is
$h(\eta \, \| \, \mu) = \frac{1}{2} \mathbb{E}_{x \sim \eta_0}{\|a - b\|^2_{\Sigma^{-1}}}$,
where $\Sigma := \sigma \sigma^T$ is the diffusion matrix and
$\|b(x)\|_{\Sigma^{-1}(x)} := \sum_{i,j = 1}^d b_i(x) \Sigma_{ij}^{-1}(x) b_j(x)$.
\end{example}
\begin{remark}
The preceding example can be generalized to 
the case where $\sigma$ is singular 
but $b - a \in range(\sigma)$ by the following
similar steps in the proof of \cite[Theorem 4.1 and 4.2]{da2022entropy}.
In this case, $\Sigma^{-1}$ denotes the Moore-Penrose matrix pesudo-inverse of $\Sigma$.
On the other hand, when $X$ and $Z$ have different diffusion
coefficient $\sigma$, one can still compute the relative
entropy rate at least in the case of martingale diffusions
in dimension 1 (see \cite{backhoff2022specific}). 
\end{remark}

\begin{example}[Shifts of finite type (SFT)] \label{eg:SFT}
    Let $\T = \nat$ and $\S$ be a finite alphabet. 
    Let $\sigma_0 \in \M_1(\S)$ be the uniform probability measure, 
    and $\sigma = \sigma_0^{\otimes \T} \in \M_1(\cX)$ be the 
    product measure on $\cX$.
    Let $\eta \in \M_1(\cX, \xmap)$ be an arbitrary $\Phi$-invariant measure.
    Then the relative entropy rate of $\eta$ with respect to $\sigma$ is
    $h(\eta \, \| \, \sigma) = h(\sigma) - h(\eta)  = \log |\S| - h(\eta)$, 
    where $h(\eta)$ denotes the measure-theoretic
    or Kolmogorov-Sinai entropy of $\eta$.
\end{example}

\begin{example}[Gibbs processes on SFTs]
    Let $\T = \nat$ and $\S$ be a finite alphabet. 
    Given a H\"older continuous potential function $\phi: \cX \to \real$, 
    let $\mu \in \M_1(\cX, \xmap)$ be the corresponding Gibbs measure.  
    Given an invariant measure $\eta \in \M_1(\cX, \xmap)$,
    the relative entropy rate of $\eta$ with respect to $\mu$ is
    $h(\eta \, \| \, \mu) = \cPP - \int \phi \,d\eta - h(\eta)$, 
    where $h(\eta)$ denotes the measure-theoretic
    or Kolmogorov-Sinai entropy of $\eta$, and 
    $\cPP$ denotes the pressure of $\phi$.
    See the precise Definition \ref{def:Gibbs} later.
\end{example}

\begin{remark}[The relative entropy rate in ergodic theory]
From the ergodic theory point of view, 
one may view the $\sigma$-algebras $\F_t$ as measurable
partitions consisting of cylinder sets of the path space $\cX$.
From this perspective, the relative entropy rate on
many dynamical systems can be related to
other well-studied quantities in ergodic theory, such as
the Kolmogorov-Sinai entropy (as in Example \ref{eg:SFT}), 
Lyapunov exponents, and equilibrium states.
See Section \ref{sec:ent-ergodic} for a review of related work.
\end{remark}

\subsection{The exponentially continuous model family and the loss function} \label{sec:exp-cont-loss}
So far, we have a strategy for proving the 
variational characterization \eqref{eq:var-single} for a single parameter
$\theta$ using large deviation techniques.
It remains to push \eqref{eq:var-single} 
to the variatonal characterization \eqref{eq:variational} on the support
of any probability measure $\pi$ on the parameter space $\Theta$.
In this case, we need to study the large deviation behavior 
of the mixture $\int \mu_\theta \,d\pi(\theta)$ 
of probability measures $\set{\mu_\theta}_{\theta \in \Theta}$
mixing by $\pi$.
Intuitively, to recover \eqref{eq:variational} 
based on the previous reasoning using Varadhan's Lemma, 
we want the process-level LDP to hold
for the empirical process of $\int \mu_\theta \,d\pi(\theta)$
with rate function $\inf_{\theta \in \supp(\pi)} \I_\theta$, where $\I_\theta$ is
the rate function for the process-level LDP of $\mu_\theta$.
This observation is reminiscent of the work
on large deviations of mixtures \cite{wu2004large, dinwoodie1992large, biggins2004large}.
The key property required for large deviations of mixtures is exponential continuity
on the parametrization map $\theta \mapsto \mu_\theta$. 
The intuition for exponential continuity is that if $\theta_t \convergeto \theta$, then $\set{\mu_{\theta_t}}_t$ is 
an exponential approximation
of $\mu_\theta$ in the language of large deviation theory 
\cite[Section 4.2.2]{dembo2010large}
so that $\set{\mu_{\theta_t}}$ has the same large deviation asymptotics 
as $\mu_\theta$ as $t \to \infty$.
In our case, we consider exponential approximations of the empirical processes.
We adapt the definition of exponential continuity to our setting.
\begin{definition}[Exponentially
    continuous family]\label{def:exp-cont}
    We say $\set{\mu_\theta}_{\theta \in \Theta}$ is an \textbf{exponentially
    continuous} family with respect to $L$
    if the following holds: 
    for all $\theta \in \Theta$, 
    it holds for $\nu$-a.e. $y \in\cY$ that
    the following limit exists
    \begin{equation} \label{eq:var-single-2}
        \lim_{t \to \infty} \frac{1}{t}\log \partition_t^{\theta}(y)
        := \lim_{t \to \infty} \frac{1}{t}\log 
        \int_\cX \exp\paran{-L_{\theta}^t(x, y)}
        d\mu_{\theta}(x)
    \end{equation}
    which we define as $- V(\theta)$,
    and if $(\theta_t)_{t \in \T}$ is a family of parameters
    such that $\theta_t \convergeto \theta$ in $\Theta$, then
    \begin{equation} \label{eq:exp-cont}
        \lim_{t \to \infty} \frac{1}{t}\log \partition_t^{\theta_t}(y)
        := \lim_{t \to \infty} \frac{1}{t}\log 
        \int_\cX \exp\paran{-L_{\theta_t}^t(x, y)}
        d\mu_{\theta_t}(x) = -V(\theta). 
    \end{equation}
    Note that in \eqref{eq:exp-cont}, $t$ and $\theta_t$ are changing
    at the same time.
\end{definition}

For clarity of exposition, we focus on the loss function $L$ with the following assumption. Weaker assumptions are possible but will complicate theorem statements and are not as natural
when the state space $\S$ is non-compact. The trade-off for weaker
assumptions on the loss function is usually more detailed structural assumptions on the observed system and the model family, which
can be treated case by case by fine-tuning our flexible framework.
% \begin{definition}
%     We say a bounded continuous function $\bar{d}_\Y: \Y \times \Y \to [0,\infty)$ is 
%     metric-like if $\bar{d}_\Y$ satisfies the 
%     following reverse triangle inequality
%     $$|\bar{d}_\Y(y, \bar{y}) - \bar{d}_\Y(y', \bar{y})| \le C \bar{d}_{\Y}(y, y'),$$
%     for $y, y', \bar{y} \in \Y$ and some universal constant $C > 0$.
% \end{definition}

% \begin{assumption}[The loss function] \label{assumption:loss}
%     We assume that the loss function $L: \Theta \times \X \times \Y \to \real$ 
%     has the form $L_\theta(x, y) = \bar{d}_\Y(\varphi_\theta(x), y)$, for  some
%     metric-like function $\bar{d}_\Y \in C_b(\Y^2)$ and some 
%     continuous function $\varphi: \Theta \times \X \to \Y, (\theta, x) 
%     \mapsto \varphi_\theta(x)$ such that the map
%     $\varphi_\theta: \X \to \Y, x \mapsto \varphi_\theta(x)$ 
%     belongs to $C_{b,loc}(\X \times \Y)$,
%     and
%     $$\bar{d}_\Y(\varphi_\theta(x), \varphi_{\theta'}(x')) 
%     \le \omega(d_\Theta(\theta, \theta')) + \omega(d_\X(x, x')),$$
%     for any $\theta, \theta' \in \Theta$ and any $x, x' \in \X$, 
%     %  $$\bar{d}_\Y(\varphi_\theta(x), \varphi_{\theta'}(x)) 
%     % \le \omega(d_\Theta(\theta, \theta')),$$
%     % for any $\theta, \theta' \in \Theta$ and any $x \in \X$, 
%     and for some function
%     $\omega: [0, \infty) \to [0,\infty)$
%     with $\lim_{r \to 0^+}\omega(r) = 0$, which can be viewed as
%     the modulus of continuity of $\varphi$.
%     In this case, 
%     we will denote the constant $C$ for $\bar{d}_\Y$ by $ $.
% \end{assumption}

\begin{assumption}[The loss function] \label{assumption:loss}
    We assume that the loss function $L: \Theta \times \cX \times \cY \to \real$,
    $(\theta, x, y) \mapsto L_\theta(x, y)$
    satisfies the following properties:
    \begin{enumerate}
        \item For all $\theta \in \Theta$, the map $L_\theta: (x,y) \mapsto L_\theta(x,y)$
        belongs to $C_{b,loc}(\cX \times \cY)$.
        \item For any $\theta, \theta' \in \Theta$, $x, x' \in \cX$ and $y \in \cY$,
        \[
            |L_\theta(x, y) - L_{\theta'}(x', y)| \le 
                  \omega(d_\Theta(\theta, \theta')) + \omega(d_\cX(x, x'))
        \]
        for some function
        $\omega: [0, \infty) \to [0,\infty)$
        with $\lim_{r \to 0^+}\omega(r) = \omega(0) = 0$, 
        which can be viewed as the modulus of continuity.
    \end{enumerate}
\end{assumption}
The natural class of loss functions we are considering is the following.
\begin{example}
    Let $\bar{d}_\cY \in C_b(\cY^2)$ be a bounded metric and $\varphi: \Theta \times \cX \to \cY$,
    $(\theta, x) \mapsto \varphi_\theta(x)$
    be a uniformly continuous function such that for every $\theta \in \Theta$,
    the map $\varphi_\theta: \cX \to \cY$, $x \mapsto \varphi_\theta(x)$
    is $\F_{r(\theta)}$-measurable for some $r(\theta) > 0$.
    Then the loss function $L$ defined by $L_\theta(x ,y) := \bar{d}_\cY(\varphi_\theta(x), y)$
    satisfies Assumption \ref{assumption:loss}.
\end{example}
\begin{remark}
    Note that every metric is equivalent to a bounded metric. 
    The assumption that $L_\theta$
    belongs to $C_{b,loc}(\cX \times \cY)$ is practical, since
    if the loss $L_\theta(x, y)$ depends on the whole path of $x$,
    then it is more difficult for real-world computations.
\end{remark}

As an intermediate step of proving exponential continuity, 
it is convenient to show that given a fixed $\theta \in \Theta$, 
for $\nu$-a.e. $y \in \cY$, 
if $(\theta_t)_{t \in \T}$ is a sequence in $\Theta$ such that 
    $\theta_t \convergeto \theta$, 
    then there exists a probability space $(\Omega, \F, \mathbb{P})$ 
    such that for all $\epsilon > 0$, 
\begin{equation} \label{eq:expapprox}
    \limsup_{t \to \infty} \frac{1}{t} 
    \log\prob{\abs{\frac{1}{t}(L^t_\theta(X^{\theta_t}, y) 
    - L^t_\theta(X^\theta, y))} > \epsilon} = -\infty.
\end{equation}
The expression \eqref{eq:expapprox} is basically the same as the
definition of exponential approximations 
\cite[Definition 4.2.10]{dembo2010large}.
To show \eqref{eq:expapprox}, we will see that it requires the 
continuity of the map $\theta \mapsto \mu_\theta$, Assumption \ref{assumption:loss} on the loss function, and 
additional non-trivial restrictions 
depending on the model family, which we will see
in our example processes.

Once the exponential continuity of $\set{\mu_\theta}_{\theta \in \Theta}$
is established, the continuity of $V$ and the variational characterization \eqref{eq:variational}
needed for posterior consistency follow immediately. 
In particular, posterior consistency holds by Theorem \ref{thm:posterior}.
This step
does not require Assumption \ref{assumption:loss} on the loss
function.
\begin{lemma}\label{lem:exp-cont-v}
    If $\set{\mu_\theta}_{\theta \in \Theta}$ is an exponentially continuous
    family with respect to the loss function $L$, then 
    $V$ is a continuous function.
\end{lemma}

\begin{proposition} \label{prop:expcont}
    Suppose $\set{\mu_\theta}_{\theta \in \Theta}$ is an exponentially continuous
    family with respect to the loss function $L$ 
    and $\pi$ is a Borel probability measure on $\Theta$.
    Then for $\nu$-almost every $y \in \cY$,
    \begin{equation*}
        \lim_{t \to \infty} -\frac{1}{t}\log \partition_t^{\pi}(y) 
        = \inf_{\theta \in \supp(\pi)} V(\theta).
    \end{equation*}
\end{proposition}
We defer the proofs to Section \ref{append:exp-cont} as they adapt known arguments.

\begin{remark}[Continuity of $V$]
    Although Theorem \ref{thm:posterior} only requires lower
    semicontinuity of $V$, the continuity of $V$ is actually
    of significance. 
    
    In \cite{mcgoff2019gibbs}, the authors only required lower semicontinuity on $V$ as the focus was on the dependency of the divergence term 
    $ h(\lambda \,\|\, \mu_\theta \otimes \nu)$ on $\theta$, which
    is generally only lower semicontinuous. However, if
    one considers both terms as a whole in $V$, 
    $\inf_\lambda \int L_\theta\,d\lambda + h(\lambda \,\|\, \mu_\theta \otimes \nu)$
    this expression is usually continuous in $\theta$.
    This is an advantage of the Laplace principle mentioned in Theorem \ref{thm:posterior}
    compared to the usual LDP, which is exploited for
    proving uniform LDPs in \cite[Proposition 1.12]{budhiraja2019analysis}.
    
    The continuity of $V$ resolves a subtle 
    discrepancy between the analoguous posterior consistency results of
    \cite{shalizi2009dynamics} and \cite{mcgoff2019gibbs}.
    More specifically,
    the posterior distributions in \cite{shalizi2009dynamics}
    converge to the essential minimizers (those that achieve the 
    $\pi_0$-essential infimum) of $V$ instead of
    the actual minimizers of $V$, while there is no difference 
    if $V$ is continuous. 
\end{remark}

\subsection{A summary of the strategy} \label{sec:strategy}
We summarize our strategy of proving posterior consistency 
for a given model family $\set{\mu_\theta}_{\theta \in \Theta}$
and a loss function $L$. The guiding principle is that if 
each member $\mu_\theta$ of the model family has good
large deviation behavior such as satisfying the process-level LDP, 
then we have hope to prove posterior consistency.
The strategy consists of the following two steps:
\begin{enumerate}
    \item Extend the LDP of $M_t(X^\theta)$ to the (conditional) 
    LDP of $M_t(X^\theta, y)$ for $\nu$-a.e. $y \in \cY$ so that
    the following limit exists:
    \begin{equation*} 
        \lim_{t \to \infty} \frac{1}{t}\log 
        \int_\cX \exp\paran{-L_{\theta}^t(x, y)}
        d\mu_{\theta}(x) =: -V(\theta). 
    \end{equation*}
    As a byproduct, we obtain an explicit form of $V$:
    \[
        V(\theta) = \inf_{\lambda \in \J(\xmap : \nu)} \set{\int L_\theta \,d\lambda + h(\lambda\,\|\,\mu_\theta \otimes \nu)}.
    \]
    
    \item Prove exponential continuity (Definition \ref{def:exp-cont}) 
    of the map
    $\theta \mapsto \mu_\theta$ with respect to the loss function $L$.
    Then posterior consistency follows from Lemma \ref{lem:exp-cont-v},
    Proposition \ref{prop:expcont} and Theorem \ref{thm:posterior}.
\end{enumerate}
These steps can be taken in a generic way, meaning we can use any
variant of large deviation behaviors that exist to guarantee 
the needed variational characterization \eqref{eq:variational}.
In particular, we will apply a standard approach of directly proving 
the large deviation behavior and exponential contintuity in
the case of hypermixing processes, while we will reduce 
the case of Gibbs processes on shifts of finite type to
known results for i.i.d. processes by exploiting the 
large deviation structure. These two different approaches
highlights the flexibility of our framework.

% We believe that our arguments and assumptions are far from being the 
% sharpest to prove posterior consistency, since we only 
% employ classical techniques in large deviation theory.

\subsection{Related work} 
We now summarize more previous work related to our framework.

\subsubsection{Large deviations in Bayesian inference}
% Large deviation theory is widely studied in statistical inference, 
% but we mainly survey those results from the Bayesian perspective.
In \cite{ganesh1999inverse}, the authors prove the large deviation principle
for posterior distributions of i.i.d. random variables with finite values  
as an inverse of Sanov's Theorem, 
Large deviation results are obtained in \cite{papangelou1996large, eichelsbacher2002bayesian} for posterior distributions of Markov chains
with a finite state space. In \cite{ganesh2000large}, a large deviation result is proved for Dirichlet posteriors.
In \cite{grendar2007bayesian, grendar2009asymptotic}, the authors give a probabilistic justification of empirical likelihoods through large deviations.

% The setting of this work is very similar to \cite{shalizi2009dynamics}, where the author establishes a similar posterior consistency 
% along with a large deviation principle for some dependent processes with likelihood ratios instead of loss functions. There are 
% technical difficulties in verifying the main assumptions in \cite{shalizi2009dynamics}, in particular when the state space is not finite, 
% and no procedures are provided in \cite{shalizi2009dynamics} to check the assumptions. Our work can be seen as a general framework for verifying those assumptions. Additionally, we point out that the method of \cite{shalizi2009dynamics} and this work is very similar to the entropy-based argument in \cite{mcgoff2015consistency}, which originates from \cite{douc2011consistency}.

% More recently, the authors of \cite{lopes2020bayes} reformulate the
% posterior consistency of Gibbs processes on mixing subshifts of finite type, 
% as established in \cite{mcgoff2019gibbs}, 
% using large deviation results.

The authors of \cite{puhalskii1998large} develop a large-deviation analogue of asymptotic decision theory,
which establishes an asympotic lower bound for the sequence of minimax risks, 
in analogy with Le Cam's
Minimax Theorem.

\subsubsection{Large deviation theory}
A huge and growing body of literature is devoted to the study of 
large deviations, and we refer to textbooks 
\cite{ellis2007entropy, dembo2010large, rassoul2015course, budhiraja2019analysis}
for general introductions. 
In particular, the weak convergence approach based on relative entropy
in \cite{dupuis2011weak, budhiraja2019analysis}
has a strong influence to our understanding of the large deviation phenomenon
in this work. 
We survey some of the literature that
are most relevant.

Large deviation theory on i.i.d. processes 
can be found in the book \cite{ellis2007entropy}.
Donsker and Varadhan established large deviations of 
certain Markov processes in a series of 
work \cite{donsker1975asymptotic, donsker1976asymptotic, donsker1983asymptotics},
and the extension to hidden Markov models is established in \cite{hu2011large}.
Large deviations on some mixing stochastic processes are studied in 
\cite{bryc1996large, chiyonobu1988large}, and that on Gaussian processes
are studied in \cite{cmp/1103941986}.
In particular, we will apply our framework to processes satisfying the hypermixing condition in \cite{chiyonobu1988large}.
Besides stochastic processes, large deviation properties are also shown in 
\cite{orey1988large, young1990large, kifer1990large, rey2008large} for
many well-behaved dynamical systems, such as Gibbs processes on mixing subshifts of finite type, Axiom A systems, and some non-uniform hyperbolic systems.
We point out that following our strategy, one may adapt the proof
for the large deviation behavior of Gibbs states in \cite{young1990large}
to obtain the posterior consistency result in \cite{mcgoff2019gibbs} for Axiom A systems.

The large deviation principle conditioning on a given observation $y$ is referred as the conditional large deviation
principle or the quenched large deviation principle.  Quenched large deviation principles are widely studied in various settings, including irreducible Markov chains with random transitions \cite{seppalainen1994large}, random walks in a random environment \cite{comets2000quenched}, finite state Gibbs random fields \cite{chi2002conditional}
and some mixing processes \cite{chi2001stochastic}. As we pass from a single parameter to a mixture over the parameter space, we move to what is called the large deviation of mixtures or
the annealed large deviation principle. This formalism has been widely studied for exchangeable sequences 
\cite{dinwoodie1992large, wu2004large}, random walks in a random environment \cite{comets2000quenched},
and other stochastic models \cite{biggins2004large}. The idea of exponential continuity we use appears in \cite{dinwoodie1992large} and is implicit in the study of uniform large deviation principles
of Markov processes with different initial conditions
in \cite{dupuis2011weak, budhiraja2019analysis}, here the ``parameter'' is
the initial condition of the process.

Large deviation techniques extend to infinite-dimensional processes \cite{budhiraja2008large} and random fields, i.e., processes indexed by lattices \cite{rassoul2015course}. This highlights the potential applicability of our framework to a very rich class of models.
There are also connections of various ideas with statistical
mechanics, as presented in \cite{ellis2007entropy, rassoul2015course, touchette2009large}.

\subsubsection{The relative entropy rate in ergodic theory} \label{sec:ent-ergodic}
There are connections between the relative entropy rate and other notions of entropy in ergodic theory, such as the Kolmogorov-Sinai entropy
\cite{walters2000introduction} of a dynamical system. 
One can either compare the limit definitions with other notions of entropy, 
using the variational characterization \cite[Lemma 2.4.(e)]{budhiraja2019analysis} 
of the relative entropy over partitions of the space,
or identify the relative entropy rate with the unique 
large deviation rate function of a well-behaved dynamical system.
In the case of mixing shifts of finite type---including finite-state Markov chains or 
the Bernoulli shift---one obtains explicit characterizations of the relative entropy rate of 
Gibbs measures in terms of the pressure and the corresponding potential functions,
see \cite{young1990large, kifer1990large, olivier2000relative, chazottes1998relative} and the exposition in \cite[Chapter 8.2]{rassoul2015course}.
For continuous-space Markov chains as in \cite{donsker1975asymptotic, donsker1976asymptotic, donsker1983asymptotics}, the relative entropy rate has an explicit
expression in terms of the transition kernel and the stationary initial distribution, 
also called the Donsker-Varadhan entropy.
Moreover, the divergence term in \cite{mcgoff2016variational} is the
conditional version of the relative entropy
replacing the usual Kolmogorov-Sinai entropy with the fibre entropy of a
random dynamical system \cite{kifer2006random, kifer2001topological}. 
In the case of Axiom A systems, the relative entropy rate of
Riemannian or SRB measures are characterized by Lyapunov exponents,
see \cite{young1990large, kifer1990large}. Also,
\cite{ledrappier1977relativised} provides a more general perspective.
A large collection of examples including both stochastic processes
and dynamical systems from the point of view of
nonequilibrium statistical physics is presented in 
the exposition \cite{jiang2004mathematical}.

Finally, we point out some recent applications of the relative entropy
rate.
The first one \cite{dupuis2016path} uses it in sensitivity analsyis and uncertainty quantification,
where relative entropy rates of some Markov chains 
and some stochastic differential equations are computed 
explicitly as examples.
The second one \cite{da2022entropy} uses it for evaluating numerical simulations of 
stationary diffusions.

% \subsection{Example processes}\label{sec:ex}
% One motivation for using large deviations to prove posterior consistency is to have a flexible framework that can be applied to
% a variety of processes. This is in contrast to having to design a new proof technique to show posterior consistency for each new class of stochastic processes
% or dynamical systems. Later, we consider two very different process models: continuous time hypermixing stochastic processes, and
% Gibbs processes on shifts of finite type. Our large deviation framework can be used to prove posterior consistency for both settings.

\section{Non-dynamic example: Bayesian inverse problems}
\label{sec:ex_inverse}
\noindent In this section, we apply our framework to a trivial 
dynamical system as a sanity check. This example   
can be related to a simplified version of 
Bayesian inverse problems, where the negative
log likelihood is replaced by the loss function. It also highlights the connection
between our framework and the seminal work \cite{stuart2010inverse, dashti2017bayesian}.

Let $\T = \nat$ and let the parameter space be the same as the state space $\Theta = \S$.
For each $\theta \in \Theta$, the parameterized model 
$\mu_\theta$ is given by the Dirac delta $\delta_{(\theta, \theta, \ldots)}$,
i.e., the corresponding stationary process is given by $X^{\theta} = (\theta, \theta, \ldots)$.
Clearly, $\mu_\theta$ is an $\xmap$-ergodic measure on $\cX$,
which makes the dynamical system $(\cX, \xmap, \mu_\theta)$ trivial.
In this case, we may perform generalized Bayesian inference
with a loss function $L$
on the underlying state $\theta$ given observations $y = (y_0, y_1, \ldots)$.

\begin{example} \label{ex:signal-plus-noise}
We may consider the following signal-plus-noise model
as a simplified version of inverse problems discussed in \cite{stuart2010inverse}:
for some true parameter $\theta^* \in \Theta$,
\begin{equation*} 
	y_i = \cG(\theta^*) + \epsilon_i,
\end{equation*}
where $y_i \in \cY$ is the sequence space of some vector space with $\ymap$
being the shift map, and $(\epsilon_i)$ 
are i.i.d random noises on that vector space.
\end{example}
Using our framework, we derive the generalized posterior distribution:
for any Borel set $E \subset \Theta$,
\begin{align*}
	\pi_t(E \,| \,y) & = \frac{1}{\partition^{\pi_0}_t(y)}
		\int_E \int_\cX \exp(-L^t_\theta(x, y)) \, d\mu_\theta(x) \,d\pi_0 (\theta) \\
		& = \frac{1}{\partition^{\pi_0}_t(y)} \int_E \exp(-L^t_\theta((\theta, \theta, \ldots), y))  \,d\pi_0 (\theta)
\end{align*}
where the partition function is given by
\[
	\partition^{\pi_0}_t(y) = \int_\Theta \exp(-L^t_\theta((\theta, \theta, \ldots), y))  \,d\pi_0 (\theta).
\]
We may write $L_\theta(y) = L_\theta((\theta, \theta, \ldots), y)$.
In particular, $\pi_t(\cdot \,| \,y)$ satisfies the variational principle
of Bayesian posterior distribution
by \cite{Zellner1988}:
\[
	\pi_t(\cdot \,| \,y) = \underset{\eta \in \M_1(\Theta)}{\arg\min} 
	\int_\Theta L^t_\theta(y)\,d\eta(\theta) + K(\eta \,\|\, \pi_0),
\]
and $\pi_1(\cdot \,| \,y)$ is the loss-based version of the posterior
distribution in the Bayesian inverse problem setting 
\cite{stuart2010inverse, dashti2017bayesian}. 
For loss-based Bayes methods on inverse problems,
we refer to the recent work \cite{baek2022probabilistic, zou2019adaptive, dunlop2021stability}.

One can verify the variational characterization of posterior 
consistency (Theorem \ref{thm:posterior}) 
by following the same steps in the next section
with a straight forward argument in this 
setting, or just treat the current setting as a
special case of the next section. 
In fact, one can substantially weaken the assumptions 
on the loss function $L$ in this non-dynamical setting.
Since $\mu_\theta = \delta_{(\theta, \theta, \ldots)}$,
one can easily check for any $\lambda \in \J(\xmap : \nu)$,
the relative entropy rate
$h(\lambda\,\|\,\mu_\theta \otimes \nu) = 0$ if $\lambda = \mu_\theta \otimes \nu$,
and $h(\lambda\,\|\,\mu_\theta \otimes \nu) = \infty$, otherwise.
Hence, we obtain the corresponding function $V$ as discussed in the previous section:
\begin{align*}
	V(\theta) & = \inf_{\lambda \in \J(\xmap : \nu)} \set{\int L_\theta \,d\lambda + h(\lambda\,\|\,\mu_\theta \otimes \nu)} = \int_\cY L_\theta(y)\,d\nu(y),
\end{align*}
since $\lambda = \mu_\theta \otimes \nu$ is a minimizer.
Therefore, by Theorem \ref{thm:posterior}, as $t \to \infty$,
$\pi_t(\cdot \,| \,y)$ will concentrate on those parameters that
minimize $V$, the expected loss. This result is in consistency
with the discussion in \cite[Section 2.3.3]{baek2022probabilistic}.

In the case of Example \ref{ex:signal-plus-noise},
if the loss has the form of $L_\theta((y_0, y_1, \ldots)) = \|\cG(\theta) - y_0\|^2$ (with the right norm $\|\cdot\|$),
and the noises $(\epsilon_i)$ have mean zero and finite variance,
then the true parameter $\theta^*$ will be a minimizer of $V$.

\begin{remark}
	In \cite{mcgoff2016variational}, authors discuss extensively
	the signal-plus-noise model in a dynamical setting, like
	\[
		y_t = \cG(X_t^{\theta}) + \epsilon_t,
	\]
	where $X^\theta \sim \mu_\theta$ is a general stationary process,
	so the dynamics is non-trivial.
	They provide sufficient conditions and negative examples
	for when the signal $\theta$ is recoverable as a minimizer
	of the expected loss using least squares methods, based on
	a complexity measure of dynamical models (essentially
	the topological entropy). 
	On the other hand, our example above is simple because
	the relative entropy rate $h$ is trivial. 
	It is of interest whether by assuming that the complexity
	of the dynamical models $\{\mu_\theta\}_{\theta \in \Theta}$
	is zero as in \cite{mcgoff2016variational},
	the relative entropy rate $h(\lambda\,\|\,\mu_\theta \otimes \nu)$ will simplify.
	If true, we may try to develop the main results 
	of \cite{mcgoff2016variational} in the Bayesian setting.
\end{remark}

\section{Example: hypermixing processes}
\label{sec:ex_hypermixing}
Most posterior consistency results make at least one of the 
following assumptions on the model processes: 
(1) it is a discrete time process; 
(2) the state space $\S$ is compact or discrete;
(3) the dependency structure is simple, such as Markov chains; 
(4) the dynamics have a particular structure, such as Gaussian processes.
We are motivated to prove posterior conistency 
for a general class of continuous-time, dependent processes on
general state spaces with good large deviation behavior.

Hypermixing processes refer to a class of stochastic processes with certain strong mixing properties.
In this section, we apply our framework to prove posterior consistency when the model family is a collection of hypermixing processes. Along the way, we also obtain new quenched and
annealed large deviation asymptotics for hypermixing processes.
We will focus on the continuous-time horizon $\T = [0, \infty)$, though the results can be stated and proved analogously for discrete-time processes.

We start with relevant definitions.
Given a closed interval $I \subset \T$, we denote the $\sigma$-algebra
$\F_I = \sigma(X_t :  t \in I).$
The following definitions come from \cite[Section 5.4]{deuschel2001large} and \cite{chiyonobu1988large}.
\begin{definition}[\cite{chiyonobu1988large}]
	Given $\ell > 0$, $n \ge 2$, and real-valued functions $f_1, \ldots, f_n$ on $\cX$,
	we say that $f_1, \ldots, f_n$ are $\ell$-\textbf{measurably separated} if 
	there exist intervals $I_1, \ldots, I_n$ such that
	$\dist{I_m}{I_{m'}} \ge \ell$ for $1 \le m < m' \le n$ and
	$f_m$ is $\F_{I_m}$-measurable for each $1 \le m \le n$,
	where $\dist{I}{I'}$ denotes the distance between $I$ and $I'$.
\end{definition}

\begin{definition}[\cite{chiyonobu1988large}]
	We say that $\mu \in \M_1(\cX, \xmap)$ is \textbf{hypermixing} if there exist a number $\ell_0 \ge 0$
	and non-increasing $\alpha, \beta: [\ell_0, \infty) \to [1, \infty)$
	and $\gamma: [\ell_0, \infty) \to [0,1]$ which satisfy
    % and non-increasing $\alpha: (\ell_0, \infty) \to [1, \infty)$
    % which satisfies
	\begin{equation}
		\lim_{\ell \to \infty} \alpha(\ell) = 1, \quad
		\limsup_{\ell \to \infty} \ell(\beta(\ell) - 1) < \infty, \quad
		\lim_{\ell \to \infty} \gamma(\ell) = 0
	\end{equation}
	and for which
	\begin{equation}\label{cond:H-1} \tag{H-1}
		\norm{f_1 \cdots f_n}_{L^1(\mu)} \le 
			\prod_{k=1}^n \norm{f_k}_{L^{\alpha(\ell)}(\mu)}
	\end{equation}
	whenever $n \ge 2$, $\ell > \ell_0$, and $f_1, \ldots, f_n$ are 
	$\ell$-measurably separated functions, 
	and
	\begin{equation}\label{cond:H-2} \tag{H-2}
		\abs{\int_\cX fg\,d\mu - \paran{\int_\cX f\,d\mu} \paran{\int_\cX g\,d\mu } }
		\le \gamma(\ell) \norm{f}_{L^{\beta(\ell)}(\mu)} \norm{g}_{L^{\beta(\ell)}(\mu)}
	\end{equation}
	whenever $\ell > \ell_0$ and $f, g \in L^1(\mu)$ 
	are $\ell$-measurably separated.
\end{definition}

In this section, we denote the law of a hypermixing process by $\mu$. 
The main large deviation result in
\cite{chiyonobu1988large} states the following theorem. 
\begin{theorem}[\cite{chiyonobu1988large}]
The empirical process $M_t(X)$ of a hypermixing process
$X \sim \mu$ satisfies the large deviation principle with rate function
\[
    \I(\eta) = \begin{cases}
        h(\eta \, \|\, \mu), & \text{ if } \eta \in \M_1(\cX, \xmap), \\ 
	        \infty, & \text{ otherwise.}
    \end{cases}
\]
\end{theorem}

\begin{remark}
    We will see that property \eqref{cond:H-1} of hypermixing processes 
    is already enough for posterior consistency. 
\end{remark}

\begin{remark}
    In \cite{bryc1996large}, a process with only
    a slight modification of property \eqref{cond:H-2}  
    satisfies the LDP, but the rate function may not coincide with 
    the relative entropy rate. Properties of the relative entropy
    rate are crucial for our approach.
\end{remark}

\subsection{Concrete examples of hypermixing processes}
We now present some concrete examples of 
stochastic processes that are hypermixing.

\begin{example}[I.I.D. processes]
    Consider the discrete time case $\T = \nat$. 
    Let $(X_t)_{t \in \T}$ be an $\S$-valued i.i.d process with law $\mu$ on $\cX$,
    then $\mu$ is hypermixing.
\end{example}

\begin{example}[Markov processes]
    Let $(X_t)_{t \in \T}$ be an $\S$-valued Markov process which can be
    realized on $\cX = C(\T, \S)$, and we denote the associated Markov semigroup
    acting on bounded measurable functions on $\S$
    by $(P_t)_{t \in \T}$. Let $\mu_0 \in \M_1(\S)$ be an invariant measure
    (in the sense of Markov processes) with respect to $(P_t)_{t \in \T}$,
    and $\mu \in \M_1(\cX, \xmap)$ be the corresponding $\xmap$-invariant
    measure on the path space $\cX$. 
    Then according to \cite[Theorem 6.1.3]{deuschel2001large}, we have
    the following characterization of the hypermixing property.
\end{example} 
    \begin{proposition}[\cite{deuschel2001large}]
    $\mu$ is hypermixing if and only if the operator norm
    of $P_T : L^2(\mu_0) \to L^4(\mu_0)$ is one for some $T > 0$, i.e., 
    the Markov semigroup $(P_t)_{t \in \T}$ is $\mu_0$-hypercontractive.
    \end{proposition}

    Hypercontractivity is implied many other well-studied properties
    such as the logarithmic Sobolev inequality, which can be verified
    in a wide class of examples.
    See \cite[Chapter VI]{deuschel2001large} for more details and \cite{bakry2013analysis} for a general introduction.
    We will outline a class of Markov processes 
    precisely in Section \ref{sec:Markov} to illustrate the power of our framework,
    which essentially becomes a black box. We will also mention other 
    Markov processes, where 
    hypercontractivity is proved by other tools.

\begin{example}[More examples]
    More examples of hypermixing processes can be found at
    \cite{chiyonobu1988large}, including $\epsilon$-Markov processes
    (i.e., generalized Markov processes where the current state depends on the past up to a
    time interval of length $\epsilon$) 
    and some Gaussian processes with certain decay in their spectral density
    matrices.
\end{example}

\subsection{The structure of the large deviation proof}
We apply the strategy outlined in Section \ref{sec:strategy} to the
case of hypermixing processes. We also detail the components of the proof.

First, we work on one hypermixing process $(X_t)_{t \in \T} \sim \mu$. 
The immediate goal is to show for any $f \in C_{b, loc} (\cX \times \cY)$,
it holds for $\nu$-a.e. $y \in \cY$ that
\begin{equation*} 
    \lim_{t\to\infty} \frac{1}{t}\log \int_\cX \exp\paran{f^t(x, y)}d\mu(x)
	= \sup_{\lambda \in \J(\xmap : \nu)}\set{\int f\,d\lambda - h(\lambda\,\|\,\mu\otimes\nu)}.
\end{equation*}
We will established this in a series of steps:
\begin{enumerate}
	\item  
	We show that given $f \in C_{b, loc} (\cX \times \cY)$, 
	for $\nu$-a.e. $y\in\cY$, the limit
	\begin{align*}
 		p(f) & := \limsup_{t\to\infty} \frac{1}{t}\log
 		\int_\cX \exp\paran{f^t(x,y)}\,d\mu(x) 
	\end{align*} 
	exists and does not depend on $y$. 
	In particular, $p: C_{b, loc}(\cX \times \cY) \to \real$ 
	is called the pressure functional.
		
	\item We show that 
	$$p^*(\lambda) = 
		\begin{cases} h(\lambda\,\|\,\mu \otimes \nu) \quad 
		& \text{if } \lambda \in \J(S : \nu), \\ 
		\infty \quad & \text{else} \end{cases}$$
	where $p^*: \M_1(\cX \times \cY) \to \real$ is the Fenchel-Legendre transform of $p$, i.e.,
	$$p^*(\lambda) = \sup\set{\int f\,d\lambda - p(f) \,:
	\, f \in C_{b, loc} (\cX \times \cY)}.$$ 
	This establishes the existence of the relative entropy rate.
		
	\item As suggested by Theorem \ref{thm:laplace},
	we show the upper bound 
	\[
	    \limsup_{t\to\infty} \frac{1}{t}\log \int_\cX \exp\paran{f^t(x, y)}d\mu(x)
    	\le \sup_{\lambda \in \J(\xmap : \nu)}\set{\int f\,d\lambda - h(\lambda\,\|\,\mu\otimes\nu)}
	\]
	by proving the large deviation upper bound of 
	$(M_t(X, y))_{t \in \T}$.

	\item We show the lower bound
	\[
	    \liminf_{t\to\infty} \frac{1}{t}\log \int_\cX \exp\paran{f^t(x, y)}d\mu(x)
    	\ge \sup_{\lambda \in \J(\xmap : \nu)}\set{\int f\,d\lambda - h(\lambda\,\|\,\mu\otimes\nu)}
	\]
	by exploiting the properties of the relative entropy rate and linearity 
	of integration on measures.
\end{enumerate}
The steps above prove new quenched large deviation
asymptotics for hypermixing processes. Next,
we need to show that the parameterized family of processes 
is exponentially continuous to obtain
posterior consistency.
Toward this end, we introduce the notion of a regular family 
$\set{\mu_\theta}_{\theta \in \Theta}$ of hypermixing processes.
We show that when $\set{\mu_\theta}_{\theta \in \Theta}$ is regular, 
it is exponentially continuous with respect to the loss function $L$. As a result, we also obtain new annealed large deviation
asymptotics for hypermixing processes.

\subsection{The pressure functional}
The main result of this subsection is the existence of the pressure
functional $p$. Once $p$ is established, the existence and properties
of the relative entropy rate will hold immediately, and we will obtain
the large deviation upper bound in the following subsections.
% We will adapt the argument in \cite[Section 5.4]{deuschel2001large}
% to our setting.
% \begin{equation}
%     f^t(x,y) := f^{[0,t]}(x, y) = t\inprod{M_t(x,y)}{f}_{
%     \M_1(\X \times \Y), \, C_{b, loc}(\X \times \Y)}
% \end{equation}

\begin{theorem} \label{hm:pressure}
    If $f \in C_{b, loc}(\cX \times \cY)$, then
	the following limits exist, and the equalities hold
	\begin{equation}
		\begin{split}
		p(f) & := \lim_{t \to \infty} \frac{1}{t}\int_\cY \log\paran{
			\int_\cX \exp\paran{f^t(x,y)}
			d\mu(x)} d\nu(y) \\
	 	& = \int_\cY  \limsup_{t \to \infty} \frac{1}{t}\log\paran{
	 		\int_\cX \exp\paran{f^t(x,y)}
	 		d\mu(x)} d\nu(y) \\
	 	& = \limsup_{t \to \infty} \frac{1}{t}\log\int_\cX \exp\paran{
	 		f^t(x,y)}d\mu(x), 
		\end{split}
	\end{equation}
% 	for $\nu$-a.e. $y \in \Y$ 
    for $y \in E$, where $E \subseteq \cY$ is a set of full $\nu$-measure 
    that is independent of $f$.
\end{theorem}
The proof combines the ideas from
\cite{chiyonobu1988large} and \cite[Lemma 2]{chi2002conditional}.
\begin{proof}
	Let $f \in C_{b,loc}(\cX \times \cY)$. Then $f$ is $\tF_r$-measurable 
	for some $r \ge 0.$ Let $\ell > \ell_0$ and $r'(\ell) = \ell + r$. 
	For $y \in \cY$, write
	\[
		p_t(f, y) = \frac{1}{t} \log \int_\cX \exp 
			\paran{f^t(x,y)} d\mu(x).
	\]
	Fix $s, t \in \T, t > s$. 
	The idea is to divide $[0,t]$ into smaller intervals of length $s$ 
	that are $r'(\ell)$-separated to apply (\ref{cond:H-1}) on the smaller intervals.
	For $u \in \T$, let $J_u^{s,t}$ be the collection of disjoint intervals 
	of the form $$u + (s + r'(\ell))m + [0, s] \subset [0, t], \quad m \in \nat,$$ 
	and $\bar{I}_u^{s,t} = [0, t] \setminus \bigcup_{I \in J_u^{s,t}} I$. 
% 	We supress the dependence on $s$, $t$ for ease of notations. 
	Then for $u \in [0, s + r'(\ell)]$, by properties of $f$ and (\ref{cond:H-1}),
	\begin{equation} 
	\begin{split}
		p_t(f,y) & \le \frac{1}{t} \log \int_\cX \exp 
			\paran{{|\bar{I}_u^{s,t}|}\|f\| + \sum_{I \in J_u^{s,t}} f^I(x,y)} d\mu(x) \\
		& = \frac{\|f\|}{t}|\bar{I}_u^{s,t}| 
			+ \frac{1}{t} \log \int_\cX \exp \paran{\sum_{I \in J_u^{s,t}} f^I(x,y)} d\mu(x) \\
		& \le \frac{\|f\|}{t} |\bar{I}_u^{s,t}|
			+ \frac{1}{t\alpha(\ell)} \sum_{I \in J_u^{s,t}} \log \int_\cX \exp 
				\paran{\alpha(\ell) f^I(x,y)} d\mu(x), \label{ieq:p}
	\end{split}
	\end{equation}
	where $|\bar{I}_u^{s,t}|$ denotes the total length (the Lebesgue measure) of $\bar{I}_u^{s,t}$.
	
	Next, we average the inequality \eqref{ieq:p} over $u$ to obtain a tighter estimate,
	which is given by the following lemma, whose proof is deferred
	to the end of this subsection.
	
    \begin{lemma} \label{lem:rsum}
        We have for any $t > s$,
        \begin{equation} \label{eq:ub-est}
		p_t(f, y) \le A_{s,t} + B_{s,t}(y)
	    \end{equation}
    	where
    % 	\[
    % 	    A_{s,t} := \frac{\|f\|}{t} 
    % 		 \paran{t - \left\lfloor{\frac{t - s}{s + r'(\ell)}}\right\rfloor s},
    % 	\]
        \[
    	    \lim_{s \to \infty} \lim_{t \to \infty} A_{s, t} = 0,
    	\]
    	\[
    	    B_{s,t}(y) := \frac{1}{t \alpha(\ell)} \frac{s}{s + r'(\ell)} 
    				\int_0^{t - s} p_s(\alpha(\ell) f, \ymap^u y)\,du.
    	\]
	\end{lemma}
	\begin{remark}
	    The idea of Lemma \ref{lem:rsum} is to observe
    	\begin{equation} \label{eq:intervals}
    	    \bigcup_{u \in [0, s + r'(\ell)]} J_u^{s,t} = \set{u + [0, s]\,:\, u \in [0, t - s]}.
    	\end{equation}
    	An examination of the time intervals in the integrals on the right-hand side of 
    	\eqref{ieq:p}
    	allows us to formally show, up to a vanishing remainder,
    	\begin{multline*}
    	    \int_0^{s + r'(\ell)} \sum_{I \in J_u^{s,t}} \log \int_\cX \exp 
    				\paran{\alpha(\ell) f^I(x,y)} d\mu(x) \, du \\
    		= \int_0^{t - s} \log \int_\cX \exp 
    				\paran{\alpha(\ell) f^{u + [0, s]}(x, y)} d\mu(x) \,du,
    	\end{multline*}
    	and apply $\xmap$-invariance of $\mu$.
    	The above equality is exact and straightforward in discrete time,
    	where integrals in time are replaced by summations,
    	and the equality follows by rearranging the summations using \eqref{eq:intervals}.
    	For continuous time, we pass to the limit to obtain
    	the Riemann integrals.
	\end{remark}
	
	(Proof of Theorem \ref{hm:pressure}, Continued.)
	From \eqref{ieq:p} and Lemma \ref{lem:rsum}, we have
	\begin{equation*} 
		p_t(f, y) \le A_{s,t} + B_{s,t}(y).
	\end{equation*}
% 	where by the $\xmap$-invariance of $\mu$, 
% 	\begin{align*}
% 	    B_{s,t}(y) & := \frac{1}{t \alpha(\ell)} \frac{1}{s + r'(\ell)} 
% 				\int_0^{t - s} \log \int_\X \exp 
% 				\paran{\alpha(\ell) f^{u + [0, s]}(x, y)} d\mu(x) \,du \\
% 		  & = \frac{1}{t \alpha(\ell)} \frac{1}{s + r'(\ell)} 
% 				\int_0^{t - s} \log \int_\X \exp 
% 				\paran{\alpha(\ell) f^{[0, s]}(x, \ymap^u y)} d\mu(x) \,du \\
% 		  & = \frac{1}{t \alpha(\ell)} \frac{s}{s + r'(\ell)} 
% 				\int_0^{t - s} p_s(\alpha(\ell) f, \ymap^u y)\,du.
% 	\end{align*}
	Integrate with respect to $\nu$ on both sides and obtain
	\[
	\begin{split}
		\int_\cY p_t(f, y) \,d\nu(y) 
	    & \le A_{s,t} + \int_\cY B_{s,t}(y) \,d\nu(y) \\ 
	\end{split}
	\]
	where by the $\ymap$-invariance of $\nu$,
	\[
	    \begin{split}
	     \int_\cY B_{s,t}(y) \,d\nu(y) 
	     & = \frac{t-s}{t \alpha(\ell)} \frac{s}{s + r'(\ell)} 
			\int_\cY p_s(\alpha(\ell) f, y)\,d\nu(y) \\
		 & \le \frac{t-s}{t } \frac{s}{s + r'(\ell)} 
				\paran{\int_\cY p_s(f, y)\,d\nu(y) + (\alpha(\ell) - 1) \|f\|},
		\end{split}
	\]
	and we used the bound
	\begin{equation} \label{eq:pl}
		\frac{1}{\alpha(\ell)} p_s(\alpha(\ell) f, y) 
		\le  p_s(f, y) + (\alpha(\ell) - 1) \|f\|.
	\end{equation}
	By taking $t \to \infty$, $s \to \infty$ and $\ell \to \infty$ 
	successively, we obtain
	\[
		\limsup_{t \to \infty} \int_\cY p_t(f, y)\,d\nu(y) 
			\le \liminf_{s \to \infty} \int_\cY p_s(f, y)\,d\nu(y).
	\]
	Therefore, the limit $p(f)$ exists. 
	   
	Recall from \eqref{eq:ub-est}, we have
	\[
	    p_t(f, y) \le A_{s,t}
			+ \frac{1}{t \alpha(\ell)} \frac{s}{s + r'(\ell)} 
				\int_0^{t - s} p_s(\alpha(\ell) f, \ymap^u y)\,du.
	\]
	By taking $t \to \infty$, we apply Birkhoff's pointwise ergodic theorem  
	(Theorem \ref{thm:ergodic})
	and \eqref{eq:pl}
	on the right hand side above 
	and obtain, for $\nu$-a.e. $y \in \cY$ independent of $f$, 
% 	\[
% 	    \limsup_{t \to \infty} p_t(f,y) \le 
% 	    \|f\| \paran{1 - \frac{s}{s + r'(\ell)}} + 
% 	        \frac{1}{\alpha(\ell)}\frac{s}{s + r'(\ell)}
% 	        \int_\Y p_s(\alpha(\ell)f,y)\,d\nu(y).
% 	\]
	\[
	    \limsup_{t \to \infty} p_t(f,y) \le 
	     \lim_{t \to \infty} A_{s,t} + 
	        \frac{s}{s + r'(\ell)}
	        \left(\int_\cY p_s(f,y)\,d\nu(y) + (\alpha(\ell) - 1) \|f\|\right).
	\]
	By taking $s \to \infty$ and then $\ell \to \infty$, 
	we have $\displaystyle \limsup_{t \to \infty} p_t(f,y) \le p(f)$.
	By Fatou's lemma, we also have 
	$$ p(f) \le \int_\cY \limsup_{t \to \infty} p_t(f,y) \,d\nu(y). $$
	Thus, it must be $\displaystyle \limsup_{t \to \infty} p_t(f,y) = p(f)$
	for $\nu$-a.e. $y \in \cY$.
\end{proof}
\begin{remark}
    When $\T = \nat$ and $\mu$ is the law of an i.i.d. process,
    the $\limsup$ in Theorem \ref{hm:pressure} can be replaced by the exact limit.
\end{remark}

To complete the argument, we present the proof of Lemma \ref{lem:rsum}.
\begin{proof}[Proof of Lemma \ref{lem:rsum}]
    Note that the number of intervals in $J_u^{s,t}$ 
    is at least $\left\lfloor \frac{t + r'(\ell)}{s + r'(\ell)} - 1 \right\rfloor$ 
    for $u \in [0, s + r'(\ell)]$,
    so 
    $|\bar{I}_u^{s,t}| \le t - \lfloor\frac{t - s}{s + r'(\ell)}\rfloor s$.
    By overloading notations of superscripts, we write $$F^I(y) := \log \int_\cX \exp 
			\paran{\alpha(\ell) f^{I}(x, y)} d\mu(x).$$
    Let $0 = u_0^m < u_1^m < \cdots < u_{m}^m = s + r'(\ell)$ 
    be a partition of $[0, s + r'(\ell)]$ such that $u_{i+1}^m - u_i^m = \frac{s + r'(\ell)}{m}$.
    Then by ordering the left endpoints of intervals in
    $\bigcup_{i=0}^{m-1} J_{u_i}^{s,t}$,
    we obtain a partition 
    $0 = v_0^m \le v_1^m \le \ldots \le v_{n-1}^m \le t - s$ of $[0, t-s]$
    by setting $v_n^m = t - s$, and we have $\max_j {v_{j+1}^m - v_j^m} \le \frac{s + r'(\ell)}{m}.$
    Note that $J_u^{s,t}$ has at most $\lfloor\frac{t + r'(\ell)}{s + r'(\ell)}\rfloor$ 
    intervals for any $u \in [0, s + r'(\ell)]$,
    so $\left(\lfloor\frac{t + r'(\ell)}{s + r'(\ell)}\rfloor -1 \right) m \le n \le \lfloor\frac{t + r'(\ell)}{s + r'(\ell)}\rfloor m$. 
    
    Consider the Riemann sum
    \begin{align*}
        \frac{s + r'(\ell)}{m}\sum_{i=0}^{m-1}
        \sum_{I \in J_{u_i}^{s,t}} F^I(y) 
        = \frac{s + r'(\ell)}{m}\sum_{j=0}^{n-1} F^{v_j^m + [0,s]}(y).
    \end{align*}
    Note that 
    \begin{multline*}
        \abs{\frac{s + r'(\ell)}{m}\sum_{j=0}^{n-1} F^{v_j^m + [0,s]}(y) 
        - \sum_{j=0}^{n-1} (v_{j+1}^m - v_j^m) F^{v_j^m + [0,s]}(y)} \\
        \le s\|f\| \left| \frac{n(s + r'(\ell))}{m} - (t - s) \right|
        = o(t), 
    \end{multline*}
    where $o(t)$ denotes a term with $\lim_{t \to \infty}o(t)/t = 0$
    and hence, 
    \begin{align*}
        \lim_{m \to \infty}\frac{s + r'(\ell)}{m}\sum_{i=0}^{m-1}
        \sum_{I \in J_{u_i}} F^I(y) 
        & \le o(t) +  \lim_{m \to \infty} \sum_{j=0}^{n-1} (v_{j+1}^m - v_j^m) F^{v_j^m + [0,s]}(y) \\
        & = o(t) + \int_0^{t - s} F^{u + [0,s]}(y)\,du.
    \end{align*}
    Therefore, by averaging the right-hand side of the inequality \eqref{ieq:p},
    \[
        p_t(f,y) \le A_{s,t} + \frac{1}{t\alpha(\ell)}
        \frac{1}{m}\sum_{i=0}^{m-1}
        \sum_{I \in J_{u_i}^{s,t}} F^I(y) 
        = A_{s,t} + \frac{1}{t\alpha(\ell)}\frac{1}{s + r'(\ell)}
        \frac{s + r'(\ell)}{m}\sum_{i=0}^{m-1}
        \sum_{I \in J_{u_i}^{s,t}} F^I(y)
    \]
    where 
    \[
       A_{s,t} = \frac{\|f\|}{t} 
        \paran{t - \left\lfloor{\frac{t - s}{s + r'(\ell)}}\right\rfloor s} + \frac{o(t)}{t}.
    \]
    By taking $m \to \infty$ and by the $\xmap$-invariance of $\mu$,
    \begin{align*}
	    B_{s,t}(y) & = \frac{1}{t \alpha(\ell)} \frac{1}{s + r'(\ell)} 
				\int_0^{t - s} \log \int_\cX \exp 
				\paran{\alpha(\ell) f^{u + [0, s]}(x, y)} d\mu(x) \,du \\
		  & = \frac{1}{t \alpha(\ell)} \frac{1}{s + r'(\ell)} 
				\int_0^{t - s} \log \int_\cX \exp 
				\paran{\alpha(\ell) f^{[0, s]}(x, \ymap^u y)} d\mu(x) \,du \\
		  & = \frac{1}{t \alpha(\ell)} \frac{s}{s + r'(\ell)} 
				\int_0^{t - s} p_s(\alpha(\ell) f, \ymap^u y)\,du.
    \end{align*}
\end{proof}

\subsection{The relative entropy rate}
With the existence of the pressure functional $p$, 
the existence of the relative entropy rate
can be established by a convex duality argument.
We also prove properties of the relative entropy rate
which will be crucial for proving the lower bound.

\begin{proposition} \label{prop:hm-entropy}
	For $\lambda \in \J(\xmap  : \nu)$, 
	$h(\lambda\,\|\,\mu \otimes \nu)$ as defined by Definition \ref{def:entropyXY}  
	exists as
	$h(\lambda\,\|\,\mu \otimes \nu) = p^*(\lambda)$, 
	where $p^*$ is the Fenchel-Legendre transform of $p$, i.e.,
	\begin{equation} \label{eq:p*}
	    p^*(\lambda) = 
		\sup\set{\int f\,d\lambda - p(f)  :  f \in C_{b, loc} (\cX \times \cY)}.
	\end{equation}
	In particular,
	we have for $\nu$-a.e. $y \in \cY$,
	\begin{equation} \label{eq:y-entropy}
	    \lim_{t \to \infty} \frac{1}{t} K(\lambda_y|_{\F_t}\,\|\,\mu|_{\F_t})
		= h(\lambda\,\|\,\mu \otimes \nu).
	\end{equation}
\end{proposition}

The proof of Proposition \ref{prop:hm-entropy}
requires the following two lemmas, whose proofs
are deferred to Section \ref{append:hm}.
The first one provides a useful estimate for obtaining $p^*$.
The second one establishes a generalized subaddtive property
for proving \eqref{eq:y-entropy}, which
can be viewed as a subadditive ergodic result,
and combined with Section \ref{sec:subadd}, 
it may be of independent interest for ergodic theory
of random dynamical systems.

\begin{lemma} \label{lem:pest}
	If $f \in C_{b, loc}(\cX \times \cY)$ is $\tF_r$-measurable,
	then for $\ell > \ell_0$,
	\begin{equation}
		p(f) \le \frac{1}{(\ell + r)\alpha(\ell)}
			\int_\cY \log\paran{\int_\cX 
			e^{(\ell + r)\alpha(\ell)f(x,y)}d\mu(x)}d\nu(y).
	\end{equation}
\end{lemma}

\begin{lemma} \label{lem:hm-subadd}
	Let $\lambda \in \J(\xmap : \nu)$ with its disintegration 
	$\lambda = \int \lambda_y \otimes \delta_y \, d\nu(y)$ (see Theorem \ref{thm:disint}).
	For $t \ge 0$, define a function $F_t: \cY \to [-\infty, 0]$ by 
	$F_t(y) = -K(\lambda_y|_{{{\F}_{t}}}\,\|\,\mu|_{{{\F}_{t}}})$.
	Fix $\ell > \ell_0$. Then for $\nu$-a.e. y $\in \cY$, any $n \in \nat$ and
	$t_1, \ldots, t_n \in [0, \infty)$,
	\[
		F_{(n-1)\ell + \sum_{k=1}^{n}t_k}(y) \le \frac{1}{\alpha(\ell)} 
	    \sum_{k=1}^{n} F_{t_k}(\ymap^{(k-1)\ell + \sum_{j=1}^{k-1}t_j} y).
	\]
\end{lemma}

\begin{proof}[Proof of Proposition \ref{prop:hm-entropy}]
	We start with the first statement. 
	Fix $r \ge 0$ and let $f \in C_{b, loc}(\cX \times \cY)$ be a $\tF_r$-measurable function.
	
	By the definition of $p^*$, Lemma \ref{lem:pest} and Jensen's inequality,
	\begin{align*}
		p^*(\lambda) & \ge \int \frac{f}{(\ell + r) \alpha(\ell)}\,d\lambda 
			- p\paran{\frac{f}{(\ell + r) \alpha(\ell)}} \\
		& \ge \frac{1}{(\ell + r) \alpha(\ell)} \paran{\int f\, d\lambda 
			- \int_\cY \log \paran{\int_\cX e^{f(x,y)} \, d\mu(x)}d\nu(y)} \\
		& \ge \frac{1}{(\ell + r) \alpha(\ell)} \paran{\int f\, d\lambda 
			- \log \int e^f \, d(\mu \otimes \nu)}. 
	\end{align*}
	Taking suprememum over such $f$ implies 
	$p^*(\lambda) \ge \frac{K(\lambda|_{{\tF_{r}}}\,\|\,\mu \otimes \nu|_{{\tF_{r}}})
		}{(\ell + r)\alpha(\ell)}$, 
	so by taking $r \to \infty$, we obtain
	$p^*(\lambda) \ge h(\lambda\,\|\,\mu \otimes \nu).$

	To prove the other side of inequality, 
	note that by uniqueness of disintegration (Theorem \ref{thm:disint}) 
	and the chain rule of 
	the relative entropy (Theorem \ref{rel:chain}), we have
	\begin{align*}
	    K(\lambda|_{{\tF_{t}}}\,\|\,\mu \otimes \nu|_{{\tF_{t}}}) 
		& = K(\nu\,|\,\nu) + \int_{\cY} K(\lambda_y|_{{\F_{t}}}\,\|\,\mu|_{{\F_{t}}})\,d\nu(y) \\
		& = \int_{\cY} K(\lambda_y|_{{\F_{t}}}\,\|\,\mu|_{{\F_{t}}})\,d\nu(y).
	\end{align*}  
	Now note that $f^t$ is $\tF_{t + r}$-measurable.
	In particular, for every $y \in \cY$, $f^t(\cdot, y)$ is 
	${\F}_{t+r}$-measurable. Then 
	by the variational property of relative entropy (Theorem \ref{rel:var}),
	\begin{equation} \label{eq:yent}
		K(\lambda_y|_{{{\F}_{t+r}}}\,\|\,\mu|_{{{\F}_{t+r}}}) 
		\ge \int f^t\,d\lambda_y - \log \int_\cX \exp\paran{
		f^t(x,y)} d\mu(x).
	\end{equation}
	Integrate both sides with respect $\nu$ and obtain
	\begin{align*}
		 K(\lambda|_{{\tF_{t+r}}}\,\|\,\mu\otimes\nu|_{{\tF_{t+r}}}) 
		 &\ge \int f^t\,d\lambda - \int_\cY \log \paran{
		 \int_\cX \exp\paran{f^t(x,y)}
		 d\mu(x)}d\nu(y) \\ 
		 &= t \int f\,d\lambda - \int_\cY \log \paran{
		 \int_\cX \exp\paran{f^t(x,y)}\,d\mu(x)}d\nu(y). 
	\end{align*}
	Divide both sides by $t$ and take $t \to \infty$ and get
	$$h(\lambda\,\|\,\mu \otimes \nu) \ge \int f\,d\lambda - p(f).$$ 
	Theorefore, $h(\lambda\,\|\, \mu \otimes \nu) \ge p^*(\lambda).$

	We show the last statement.
	For $t \ge 0$, define a sequence of function $F_t: \cY \to [-\infty, 0]$ by 
	$F_t(y) = -K(\lambda_y|_{{{\F}_{t}}}\,\|\,\mu|_{{{\F}_{t}}})$.
	Based on Section \ref{sec:subadd}, since \eqref{eq:subadd} holds 
	by Lemma \ref{lem:hm-subadd}, then by
	using Proposition \ref{prop:dtoc},
	we have $\lim_{t \to \infty} \frac{F_t(y)}{t}$ exists
	for $\nu$-a.e. $y \in \cY$, and
	\[
	    \int_\cY \lim_{t \to \infty} \frac{F_t(y)}{t} \,d\nu(y)
	    = -h(\lambda\,\|\,\mu \otimes \nu).
	\]
	Since $ \lim_{t \to \infty} \frac{F_t}{t}$ 
	is a $\ymap$-invariant function (see \cite{avila2009subadditive})
	and $\nu$ is $\ymap$-ergodic,  
	it must hold for $\nu$-a.e. $y \in \cY$ that
	\[
		\lim_{t \to \infty}
		\frac{1}{t}K(\lambda_y|_{{{\F}_{t}}}\,\|\,\mu|_{{{\F}_{t}}}) 
		= h(\lambda\,\|\,\mu \otimes \nu).
	\]
\end{proof}

\begin{remark} \label{rmk:pressure-iid}
    When $\T = \nat$ and $\mu$ is the law of an i.i.d. process,
    \eqref{eq:y-entropy} follows directly from the
    Subadditive Ergodic Theorem \cite{avila2009subadditive}.
\end{remark}

Once the existence of the relative entropy rate is established,
the following properties are known consequences.
\begin{proposition} \label{prop:affine}
    The relative entropy rate
	$h(\cdot\,\|\,\mu\otimes\nu)$ is a convex and lower semicontinuous function 
	on $\J(\xmap : \nu)$. Moreover, $h(\cdot\,\|\,\mu\otimes\nu)$ 
	is affine (see Lemma \ref{lem:affine}) on $\J(\xmap : \nu)$.
\end{proposition}
\begin{proof}
    The proof is the same as that for the same properties of $h(\cdot\,\|\,\mu)$ in
    \cite[Proposition 6.8]{rassoul2015course} and \cite[(5.4.23)]{deuschel2001large}.
\end{proof}

\subsection{The large deviation upper bound}
We now prove the following upper bound.
\begin{proposition} \label{hm:upper-bound}
	Let $f \in C_b(\cX \times \cY)$. 
	% be a upper semicontinuous function that is bounded above. 
	Then for $\nu$-a.e. $y \in \cY$,
	\begin{align*}
	    \limsup_{t\to\infty} \frac{1}{t}\log
    	\int_\cX \exp\paran{f^t(x, y)}d\mu(x)
    	\le \sup_{\lambda \in \J(\xmap : \nu)}\set{\int f\,d\lambda - h(\lambda\,\|\,\mu\otimes\nu)}.
	\end{align*}
\end{proposition}
According to Theorem \ref{thm:laplace}, it suffices to show
the large deviation upper bound \eqref{eq:ldpub} for
empirical process $(M_t(X, y))_{t \in \T}$ with $X \sim \mu$,
for $\nu$-a.e. $y \in \cY$. 
We first establish the weak large deviation upper bound,
which together with exponential tightness implies 
the large deviation upper bound,
as a common procedure in large deviation theory.
See \cite[Theorem 2.19]{rassoul2015course}.

\begin{lemma}\label{lem:genericUB}
	For $\nu$-a.e. $y \in \cY$, it holds that
	for any compact set $\K \subset \M_1(\cX \times \cY)$, 
	$$\limsup_{t \to \infty} \frac{1}{t}\log \mu\paran{
	M_t(\cdot,y) \in \K } 
	\le -\inf_{\K} {p}^*,$$
	where ${p}^*$ is the Fenchel-Legendre transform (also known as the convex conjugate) 
	of ${p}$, i.e.,
	$${p}^*(\lambda) := \sup\set{\int f\,d\lambda - {p}(f) \,:
	\, f \in C_{b,loc} (\cX \times \cY)}.$$
	In other words, for $\nu$-a.e. $y \in \cY$, the empirical process
	$(M_t(X, y))_{t \in \T}$ with $X \sim \mu$ satisfies the weak large
	deviation upper bound with rate function $p^*$.
	Moreover, ${p}^*(\lambda) = \infty$ if $\lambda$ is not 
	$(\xmap \times \ymap)$-invariant, or $\lambda$ does not have $\cY$-marginal $\nu$.
\end{lemma}
With Theorem \ref{hm:pressure} in place, the proof of Lemma \ref{lem:genericUB}
is standard in large deviation theory, so we defer it to Section \ref{sec:LD}.
Now we can prove the desired upper bound. 
\begin{proof}[Proof of Proposition \ref{hm:upper-bound}]
    By the first part of Lemma \ref{lem:genericUB},
    we know that for $\nu$-a.e. $y \in \cY$, the empirical process
    $(M_t(X, y))_{t \in \T}$ with $X \sim \mu$
    satisfies the weak large
	deviation upper bound with rate function $p^*$. 
	Since $(M_t(X))_{t \in \T}$ is exponentially tight (implied by its LDP),
	then by Lemma \ref{lem:exptight}, for $\nu$-a.e. $y \in \cY$,
	$(M_t(X, y))_{t \in \T}$ is also
	exponentially tight, hence satisfying the large deviation 
	upper bound \eqref{eq:ldpub}. The statement follows from
	Theorem \ref{thm:laplace}, Proposition \ref{prop:hm-entropy} and
	the second part of Lemma \ref{lem:genericUB}.
\end{proof}

\subsection{The large deviation lower bound}
We now prove the following lower bound. 
\begin{proposition}\label{hm:lower-bound}
	Let $f \in C_b(\cX \times \cY)$. 
	Then for $\nu$-a.e. $y \in \cY$,
	\begin{align*}
	    \liminf_{t\to\infty} \frac{1}{t}\log
    	\int_\cX \exp\paran{f^t(x,y)}d\mu(x)
    	\ge \sup_{\lambda \in \J(\xmap : \nu)}\set{\int f\,d\lambda - h(\lambda\,\|\,\mu\otimes\nu)}.
	\end{align*}
\end{proposition}
\begin{remark}
    When $\T = \nat$ and $\mu$ is the law of an i.i.d. process,
    the lower bound follows directly from the inequality $p^{**} \le p$ in
    convex duality, where $p^{**}$ is the convex conjugate of $p^*$.
\end{remark}
We apply a standard approach for proving lower bounds, see \cite[Proof of Theorem 6.13]{rassoul2015course} 
and \cite[Lemma 5.4.21]{deuschel2001large}. We will use the affine property of the relative entropy rate
and the integrated cost as a linear (hence affine) functional on measures, and see that it is 
enough to prove the case when $\lambda$ is ergodic. We do not need to consider
the full large deviation lower bound \eqref{eq:ldplb}. Moreover, the proof holds for other processes as long as the
corresponding relative entropy rate is an affine function on invariant measures.
The following lemma is the main technical step.
\begin{lemma} \label{lem:open}
	Let $\lambda \in \J_e(\xmap : \nu)$.
	Then for $\nu$-a.e. $y \in \cY$, it holds that
	for any open neighborhood $U \subset \M_1(\cX \times \cY)$ 
	of $\lambda$,
	\[
		\liminf_{t \to \infty} \frac{1}{t} \log 
		\mu\left(M_t(\cdot,y) \in U \right)
		\ge - h(\lambda\,\|\,\mu \otimes \nu).
	\]
\end{lemma}

\begin{proof}
	Assume that $h(\lambda\,\|\,\mu \otimes \nu) < \infty$. 
	Then $K(\lambda|_{{\tF_{t}}}\,\|\,\mu\otimes\nu|_{{\tF_{t}}}) < \infty$
	for all $t \ge 0$. Pick a countable sequence of time $t_n \to \infty$.
	Consider the disintegration $\lambda = \int \lambda_y \otimes \delta_y \, d\nu(y)$
	by Theorem \ref{thm:disint}.
	We can take a set $E \subset \cY$ of full $\nu$-measure of $y \in \cY$
	such that $K(\lambda_y |_{\F_{t_n}}\,\|\,\mu|_{\F_{t_n}}) < \infty.$
	Since $t_n \to \infty$, 
	if $y \in E$, 
	$K(\lambda_y |_{\F_{t}}\,\|\,\mu|_{\F_{t}}) < \infty,$ 
	for any $t \ge 0$.

	For $y \in E$, let $g_t^y := \frac{d\lambda_y}{d\mu}: \cX \to [0,\infty]$ be the Radon-Nikodym
	derivative on $\F_t$.
	Recall that $\M_1(\cX \times \cY)$ is a Polish space, which has a countable base.
	Let $U$ be a base element containing $\lambda$. 
	By Appendix \ref{sec:prob}, $U$ contains the
	subset
	$$\set{\eta \in \M_1(\cX \times \cY)\,:
	\, \abs{\int f_i \,d\eta - \int f_i \,d\lambda} < \epsilon, i = 1,\ldots, m} $$
	for some $\epsilon > 0$ small enough 
	and some $f_1, f_2, \ldots, f_m \in C_{b,loc}(\cX \times \cY)$ that are $\tF_r$-measurable for $r$ large enough.
	Without loss of generality, it suffices to consider $U$ of this form.
	Then the event $\{M_t(\cdot,y) \in U\}$ is $\F_{t+r}$
	measurable for $r$ large enough.
	Consider $$A_t(y) := \set{x \in \cX \,:\,M_t(x,y) \in U, \, g_{t+r}^y(x) > 0}.$$ By change of measure and Jensen's inequality,
	\begin{align*}
		&\, \frac{1}{t} \log 
		\mu\left(M_t(\cdot,y) \in U \right) \\
		 \ge &\, \frac{1}{t} \log \int_{A_t(y)} {(g^y_{t+r})}^{-1}(x) \,d\lambda_y(x) \\
		 = &\, \frac{1}{t} \log \lambda_y(A_t(y)) + 
			\frac{1}{t} \log\left(\frac{1}{\lambda_y(A_t(y))}
				\int_{A_{t}(y)} {(g^y_{t+r})}^{-1}(x) \,d\lambda_y(x) \right) \\ 
		\ge &\, \frac{1}{t} \log \lambda_y(A_t(y)) - 
			\frac{1}{t\lambda_y(A_t(y))} \left(
				\int_{A_t(y)} 
				\log {g^y_{t+r}}(x) \,d\lambda_y(x) \right). 
	\end{align*}
	By using $x\log x \ge -1/e$ for $x > 0$,
	\begin{align*}
		\int_{A_t(y)} \log {g^y_{t+r}}(x) \,d\lambda_y(x) 
		& = \int_\cX \log g^y_{t+r}\, d\lambda_y 
		- \int_{A_t(y)} g^y_{t+r} \log g^y_{t+r}\, d\mu \, \\
		& \le K(\lambda_y |_{\F_{t+r}}\,\|\,\mu|_{\F_{t+r}}) + 1/e.
	\end{align*}
	By Birkhoff's pointwise ergodic theorem,
	$$
	    \lambda\paran{\set{(x,y) \in \cX \times \cY \,: M_t(x,y) \in U} } \convergeto 1.
	$$
	Thus, there exists a set $E' \subset \cY$ of full $\nu$-measure
	such that if $y \in E'$,
	\[
		\lambda_y(M_t(\cdot,y) \in A_t(y)) = \lambda_y(M_t(\cdot,y) \in U)  \convergeto 1.
	\]
	Combining the above and \eqref{eq:y-entropy}, we have for $y \in E \cap E'$,
	\begin{align*}
		\liminf_{t \to \infty} \frac{1}{t} \log 
		\mu\left(M_t(\cdot,y) \in U \right) 
		\ge - \lim_{t \to \infty} \frac{1}{t} K(\lambda_y |_{\F_t}\,\|\,\mu|_{\F_t}) 
		= -h(\lambda\,\|\,\mu \otimes \nu).
	\end{align*}
\end{proof}

\begin{proof}[Proof of Proposition \ref{hm:lower-bound}]
	Note that by Proposition \ref{prop:affine},
	the expression inside the supremum of the right hand side is affine in $\lambda$.
	By Lemma \ref{lem:affine} and Proposition \ref{prop:affine2}, it suffices to prove
	for $\nu$-a.e. $y \in \cY$,
	\begin{equation}\label{eq:lower-bound}
	    \liminf_{t\to\infty} \frac{1}{t}\log
    	\int_\cX \exp\paran{f^t(x,y)}d\mu(x) 
    	\ge \int f\,d\lambda - h(\lambda\,\|\,\mu\otimes\nu),
	\end{equation}
	where $\lambda \in \J_e(\xmap :\nu)$ belongs to a maximizing sequence of the right
	hand side.
	Fix $\epsilon > 0$. Define the set
	\[
		U = \set{\eta \in \M_1(\cX \times \cY)  :  \int f\, d\eta  
		> \int f\,d\lambda - \epsilon},
	\]
	which is an open neighborhood of $\lambda$. Then by Lemma \ref{lem:open},
	for $\nu$-a.e. $y \in \cY$,
	\begin{align*}
		\liminf_{t\to\infty}  \frac{1}{t} & \log
		 \int_\cX \exp\paran{f^t(x, y)}d\mu(x) 
		\\
		& \ge \liminf_{t\to\infty} \frac{1}{t}\log
		\int_{\set{M_t(\cdot, y) \in U}} \exp\paran{f^t(x,y)}d\mu(x)
		\\
		& \ge \int f\,d\lambda - \epsilon  
		+ \liminf_{t\to\infty} \frac{1}{t}\log \mu(M_t(\cdot, y) \in U) 
		\\
		& \ge \int f\,d\lambda - \epsilon - h(\lambda\,\|\,\mu \otimes \nu).
	\end{align*}
	We can take a countable sequence of $\epsilon \to 0$ so that
	we still have a set of full $\nu$-measure of such $y \in \cY$ for 
	\eqref{eq:lower-bound}.
\end{proof}

{ \subsection{The exponentially continuous family of hypermixing processes}}
\noindent Assume that we have a model family $\set{\mu_\theta}_{\theta \in \Theta}$ of
hypermixing processes, and we define the function $V: \Theta \to \real$ 
by \eqref{eq:v}. We have shown that given a fixed $\theta \in \Theta$,
it holds for $\nu$-a.e. $y \in \cY$,
\begin{equation} \label{eq:var-single-3}
    \lim_{t \to \infty} \frac{1}{t} \log \int_\cX 
            \exp(-L_{\theta}^t(x,y))\,d\mu_{\theta}(x) = - V(\theta),
\end{equation}
as long as $L_\theta \in C_b(\cX \times \cY)$, by Proposition \ref{hm:upper-bound} 
and Proposition \ref{hm:lower-bound}. 
Now we introduce a sufficient condition on the family of hypermixing processes
for posterior consistency.
\begin{definition} \label{def:regfam}
    A family $\set{\mu_\theta}_{\theta \in \Theta}$ of
    hypermixing processes is called \textbf{regular} if
    \begin{enumerate}
        \item the map $\theta \mapsto \mu_\theta$ is continuous, and
        
        \item  if $(\theta_t)_{t \in \T}$ is a sequence in $\Theta$ with 
	$\theta_t \convergeto \theta$, then 
	there exists $\ell_0 > 0$, 
	and non-increasing $\alpha_t: [\ell_0, \infty) \to [1, \infty)$ 
	which corresponds to the $\alpha$ in \eqref{cond:H-1} 
	for $\mu_{\theta_t}$ 
	such that $\sup_{t \in \T} \alpha_t(\ell_0) < \infty$.
    \end{enumerate}
\end{definition}

The hypermixing property $\eqref{cond:H-1}$ can be 
viewed as strong decay of correlation between the values of 
the process at distant times, where
the coefficient $\alpha$ measures the speed of decorrelation.
Intuitively, the second condition
of Definition \ref{def:regfam} says that any family of hypermixing processes
 $\set{\mu_{\theta_t}}$ indexed 
by a converging sequence of parameters $(\theta_t)$ do not have arbitrarily 
slow decay of correlation. For many Markov processes,
one can often compute explicitly the mixing 
coefficients $\alpha$
(see \cite[Chapter VI]{deuschel2001large}), 
so this condition can be verified. 
We will provide such examples at the end of this section.

The main result of this subsection is the following proposition, 
which is the final step of proving posterior consistency on 
a regular family of hypermixing processes.
\begin{proposition} \label{hm:exp-cont}
    Suppose $L$ is a loss function satisfying Assumption \ref{assumption:loss},
    and $\set{\mu_\theta}_{\theta \in \Theta}$ is a regular family
    of hypermixing processes.
    Then $\set{\mu_\theta}_{\theta \in \Theta}$ is an exponentially
    continuous family with respect to the loss function $L$.
\end{proposition}

As mentioned in Section \ref{sec:exp-cont-loss}, 
the following lemma is a useful intermediate step for proving exponential continuity, the proof adapts an exponential approximation argument used in the proof of \cite[Lemma 2.5]{wu2004large}.

\begin{lemma} \label{lem:hm-exp-cont}
    Suppose $L$ is a loss function satisfying Assumption \ref{assumption:loss}
    and $\set{\mu_\theta}_{\theta \in \Theta}$ is a regular family
    of hypermixing processes.
	If $\theta \in \Theta$, then it holds that for every $y \in \cY$ that 
	if $(\theta_t)_{t \in \T}$ is a sequence in $\Theta$ such that 
	$\theta_t \convergeto \theta$, 
	then there exists a probability space $(\Omega, \F, \mathbb{P})$ 
	such that for all $\epsilon > 0$, 
	\begin{equation} \label{eq:exp-approx}
	    \limsup_{t \to \infty} \frac{1}{t} 
		\log\prob{\abs{\frac{1}{t}(L^t_\theta(X^{\theta_t}, y) 
		- L^t_\theta(X^\theta, y))} > \epsilon} = -\infty,
	\end{equation}
	which implies, for $\nu$-a.e. $y \in \cY$,
	\begin{equation} \label{eq:exp-approx-2}
	    \lim_{t \to \infty} \frac{1}{t} \log \int_\cX 
            \exp(-L_{\theta}^t(x,y))\,d\mu_{\theta_t}(x) = - V(\theta).
	\end{equation}
\end{lemma}
\begin{proof} 
    Fix $\theta \in  \Theta$. Let $\{\theta_t\}_{t \in \T}$ be a sequence in $\Theta$ such that 
	$\theta_t \convergeto \theta$.  For ease of notation, we write $X^t := X^{\theta_t}$, $X := X^\theta$, $f : = - L_\theta$, and
	$$d_t(X^s, X, y) := \abs{\frac{1}{t}(f^t(X^{s}, y) 
		- f^t(X, y))}.$$
    
    We start with proving \eqref{eq:exp-approx}. 
    By Skorohod's Representation Theorem, there exists a probability space $(\Omega, \F, \mathbb{P})$
	on which $(X^{t}, X)$ has law $\mu_{\theta_t} \otimes \mu_\theta$ and
	$X^{t} \convergeto X$ as $t \to \infty$, $\mathbb{P}$-almost surely.
    Fix $\epsilon > 0$ and $N > 0$. 
    % Note that by Assumption \ref{assumption:loss},
    % \[
    %     d_t(X^t, X, y) \le \frac{1}{t}\int_0^t \omega(d_\X(X^s, X))\,ds
    % \]
    By the Markov-Chebyshev inequality,
	\begin{align*}
		\limsup_{t\to\infty}\,  & \frac{1}{t}\log
			\prob{d_t(X^t, X, y)> \epsilon} \\
% 		& \le 	\limsup_{t\to\infty}  \frac{1}{t}\log
% 			\prob{\frac{1}{t}\int_0^t \omega(d_\X(X^s, X))\,ds > \epsilon}
% 		 \\ 
		& \le -N\epsilon + \limsup_{t \to \infty} 
		  \frac{1}{t}\log\expect{\exp\paran{N t d_t(X^t, X, y)}}.
	\end{align*}
	It suffices to show 
% 	$$\limsup_{t \to \infty} 
% 		  \frac{1}{t}\log\expect{\exp\paran{M \int_0^t \omega(d_\X(X^s, X))\,ds}} = 0.$$
	$$
	    \limsup_{t \to \infty} 
		  \frac{1}{t}\log\expect{\exp\paran{N t d_t(X^t, X, y)}} = 0.
	$$
	By Assumption \ref{assumption:loss}, $f \in C_{b, loc}(\cX \times \cY)$. 
	Also, it holds that $(X^{\theta_t}_s, X_s)_{s \in \T}$ is hypermixing on
	$(\cX^2, \xmap \times \xmap) \simeq (C(\T,\S^2), \xmap \times \xmap)$, for every $t \in \T$.
	As in the proof of Theorem \ref{hm:pressure}, we write
	\[
	    p_s(N, t,y) = \frac{1}{s} \log \expect{\exp\paran{N s d_s(X^t, X, y)}},
	\]
	and following the same argument of Lemma \ref{lem:rsum}, 
	we have for fixed $t > s$, there exists $\ell > 0$ such that
	\[
	    p_t(N, t,y) \le A_{s,t} + \frac{1}{t}\frac{s}{s + r + \ell} 
	        \int_0^{t - s} p_s(C N, t, \ymap^u y)\,du,
	\]
	where the constant $C$ comes from Definition \ref{def:regfam}, and
	$\lim_{s\to\infty} \lim_{t \to \infty} A_{s,t} = 0$.
	By Assumption \ref{assumption:loss}, 
	\[
	   \sup_y s\, d_s(X^t, X, y) \le \int_0^s (\omega(d_\Theta(\theta_t, \theta)) 
	    + \omega(d_\cX(\xmap^r X^t, \xmap^r X))) \,dr
	\]
	which goes to zero as $t \to \infty$. Therefore, by Lebesgue dominated convergence,
	$$\sup_u p_s(C N, t, \ymap^u y) \convergeto 0$$ as $t \to \infty$.
	Then by taking $t \to \infty$ and $s \to \infty$ successively, we have
	\[
	    \limsup_{t \to \infty} p_t(N, t,y) = 0
	\]
	as desired.
	
    The proof of \eqref{eq:exp-approx-2} follows 
    the same argument of \cite[Theorem 1.17]{budhiraja2019analysis}, 
    so we defer it to Section \ref{append:hm}.
\end{proof}
\begin{remark}
	When $\set{\mu_\theta}_{\theta \in \Theta}$ is 
	a family of i.i.d. processes, Condition 1 of 
	Definition \ref{def:regfam} is already enough
	for proving exponential continuity with
	a similar but simpler argument
	as done in \cite{wu2004large}.
\end{remark}

\begin{proof}[Proof of Proposition \ref{hm:exp-cont}]
    Fix $\theta \in  \Theta$. Let $\{\theta_t\}_{t \in \T}$ be a sequence in $\Theta$ such that 
	$\theta_t \convergeto \theta$. By Lemma \ref{lem:hm-exp-cont},
    take $y \in \cY$ from a set of full $\nu$-measure such that \eqref{eq:var-single-3} and
    \eqref{eq:exp-approx-2} hold. Then
    \begin{align*}
        \limsup_{t \to \infty} & \,\frac{1}{t}  \log \int_\cX 
            \exp(-L_{\theta_t}^t(x,y))\,d\mu_{\theta_t}(x) \\
        & \le \limsup_{t \to \infty} \frac{1}{t} \log \int_\cX 
            \exp(-L_{\theta}^t(x,y) + t\omega(d_\Theta(\theta_t, \theta)))
            \,d\mu_{\theta_t}(x) \\
        & = \limsup_{t \to \infty} \omega(d_\Theta(\theta_t, \theta))
            + \frac{1}{t} \log \int_\cX 
            \exp(-L_{\theta}^t(x,y))\,d\mu_{\theta_t}(x) \\
        & = - V(\theta).
    \end{align*}
    The other side of inequality follows from replacing $\omega$ with $-\omega$ 
    and limsup with liminf.
\end{proof}

{\subsection{Concrete example: Markov processes with bounded log-Sobolev constants and beyond
}}
\label{sec:Markov}
\noindent In this section, we apply our framework to (Feller) Markov processes with
bounded log-Sobolev constants. 
We will not focus on the technical details, such as definitions of classical terms and domains of operators.
One can consult \cite{bakry2013analysis} and its references for a rigorous treatment.
Let $(X_t)_{t \in \T}$ be an ergodic (Feller) Markov process on $\S$ with the corresponding Markov semigroup $(P_t)_{t \in \T}$ and
stationary distribution $\mu_0 \in \M_1(\S)$.
In particular, we can define its infinitesmal generator $\mathcal{L}$.
The corresponding carr\'e du champ associated to $\mathcal{L}$ is
given by 
\[
	\Gamma(f, g) := \frac{1}{2}\left(\cL(fg) - f \cL g - g \cL f \right).
\]
\begin{definition}[log-Sobolev inequality]
Given the setting above, we say that $\mu_0$ satisfies the log-Sobolev inequality
with constant $C_{LS}$ if for any $\lambda \in \M_1(\S)$,
\[
	K(\lambda \, \| \, \mu_0) \le \frac{C_{LS}}{2}I(\lambda \, \| \mu_0)
\]
where $I(\lambda \, \|\, \mu_0)$ is called the Fisher information
defined by
\[
	I(\lambda \, \| \, \mu_0) = 
	\begin{cases}
		\int_\S \frac{\Gamma(f)}{f} \, d\mu_0 & \text{ if } \lambda \ll \mu_0 \text{ and } f = \frac{d\lambda}{d\mu_0}, \\ 
		\infty, & \text{ otherwise.}
	\end{cases}
\]
\end{definition}
\begin{remark}
	See \cite{wu2000uniformly} for a large deviation principle on empirical measures
	of Markov processes where the rate function is given by a constant multiple of the 
	Fisher information, which is also known as the Dirichlet form.
	We also refer to \cite{holley1987logarithmic, stroock1992equivalence} for similar connections between the log-Sobolev inequality, mixing, and large deviations 
	on Ising-type models.
\end{remark}

The following theorem comes from \cite[Theorem 5.5.17, Theorem 6.1.3 and Corollary 6.1.17]{deuschel2001large}.
\begin{theorem}[log-Sobolev inequality, hypercontractivity, and hypermixing \cite{deuschel2001large}] \label{thm:LS_equiv}
	Suppose $\mu_0$ is $(P_t)$-reversible. The following are equivalent:
	\begin{enumerate}
		\item $\mu_0$ satisfies the log-Sobolev inequality with constant $C_{LS}$.
		\item $P_t$ is $\mu_0$-hyperconctrative, i.e., for all $f \in L^p(\mu_0)$, we have
		\[
			\| P_t f\|_{L^q(\mu_0)} \le \|f\|_{L^p(\mu_0)}
		\]
		for $t > 0$ and $1 < p \le q < \infty$ such that $e^{2t / C_{LS}} \ge (q - 1) / (p - 1).$
		In particular, 
		\[
			\| P_t f\|_{L^4(\mu_0)} \le \|f\|_{L^2(\mu_0)}
		\]
		for $t \ge T_0$, where $T_0 := \frac{C_{LS}\log 3}{2}$.
		\item Let $\mu \in \M_1(\cX, \xmap)$ be the corresponding $\Phi$-invariant measure
		on path space $\cX$ given initial distribution $\mu_0$ and Markov semigroup $(P_t)$. It holds that $\mu$ is hypermixing
		with $\ell_0 = 12\, T_0$ and $\alpha: [\ell_0, \infty) \to [1, \infty)$ given by
		\[
			\alpha(\ell) = \frac{(1 + \exp(\alpha_0 \ell)) (1 + \exp(-\alpha_0 \ell))}{\exp(\alpha_0 \ell) - \exp(-\alpha_0 \ell)}
		\]
		where $\alpha_0 := \frac{\log(3 / 2)}{8\,T_0}$. Other coefficients can also be explicitly
		computed as well.
	\end{enumerate}
	In particular, the equivalence between (2) and (3) does not
	require $\mu_0$ to be $(P_t)$-reversible.
\end{theorem}

Now we may consider a parameterized family of Markov processes $\{X^\theta\}_{\theta \in \Theta}$ with
correponding Markov semigroup $\{(P^\theta_t)\}_{\theta \in \Theta}$ and starting at stationary distribution $\{\mu_{0, \theta}\}_{\theta \in \Theta}$.
Since we have explicit dependence of the mixing coefficient $\alpha$ on the
log-Sobolev constant $C_{LS}$, we have the following straightforward conclusion.

\begin{proposition}\label{prop:log-Sol}
	$\{X^\theta\}_{\theta \in \Theta}$ is a regular family of hypermixing processes 
	(Definition \ref{def:regfam}) if
	\begin{enumerate}
		\item Each $\mu_{0, \theta}$ satisfies the log-Sobolev inequality with constant $C_{LS}(\theta)$. 
		\item If $\theta_s \convergeto \theta$, then $\mu_{0, \theta_s} \convergeto \mu_{0, \theta}$ weakly, 
		and $P_t^{\theta_s}f \convergeto P_t^\theta f$ in sup norm
		for all $t \ge 0$ and $f \in C_0(\S)$, continuous functions vanishing at infinity. 
		\item If $\theta_s \convergeto \theta$, then $\sup_s C_{LS}(\theta_s) < \infty$.
	\end{enumerate}
\end{proposition}
\begin{proof}
	Condition 1 implies that each $X^\theta$ is hypermixing. Condition 2 implies 
	that if $\theta_t \convergeto \theta$, then $X^{\theta_t} \convergeto X^\theta$
	weakly by classical results on convergence of (Feller) Markov pocesses (see \cite[Theorem 17.25]{kallenberg1997foundations}).
	Condition 3 verifies the boundedness condition of mixing coefficient $\alpha_\theta$.
\end{proof}

\begin{example}[overdamped Langevin dynamics with convex potentials]
	\label{ex:Langevin_convex}
	Suppose we want to perform generalized Bayesian inference on a
	set of parametrized potentials $\{W_\theta\}_{\theta \in \Theta}$ on $\R^d$.
	The Markov processes $\{X^\theta\}$ are given by the corresponding 
	overdamped Langevin stochastic differential equation 
	with unique stationary initial distributions $\mu_{0, \theta} \propto e^{-W_\theta}$:
	\[
		dX^\theta_t = - \nabla W_\theta(X^\theta)\,ds + \sqrt{2} \, dB_t.
	\]
	The corresponding infinitesmal generator $\cL^\theta$ and carr\'e du champ $\Gamma^\theta$ are given by
	\[
		\cL^\theta f = \Delta f - \langle{\nabla W_\theta, \nabla f}\rangle_{\real^d}, 
		\quad \Gamma^\theta(f, g) = \langle{\nabla f, \nabla g}\rangle_{\real^d}.
	\]
	If $\mu_{0, \theta}$ is $\rho_\theta$-strongly log-concave 
	(or $W_\theta$ is $\rho_\theta$-strongly convex), then 
	it satisfies the log-Sobolev inequality with constant $C_{LS}(\theta) = 1 / \rho_\theta$,
	which follows from Bakry-\'Emery curvature-dimension criterion $CD(\rho_\theta, \infty)$
	(see \cite[Section 5.7]{bakry2013analysis}).
	Hence, Condition 1 of Proposition \ref{prop:log-Sol} is satisfied.
	Mild continuity conditions on $\theta \mapsto W_\theta$ will
	verify Condition 2 of Proposition \ref{prop:log-Sol}, for example,
	by continuous dependence
	of solutions of parabolic PDEs defined by $\cL^\theta$ on the parameter $\theta$.
	Similar to above, as long as we have a lower bound $\inf_t \rho_{\theta_t} \ge \rho > 0$ 
	for any converging $(\theta_t)$, Condition 3 of Proposition \ref{prop:log-Sol}
	is satisfied. We have posterior consistency.
\end{example}

The previous example can be generalized in at least two directions. 
The first direction is to consider the Bakry-\'Emery curvature-dimension criterion. 
Indeed, we can easily generalize the same
setting of Langevin dynamics from Euclidean spaces to Riemannian manifolds 
\cite[Section 5.7]{bakry2013analysis}.
The second direction is to consider estimating log-Sobolev constants 
for more general potentials beyond the strongly convex case. 
Recent work such as \cite{menz2014poincare, ma2019sampling} 
provide explicit estimates on log-Sobolev constants when
potentials are locally nonconvex or have complex energy landscapes. 
We will outline how to extend Example \ref{ex:Langevin_convex}
to the case of locally nonconvex potentials described in \cite{ma2019sampling}.
\begin{definition}[local nonconvexity \cite{ma2019sampling}]
	$W: \real^d \to \real$ is $(\ell, \rho, r)$-locally nonconvex if:
	\begin{enumerate}
		\item $\nabla W$ is $\ell$-Lipschitz continuous, and the Hessian $\nabla^2 W$ exists for all points.
		\item W is $\rho$-strongly convex outside $B(0, r)$ the Euclidean ball of radius $r$.
		\item W has a local extremum.
	\end{enumerate}
\end{definition}

One of the main result in \cite{ma2019sampling} is the following
computation of the log-Sobolev constant for $\mu_0 \propto e^{-W}$
with locally nonconvex potential $W$, which follows from
approximating $W$ by a strongly convex potential
and using the Holley-Stroock perturbation principle \cite{holley1987logarithmic}.
\begin{proposition}[\cite{ma2019sampling}]
	Let $W$ be a $(\ell, \rho, r)$-locally nonconvex potential .
	For $\mu_0 \propto e^{-W}$, $\mu_0$ satisfies
	a log-Sobolev inequality with constant $\displaystyle C_{LS} = \frac{2}{\rho}e^{16\, \ell\, r^2}.$
\end{proposition}

\begin{example}[overdamped Langevin dynamics with locally nonconvex potentials]
	We employ the same setting as Example \ref{ex:Langevin_convex}.
	Now we have a parametrized family of potentials $\{W_\theta\}_{\theta \in \Theta}$,
	where $W_\theta$ is $(\ell_\theta, \rho_\theta, r_\theta)$-locally nonconvex.
	Again, based on Proposition \ref{prop:log-Sol}, 
	the parametrized family of Markov processes is regular,
	which implies posterior consistency, if
	for any converging $(\theta_t)$, we have
	$\sup_t \ell_{\theta_t} < \infty$, 
	$\sup_t r_{\theta_t} < \infty$,
	and $\inf_t \rho_{\theta_t} > 0$. 

	The assumption of local nonconvexity is standard in
	the literature of sampling from nonconvex potentials. See \cite{cheng2018sharp}
	and references therein.
	 Local nonconvexity holds for many mixture models with strongly log-concave priors,
	as mentioned in \cite{ma2019sampling}. 
	For example, if $W$ is the negative log-likelihood of a mixture of Gaussians, 
	$W$ is locally nonconvex.
\end{example}

Basically, in the case of Markov processes, we can almost use
our framework as a black box. Whenever there are new explicit
estimates for the log-Sobolev constants and related conditions,
one can obtain new posterior consistency results.
Such estimates are available for other types of Markov processes,
such as interacting particle systems, like 
stochastic Ising models \cite{holley1987logarithmic} and 
McKean-Vlasov dynamics \cite{guillin2022uniform}. 

It is well-known that many non-reversible Markov processes
do not satisfy the log-Sobolev inequality. Still,
some of them can be proved to be hypercontractive, which
allows us to obtain conditions for posterior consistency.
We sketch the argument for two classes of examples below. 

\begin{example}[Beyond log-Sobolev inequality: underdamped Langevin dynamics]
Since our primary example above is overdamped Langevin dynamics,
it is natural to consider the case of underdamped Langevin dynamics.
These are degenerate, hypoelliptic SDEs of the following form:
\[
	\begin{cases}
		dX_t = V_t \, dt, \\
		dV_t = -\nabla W(X_t)\,dt - V_t \,dt + \sqrt{2}\,dB_t,
	\end{cases}
\]
also known as examples of stochastic Hamiltonian systems.
It is easy to check that these Markov processes are non-reversible, and the log-Sobolev inequality does not 
hold (see \cite{wang2017hypercontractivity}).
However, hypercontractivity still holds under natural assumptions, 
which is proved in \cite{wang2017hypercontractivity}
by leveraging other tools: Harnack inequality, coupling, and
concentration. Hence, these are also hypermixing processes by
Theorem \ref{thm:LS_equiv}.
Unfortunately, the constant for hypercontractivity is harder to keep track
in the proof. 
Nevertheless, we believe that if we rigorously follow the
constants line by line, we can work out the explicit dependence on  parameters of potential $W$, hence providing explicit 
conditions for exponential continuity.
\end{example}

\begin{example}[Beyond log-Sobolev inequality: stochastic delay equations]
Another natural class of 
non-reversible Markov processes
do not satisfies the log-Sobolev inequality is
called stochastic delay equations (see \cite{bao2015hypercontractivity1}
for the explanation). These are SDEs of the form:
\[
	dX_t = b(X_t)\,dt + \sigma(X_t)\,dB_t
\]
where the state space is the space of segments of paths, i.e., 
$X_t \in C_b([-\epsilon, 0], \real^d)$, 
which is also the domain of $b, \sigma$. Here $\epsilon$ is a delay
parameter. 
We believe that these are $\epsilon$-Markov processes
as mentioned at the begining of the section
and introduced in \cite{chiyonobu1988large}.
Although the log-Sobolev inequality is not applicable,
hypercontractivity still holds, which is proved in 
\cite{bao2015hypercontractivity1} and also \cite{bao2015hypercontractivity2}
for the infinite-dimensional case, i.e., SPDEs,
by using the same tools as 
in the case of stochastic Hamiltonian systems above.
Again, the constant for hypercontractivity is also harder to track, but still it seems possible to work out explicitly the dependence
on the parameters.
\end{example}

\section{Example: Gibbs processes on shifts of finite type}
\label{sec:ex_Gibbs}
We proved posterior consistency for some families of continuous-time stochastic processes. We now show that the same framework can be used to prove posterior consistency of Gibbs processes on shifts of finite type (SFT). These are discrete-time, discrete-state dynamical systems that include finite-state 
Markov chains with arbitrary order of dependencies. As posterior consistency was already proven in \cite{mcgoff2019gibbs}, we will be concise on some of the technical details. Our motivation for including this example is to highlight the flexibility of the large deviation framework that can be used to prove posterior consistency.

Consider $\T = \nat$ and $\S$ is a finite alphabet endowed with the discrete topology.
We endow $\cX = \S^\nat$ with the product topology and a compatible metric $d_\cX$,
for example, $d_\cX(x, x') = 2^{-n(x,x')}$ where $n(x, x')$ is the infimum of those $m$
such that $x_m \neq x'_m$.

\begin{definition}
\label{def:Gibbs}
Given a H\"older continuous potential function $\phi: \cX \to \real$, a Borel probability measure $\mu \in \M_1(\cX, \xmap)$ is the corresponding Gibbs measure, if the following Gibbs property is satisfied:
there exists constants $\cPP \in \real$ and $N > 0$ s.t. 
for all $x \in \cX$ and $t \ge 1$,
\begin{equation} \label{eq:Gibbs}
    N^{-1} \le \frac{\mu([x_0, x_1, \ldots, x_{t-1}])}{\exp \paran{-\cPP t + \phi^t(x)}} \le N,
\end{equation}
where we write 
    $[w] = \set{\bar{x} \in \cX \,;\, \bar{x}_i = w_i, \, i = 0, \ldots, t - 1}$
for any $w \in \S^{t}$. 
We call the process $(X_t)_{t \in \T}$ with law $\mu$ the Gibbs process
corresponding to $\phi$.
\end{definition}

See \cite{bowen1975equilibrium} for details of Gibbs processes. 
For notational convenience, we focus on the case of the
one-sided full shift, but one may also prove the same 
result for two-sided mixing subshift as in \cite{mcgoff2019gibbs}.
\begin{example}
We refer back to \cite{mcgoff2019gibbs} for concrete examples 
of Gibbs processes 
which include finite state 
Markov chains with arbitrary order of dependencies
and models from statistical physics, such as Ising models, Potts models,
and nearest neighbor spin systems. 
\end{example}

\subsection{Large deviation asymptotics for one Gibbs process}
Recall that the first step is to prove \eqref{eq:var-single} for a
single Gibbs measure $\mu$. Unlike the previous example,
instead of directly proving the needed large deviation behavior 
for $\mu$, we reduce the problem to the case of product measures,
using ideas of exponential tilting.

Let $\sigma_0 \in \M_1(\S)$ be the uniform probability measure 
on the finite alphabet $\S$, and let
 $\sigma = \sigma_0^{\otimes \T} \in \M_1(\cX)$ be the corresponding 
product measure on $\cX$, which is hypermixing.
\begin{lemma} \label{lem:GibbsToCount}
    If $E \in \F_{t-1}$, then
    \[
        N^{-1} \le \frac{\mu(E)}{|S|^t \int_E \exp(-\cPP t + \phi^t(x))d\sigma(x)} \le N.
    \]
\end{lemma}
\begin{proof}
    Let $E \in \F_{t-1}$. Then $\one_E$ only depends on the first $t$ coordinates, so
    by conditioning on the first $t$ coordinates and the Gibbs property \eqref{eq:Gibbs},
    \begin{align*}
        \mu(E) & = \sum_{w \in \S^t} \mu([w]) \one_E(w) \\
        & \le N \sum_{w \in \S^t} 
            \min_{x \in [w]} \one_E(x) \exp(-\cPP t + \phi^t(x)) \\
        & \le  N \sum_{w \in \S^t} \frac{1}{\sigma([w])}  \int_{[w]} \one_E(x) \exp(-\cPP t + \phi^t(x))\,d\sigma(x) \\
        & = N |\S|^t \int_E \exp(-\cPP t + \phi^t(x))\,d\sigma(x).
    \end{align*}
    The other side of inequality is proved similarly by replacing $\min$ with $\max$.
\end{proof}
For convenience, we write
$\inprod{\eta}{f} := \int f\,d\eta$ for $\eta \in \M_1(\cX \times \cY)$ 
and $f \in C_{b,loc}(\cX \times \cY)$.
Following the same idea in \cite[Chapter 6]{rassoul2015course},
we denote the periodic empirical process by
$\widetilde{M}_t(x,y) := M_t(\widetilde{x}^{(t)}, y)$,
where $\widetilde{x}^{(t)} \in \cX$
is given by $\widetilde{x}^{(t)}_s = x_s$ for $0 \le s \le t - 1$
and $\widetilde{x}^{(t)}_s = \widetilde{x}^{(t)}_r$ if $s = r \mod t$.
For any $f \in C_{b, loc}(\cX \times \cY)$ and any $y \in \cY$,
\begin{equation} \label{eq:approxByPeri}
   \lim_{t \to \infty} \sup_{x \in \cX} \abs{\inprod{{M}_t(x,y)}{f} -  \inprod{\widetilde{M}_t(x,y)}{f}} = 0.
\end{equation}
The key point of introducing $\widetilde{M}_t(x,y)$ is that $\one_{\set{\widetilde{M}_t(\cdot,y) \in E}}$ is $\F_{t-1}$-measurable 
for any Borel set $E \subset \M_1(\cX \times \cY)$. 
Also, when $X \sim \sigma$, using \eqref{eq:approxByPeri},
we can prove the same large deviation result for $\widetilde{M}_t(X,y)$
as that for ${M}_t(X,y)$ in the previous section.
Hence, we can reduce the problem of proving \eqref{eq:var-single} of the Gibbs measure $\mu$
to the product measure $\sigma$ by Lemma \ref{lem:GibbsToCount}:
\begin{align*}
     \lim_{t \to \infty} &\, \frac{1}{t} \log \int_\cX \exp({f^t}(x,y)) \,d\mu(x) \\
    & = \lim_{t \to \infty}  \frac{1}{t} \log \int_{\cX} 
    \exp\paran{t\inprod{{M}_t(x,y)}{f}} d{\mu}(x) \\ 
    & = \lim_{t \to \infty}  \frac{1}{t} \log \int_{\cX} \exp\paran{t\inprod{\widetilde{M}_t(x,y)}{f}} d{\mu}(x) \\ 
    & = \log|\S| - \cPP + \lim_{t \to \infty} \frac{1}{t} \log \int_{\cX} \exp\paran{t\inprod{\widetilde{M}_t(x,y)}{f + \phi}} d{\sigma}(x) \\
    & = \log|\S| - \cPP + \sup_{\lambda \in \J(\xmap : \nu)}\set{\int f + \phi \,d\lambda - h(\lambda\,\|\,\sigma \otimes\nu)}
\end{align*}
for $\nu$-a.e. $y \in \cY$. 
One can check, as Example \ref{eg:SFT}, that $h(\lambda\,\|\,\sigma \otimes\nu) = \log |\S| - h^\nu(\lambda)$, where $h^\nu(\lambda)$ is the fibre entropy (see \cite{kifer2006random}) of
$\lambda$ with respect to $\nu$. Therefore, we recover that for any $f \in C_{b,loc}(\cX \times \cY)$,
it holds that for $\nu$-a.e. $y \in \cY$,
\begin{align*}
    \lim_{t \to \infty} & \frac{1}{t} \log \int_\cX \exp({f^t}(x,y)) \,d\mu(x)
    = \sup_{\lambda \in \J(\xmap : \nu)}\set{\int f \,d\lambda - \cPP + \int \phi \,d\lambda + h^\nu(\lambda)}.
\end{align*}
Although the above is already sufficient for our Assumption \ref{assumption:loss} on loss functions, we note that with a little more work we can recover the above equality for $f \in C_b(\cX \times \cY)$.
One can also prove that for $\lambda \in \J(\xmap : \nu)$,
\[
    h(\lambda \,\|\, \mu \otimes \nu) = \cPP - \int \phi \,d\lambda - h^\nu(\lambda)
\]
by following similar ideas in \cite{chazottes1998relative}
and \cite[Lemma 7]{mcgoff2019gibbs}.

\begin{remark}[Gibbs processes on general state spaces]
    One can establish the same convergence result for Gibbs processes on
    general state spaces by a similar exponential tilting argument 
    (see \cite[Chapter 8]{rassoul2015course}),
    but the exponential continuity in this case may not be 
    as straightforward.
\end{remark}

{ \subsection{The exponentially continuous family of Gibbs processes}}
\noindent We first recall the definition of a regular model family 
$\set{\mu_\theta}_{\theta \in \Theta}$ in \cite{mcgoff2019gibbs}.
Recall that to obtain the main convergence result \eqref{eq:variational}, 
it suffices to prove exponential continuity (Definition $\ref{def:exp-cont}$) of 
$\set{\mu_\theta}_{\theta \in \Theta}$ with respect to the loss
function $L$.

\begin{definition}
    A family $\set{\phi_\theta}_{\theta \in \Theta}$ of potential functions
    is regular if all the $\phi_\theta: \cX \to \real$ are $r$-H\"older continuous for 
    some $r > 0$ and the map $\theta \mapsto f_\theta$ is continuous 
    with respect to the usual H\"older norm $\|\cdot\|_r$.
    We denote the corresponding regular family of Gibbs measures by 
    $\set{\mu_\theta}_{\theta \in \Theta}$.
\end{definition}
When $\set{\phi_\theta}_{\theta \in \Theta}$ is a regular family, the 
corresponding map $\theta \mapsto \mu_\theta$ is continuous with respect
to the weak topology on measures. Moreover, the maps from $\theta$ to
the corresponding constants $N_\theta$ and $\cPP_\theta$ in \eqref{eq:Gibbs} are also continuous
(see \cite{alves2019equilibrium}).
As a result, by the compactness of $\Theta$, we have
a uniform Gibbs property: there exist a uniform constant $N$ and 
a continuous map $\theta \to \cPP_\theta$ such that 
\begin{equation} \label{eq:uniformGibbs}
    N^{-1} \le \frac{\mu(\set{\bar{x} \in \cX \,;\, \bar{x}_i = x_i, \, i = 0, \ldots, t - 1})}{\exp \paran{-\cPP_\theta t + \phi^t_\theta(x)}} \le N.
\end{equation}

We have shown that given a fixed $\theta \in \Theta$, 
for $\nu$-a.e. $y \in \cY$,
\[
    \lim_{t \to \infty} \frac{1}{t} \log \int_\cX \exp(-L^t_\theta(x,y)) \,d\mu_\theta(x)
    = - V(\theta),
\]
as long as $L_\theta \in C_{b,loc}(\cX \times \cY)$, where
\[
    V(\theta) = \inf_{\lambda \in \J(\xmap : \nu)}\set{\int L_\theta \,d\lambda + \cPP_\theta - \int \phi_\theta \,d\lambda - h^\nu(\lambda)}.
\]
Finally, we show that $\set{\mu_\theta}_{\theta \in \Theta}$ is exponential continuous 
(Definition $\ref{def:exp-cont}$) with respect to the loss
function $L$, by exploiting the uniform Gibbs property \eqref{eq:uniformGibbs}
in basically the same way as in \cite{mcgoff2019gibbs}.
\begin{lemma} \label{lem:lossRatio}
    Suppose $L$ is a loss function satisfying Assumption \ref{assumption:loss}.  Then for any $\epsilon > 0$,
    there exists $T > 0$ such that
    for all $\theta, \theta' \in \Theta$, $y \in \cY$, and $t \ge 1$,
    \[
    \frac{\int_\cX \exp(-L^t_{\theta}(x, y)) \, d\mu_{\theta}(x)}{
       \int_\cX \exp(-L^t_{\theta'}(x', y)) \, d\mu_{\theta'}(x')}
      \le N^2 e^{  t(\omega(d_\Theta(\theta, \theta')) + \epsilon)
        + (t + T)(|\cPP_\theta - \cPP_{\theta'}| + \|\phi_\theta - \phi_{\theta'}\|)}.
    \]
\end{lemma}
\begin{proof}
Let $\theta, \theta' \in \Theta$ and $y \in \cY$. 
Note that by the uniform Gibbs property \eqref{eq:uniformGibbs}, for any $t \ge 1$ and
$w \in \S^{t}$,
\begin{equation} \label{eq:GibbsRatio}
    \frac{\mu_\theta([w])}{\mu_{\theta'}([w])} 
    \le N^2 e^{ t(|\cPP_\theta - \cPP_{\theta'}| + \|\phi_\theta - \phi_{\theta'}\|)}.
\end{equation}
Also, by Assumption \ref{assumption:loss}, we see that
\begin{equation} \label{eq:expLossRatio}
    \frac{\exp(-L^t_{\theta}(x, y))}{\exp(-L^t_{\theta'}(x', y))}
    \le e^{  t(\omega(d_\Theta(\theta, \theta')) + \sum_{s=0}^{t-1} \omega(d_\cX(\xmap^s x, \xmap^s x'))}.
\end{equation}
Take $T > 0$ such that for any $x, x' \in \cX$ with 
$x_i = x'_i$ for $0 \le i \le T-1$, we have $d_\cX(x, x') < \epsilon$. 
Using \eqref{eq:expLossRatio} and \eqref{eq:GibbsRatio}, we obtain
\begin{align*}
    \int_\cX & \exp(-L^t_{\theta}(x, y)) \, d\mu_{\theta}(x) \\
    & = \sum_{w \in \S^{t+T}} \mu_\theta([w])
         \int_{[w]} \exp(-L^t_{\theta}(x, y)) \, d\mu_{\theta}(x\,|\,[w]) \\
    & \le e^{  t(\omega(d_\Theta(\theta, \theta')) + \epsilon)} 
    \sum_{w \in \S^{t+T}} \mu_{\theta}([w]) 
        \int_{[w]} \exp(-L^t_{\theta'}(x', y)) \, d\mu_{\theta'}(x'\,|\,[w]) \\
    & \le N^2 e^{  t(\omega(d_\Theta(\theta, \theta')) + \epsilon) 
        + (t + T)(|\cPP_\theta - \cPP_{\theta'}| + \|\phi_\theta - \phi_{\theta'}\|)} \\
    & \quad \cdot \sum_{w \in \S^{t+T}} \mu_{\theta'}([w]) 
        \int_{[w]} \exp(-L^t_{\theta'}(x', y)) \, d\mu_{\theta'}(x'\,|\,[w]) \\
    & = N^2 e^{  t(\omega(d_\Theta(\theta, \theta')) + \epsilon)
        + (t + T)(|\cPP_\theta - \cPP_{\theta'}| + \|\phi_\theta - \phi_{\theta'}\|)} \\
    & \quad \cdot \int_\cX \exp(-L^t_{\theta'}(x', y)) \, d\mu_{\theta'}(x').
\end{align*}
\end{proof}
\begin{corollary}
    Suppose $L$ is a loss function satisfying Assumption \ref{assumption:loss},
    and $\set{\mu_\theta}_{\theta \in \Theta}$ is a regular family
    of Gibbs processes.
    Then $\set{\mu_\theta}_{\theta \in \Theta}$ is an exponentially
    continuous family with respect to the loss function $L$.
\end{corollary}
\begin{proof}
    Fix $\theta \in \Theta$. Then for $\nu$-a.e. $y \in \cY$,
    \[
        \lim_{t \to \infty} \frac{1}{t} \log \int_\cX \exp(-L^t_\theta(x,y)) \,d\mu_\theta(x)
        = - V(\theta).
    \]
    Let  $y \in \cY$ be as above. By Lemma \ref{lem:lossRatio}, for any $\epsilon > 0$, there exists $T > 0$ such that
    \begin{align*}
       \lim_{t \to \infty} & \abs{\frac{1}{t} \log \int_\cX \exp(-L^t_{\theta_t}(x,y)) \,d\mu_{\theta_t}(x) - (-V(\theta))}  \\
        & \le \lim_{t \to \infty} \frac{1}{t} \abs{  \log \int_\cX \exp(-L^t_{\theta_t}(x,y)) \,d\mu_{\theta_t}(x)
        - \log \int_\cX \exp(-L^t_\theta(x,y)) \,d\mu_\theta(x)} \\
        & \le \epsilon + \lim_{t \to \infty} \left[   \omega(d_\Theta(\theta_t, \theta)) + \frac{t + T}{t} (|\cPP_{\theta_t} - \cPP_{\theta}| 
        + \|\phi_{\theta_t} - \phi_{\theta}\|) \right] = \epsilon.
    \end{align*}
    Hence, we obtain the desired result.
\end{proof}

\section{Discussion}
We believe that our large deviation arguments are far from being
sharp, since we made a few assumptions for the sake of clarity of exposition, but still we at least included two rich classes of examples. Compared to the existing literature on posterior 
consistency, it is straightforward to track which step uses which assumptions.
There are technical conditions in our framework that are natural to relax such as those of the loss function. Our framework starts with Varadhan's Lemma which requires the loss function to be continuous and bounded. It is not unreasonable to consider more general loss functions and work with generalizations of Varadhan's lemma developed in large deviation theory that do not require continuity or boundedness, see \cite[Theorem 1.18 \& 1.20]{budhiraja2019analysis}. 
It is also of interest to explore our framework for discontinuous processes. Our framework is not restricted to continuous processes, since process-level large deviation results 
for discontinuous processes are also known \cite[Theorem 5.4.27]{deuschel2001large}. 
In that case, we let $\cX = D(\T, \S)$ be the Skorohod space equipped
with the Skorohod topology and follow the framework we have outlined.

It is tempting to consider posterior consistency over non-compact parameter spaces. Most work in posterior consistency, including ours, requires compactness of the parameter space. A basic question is whether compactness can be relaxed. In our large deviation framework, posterior consistency is obtained by requiring a process-level large deviation behavior. In \cite{wu2004large}, the author proves that the process-level LDP of a mixture of  
i.i.d. processes $\set{\mu_\theta}_{\theta \in \Theta}$  by the fully supported prior $\pi_0$ is equivalent to 
the compactness of the parameter space $\Theta$. This suggests that from the perspective of large deviation theory,
the assumption on the compactness of $\Theta$ may be difficult to relax.
Still, since posterior consistency does not require 
the full large deviation principle, we believe that one can extend
our framework to non-compact parameter space, for example,
by assuming some exponential decay of the prior distribution.

We showed that a well-studied hypermixing condition
can imply posterior consistency.
Although the hypermixing condition can be defined for
general dynamical systems, the condition is harder
to check beyond Markov processes, because general
dynamical systems do not have access to convenient tools,
like the log-Sobolev inequality and hypercontractivity. 
Our posterior consistency framework may serve as a motivation for proving mixing properties of general dynamical systems, as one can in turn perform Bayesian inference for these systems beyond Markov processes. For example, it is natural to expect
some mixing properties to hold for non-Markovian SDEs
driven by fractional Brownian motions, though there are
not so many results in this direction beyond ergodicity \cite{hairer2005ergodicity, hairer2007ergodic}. On the other hand, there are many dynamical systems we have not covered exhibiting good statistical properties, including weaker forms of large deviation principles, such as nonuniformly hyperbolic systems admitting Young tower extensions with exponential return times \cite{rey2008large}. 
Also, another class of dynamical models we do not consider in this paper arise from algorithms, including recurrent neural networks such as long short-term memory networks. It is of great interest to expand our framework to include those systems as well.

\vspace{0.1cm}
\noindent \textbf{Declaration of interest}: none

\vspace{0.1cm}
\noindent \textbf{Acknowledgements}: Sayan Mukherjee and Langxuan Su would like to thank Jonathan Mattingly, Kevin McGoff, and Andrew Nobel for discussions. 
Langxuan Su would like to thank Andrea Agazzi for suggesting some of the Markov process examples and giving helpful feedback for an earlier version of the manuscript.
Sayan Mukherjee would like to acknowledge partial funding from HFSP RGP005, NSF DMS
17-13012, NSF BCS 1552848, NSF DBI 1661386, NSF IIS 15-46331, NSF DMS 16-13261, as
well as high-performance computing partially supported by grant 2016-IDG-1013 from the
North Carolina Biotechnology Center
as well as the Alexander von Humboldt Foundation, the BMBF and the Saxony State Ministry for Science. 

\appendix
\section{Background material}
\subsection{Probability measures on Polish spaces} \label{sec:prob}
The following results can be found at \cite{glasner2003ergodic}
and \cite[Chapter 4]{rassoul2015course}.

Let $\U$ be a Polish space. A duality pairing of $\M(\U)$
and $C_b(\U)$ is given by $\inprod{\eta}{f} = \int_\U f\,d\eta$ for $\eta \in \M(\U)$
and $f \in C_b(\U)$.
This duality pairing generates a weak topology on $\M(\U)$
 and hence on the closed convex subset $\M_1(\U)$, 
 with a base given by sets of 
 the form
\[
    \set{\eta \in \M(\U)\,:\,\abs{\inprod{\eta}{f_i} - \inprod{\eta_0}{f_i}} < \epsilon
        , \quad i = 1, \ldots, m}
\]
for $\epsilon > 0$, $m \in \nat$, $\eta_0 \in \M(\U)$
and $f_i \in C_b(\U)$, $1 \le i \le m$.
In particular, $\M_1(\U)$ is a Polish space under this weak topology
which also characterizes the weak convergence of probability measures.
In the case when $\U$ is $\cX$ or $\cX \times \cY$, the duality pair 
of $\M(\U)$ and $C_{b, loc}(\U)$ with the same pairing also generates
the same weak topology on $\M(\U)$ and $\M_1(\U)$.
\begin{definition}[Weak convergence]
    A sequence of probability measures $\eta_n \in \M_1(\U)$
    converges weakly to another probability measure $\eta \in \M_1(\U)$ if
    $\int f \,d\eta_n \convergeto \int f\,d\eta$,
    for every $f \in C_b(\U)$.
\end{definition}
\begin{theorem}[Disintegration of measures] \label{thm:disint}
    Let $\U$ and $\cY$ be Polish spaces. Let $\lambda \in \M_1(\U \times \cY)$
    and let $\nu$ be the $\cY$-marginal of $\lambda$. Then 
    there exists a Borel map $\cY \to \M_1(\U)$, $y \mapsto \lambda_y$
    such that for every bounded Borel function $f: \U \times \cY \to \real$,
    \[
        \int_{\U \times \cY} f\,d\lambda = \int_\cY \int_\U f(u, y)\,d\lambda_y(u)\,d\nu(y).
    \]
    In particular, such a map is unique in the sense that
    if $y \mapsto \lambda_y'$ is another such map, then $\lambda_y = \lambda_y'$
    for $\nu$-a.e. $y \in \cY$.
    We also write $\lambda = \int_\cY \lambda_y \otimes \delta_y \,d\nu(y).$
\end{theorem}

\subsection{Properties of relative entropy}
The following results can be found at \cite[Section 2]{budhiraja2019analysis}
and \cite[Exercise 6.11]{rassoul2015course}.
\begin{theorem}[Donsker-Varadhan variational formula] \label{rel:var}
    Let $\U$ be a Polish space equipped with the Borel $\sigma$-algebra, 
    and let $\lambda, \eta \in \M_1(\U)$. Then
    we have the following variational formula of the relative entropy:
    \begin{align*}
        K(\lambda\,\|\,\eta) & =
        \sup_{f \in B_b(\U)} 
        \int_\U f \,d\lambda - \log \int_\U \exp(f)\,d\eta \\
        & = \sup_{f \in C_b(\U)} 
        \int_\U f \,d\lambda - \log \int_\U \exp(f)\,d\eta,
    \end{align*}
    where $B_b(\U)$ denotes the set of bounded measurable
    functions on $\U$.
\end{theorem}

\begin{theorem}[Chain rule] \label{rel:chain}
    Let $\U$ and $\cY$ be Polish spaces, and 
    let $\lambda, \eta \in \M_1(\U \times \cY)$.
    Let $\lambda_\cY$ and $\eta_\cY$ be the $\cY$-marginal of $\lambda$ and $\eta$.
    Then we have the following chain rule on the relative entropy:
    \[
        K(\lambda\,\|\,\eta) = K(\lambda_\cY\,\|\,\eta_\cY) 
            + \int_\U K(\lambda_y\,\|\,\eta_y)\, \lambda_\cY(dy),
    \]
    where $\lambda_y$ and $\eta_y$ are given by the disintegration
    of $\lambda$ and $\eta$ with respect to their $\cY$-marginal, given
    by Theorem \ref{thm:disint}.
\end{theorem}

\begin{theorem}[Monotone convergence] \label{rel:mono}
    Let $\set{\F_t}$ be an increasing family of $\sigma$-algebras 
    and let $\F$ be their limit. Let $\lambda$ and $\eta$ be two probability 
    measures on $\F$. Then the relative entropy
    $K(\lambda|_{\F_t}\,\|\,\eta|_{\F_t})$ is monotonically increasing and
    converges to $K(\lambda|_{\F}\,\|\,\eta|_{\F})$.
\end{theorem}

\subsection{Ergodic theory}
We refer to \cite{petersen1989ergodic, walters2000introduction} for 
general introductions to ergodic theory. 
\subsubsection{Standard results}
The following results can be found in \cite[Section 5.2]{deuschel2001large}.
Let $\Z$ be a Polish space equipped with the Borel $\sigma$-algebra, 
and $R^t: \Z \to \Z$ be an one-parameter family 
of measurable transformations on $\Z$.
\begin{theorem}[Birkhoff's pointwise ergodic theorem] \label{thm:ergodic}
    Let $\eta \in \M_1(\Z, R)$ be an $R$-ergodic measure and $f \in C_b(\Z)$.
    Then there exists a measurable set $E \subset \Z$ with full $\eta$-measure
    such that for all $z \in E$,
    \[ 
        \lim_{t \to \infty} \frac{1}{t}\int_0^t f(R^s z) \,ds
        = \int_\Z f \,d\eta.
    \]
    In fact, $E$ can be taken independently of $f \in C_b(\Z)$. See \cite[Lemma 7.2]{mcgoff2016variational}.
\end{theorem}

\begin{theorem}[The ergodic decomposition]
    Let $\eta \in \M_1(\Z, R)$. There exists a Borel probability
    measure $Q$ on $\M_1(\Z, R)$ such that $Q$ assigns full measure to
    the set of $R$-ergodic measures in $\M_1(\Z, R)$, and
    if $f \in L^1(\eta)$, then $f \in L^1(\rho)$ for 
    $Q$-a.e. $\rho \in \M_1(\Z, R)$, and
    $$ \int_\Z f \,d\eta = \int_{\M_1(\Z, R)} \int_\Z f\,d\rho \,dQ(\rho).$$
    We also write $\eta = \int \rho\, dQ$.
\end{theorem}

\subsubsection{Properties of joint processes}
The following results can be found at \cite{kifer2006random} 
and \cite[Lemma A.2]{mcgoff2016variational}.
Let $\Z$, $\cY$ be Polish spaces equipped with the Borel $\sigma$-algebras. 
Let $R^t: \Z \to \Z$ and $\ymap^t: \cY \to \cY$ be
 one-parameter families of measurable transformations on $\U$ and $\cY$.
Let $\nu \in \M_1(\cY, \ymap)$ be an $\ymap$-ergodic measure.

\begin{proposition} \label{prop:inv-join}
    Let $\lambda \in \J(R : \nu)$ with its disintegration 
    $\lambda = \int_\cY \lambda_y \otimes \delta_y\,d\nu(y)$ over $\nu$,
    given by Theorem \ref{thm:disint}. Then for any $t \ge 0$,
    we have for $\nu$-a.e. $y \in \cY$,
     $\lambda_y \circ (R^t)^{-1} 
     = \lambda_{\ymap^t y},$ 
    and so for every $f \in C_b(\Z)$, we have
    $$ \int_\Z f(R^t z)\,d\lambda_y(z) 
    = \int_\Z f(z)\,d\lambda_{\ymap^t y}(z). $$
\end{proposition}

\begin{theorem}[Structure of the space of joint processes] \label{thm:ergodicDecomp}
    $\J(R : \nu)$ is a non-empty and convex set. 
    If $\lambda \in \J(R : \nu)$ and $\lambda = \int \rho\, dQ$ is
    its ergodic decomposition, then $Q$-a.e. $\rho \in \J(R : \nu).$
\end{theorem}

\begin{lemma}\label{lem:affine}
    Let $F: \J(R : \nu) \to [0, \infty]$ be an affine functional in the sense that
    $F(a\eta + (1 - a)\lambda) = a F(\eta) + (1 - a)F(\lambda)$ for any $a \in (0, 1)$
    and $\eta, \lambda \in \J(R : \nu)$.
    Then for any $\lambda \in \J(R : \nu) $ with its ergodic decomposition
    $\lambda = \int \rho\, dQ$, we have  
    \[
        F(\lambda) = \int_{\J_e(R : \nu)} F(\rho)\,dQ(\rho).
    \]
\end{lemma}
\begin{proof}
    Since $F$ is affine, 
    if $\lambda \in \J(R : \nu)$ 
    and $\lambda = \int \rho\, dQ$ is its ergodic decomposition,
    then by \cite[Lemma 5.4.24]{deuschel2001large} and Theorem \ref{thm:ergodicDecomp},
    \[
        F(\lambda) = \int_{\J_e(R : \nu)} F(\rho)\,dQ(\rho).
    \]
\end{proof}

\begin{proposition}\label{prop:affine2}
    Let $F: \J(R : \nu) \to \real$ be a functional such that
    for any $\lambda \in \J(R : \nu) $ with its ergodic decomposition
    $\lambda = \int \rho\, dQ$, 
    \[
        F(\lambda) = \int_{\J_e(R : \nu)} F(\rho)\,dQ(\rho).
    \]
    Then we have $$\sup_{\lambda \in \J(R : \nu)} F(\lambda)
    = \sup_{\lambda \in \J_e(R : \nu)} F(\lambda). $$
\end{proposition}
\begin{proof}
   Straightforward.
\end{proof}

\subsection{Conditional large deviation lemmas}\label{sec:LD}

\subsubsection{A conditional weak large deviation upper bound}
\begin{proof}[Proof of Lemma \ref{lem:genericUB}]
    We basically reproduce the argument of \cite[Theorem 4.24]{rassoul2015course}
    and only need to make sure the result holds for a subset of $\cY$ with full $\nu$-measure.
    
    Let $\K \subset \M_1(\cX \times \cY)$ be a compact subset. 
    Fix an $\epsilon > 0$ and $c < \inf_\K p^*.$   
    For every $\lambda \in \K$, there exists $f \in C_{b. loc}(\cX \times \cY)$
    such that $\int f \,d\lambda - p(f) > c$.
    Consider the following open neighborhood of $\eta$:
    \[
        B_\epsilon(\lambda, f) = \set{\eta \in \M_1(\cX \times \cY)\,:\,
        \abs{\int f \,d\lambda - \int f \,d\eta} < \epsilon}.
    \]
    Since $\K$ is compact, it can be covered with 
    $B_\epsilon(\lambda_1, f_1), \ldots, B_\epsilon(\lambda_m, f_m)$ with corresponding
    $f_1, \ldots, f_m \in C_{b, loc}(\cX \times \cY)$.
    Then for every $y \in \cY$,
    \begin{align*}
        & \mu(M_t(\cdot, \, y) \in B_\epsilon(\lambda_i, f_i)) \\
        & \quad = 
        \int_{\set{M_t(\cdot, y) \in B_\epsilon(\lambda_i, f_i)}} 
            \exp(-t f^t_i(x,y) + t f^t_i(x,y))\,d\mu(x) \\
        & \quad \le \exp\paran{t\paran{- \int f_i\,d\lambda_i + \epsilon}}
         \int_{\set{M_t(\cdot, y) \in B_\epsilon(\lambda_i, f_i)}} 
            \exp(t f^t_i(x,y))\,d\mu(x).
    \end{align*}
    By Theorem \ref{hm:pressure}, for $\nu$-a.e. $y \in \cY$ independent of $f_i$, 
    hence independent of $\mathcal{K}$,
    \[
        \limsup_{t \to \infty} \frac{1}{t}
            \log\mu(M_t(\cdot, y) \in B_\epsilon(\lambda_i, f_i))
        \le - \int f_i \,d\lambda_i + \epsilon + p(f_i) \le -c + \epsilon.
    \]
    Hence, for $\nu$-a.e. $y \in \cY$ independent of $\mathcal{K}$,
    \begin{align*}
        \limsup_{t \to \infty} \frac{1}{t}
            \log\mu(M_t(\cdot, y) \in \K) 
       & \le \max_{1 \le i \le m}  \limsup_{t \to \infty} \frac{1}{t}
            \log\mu(M_t(\cdot, y) \in B_\epsilon(\lambda_i, f_i)) \\
       & \le -c + \epsilon.
    \end{align*}
    We take $\epsilon \convergeto 0$
    and $c \convergeto \inf_\K p^*$ to yield the desired result.
    % so 
    % \begin{equation} \label{eq:wldpup}
    %      \limsup_{t \to \infty} \frac{1}{t}
    %         \log\mu(M_t(\cdot, y) \in \K) \le \inf_\K p^*
    % \end{equation}
    % holds for
    % a set of full $\nu$-measure of $y \in \Y$ .
    % Moreover, since $\M_1(\X \times \Y)$ is Polish, 
    % it suffices to consider at most countably many of such
    % $B_\epsilon(\lambda, f)$ and $p(f)$, so 
    % the weak large deviation upper bound \eqref{eq:wldpup}
    % holds for a set of full $\nu$-measure of $y \in \Y$.
    
    That $p^*(\lambda) = \infty$ if $\lambda$ is not $(\xmap \times \ymap)$-invariant
    follows from the same proof as in \cite[p. 218]{deuschel2001large}.
    The last statement is also straightforward:
    if the $\cY$-marginal of $\lambda$ is not $\nu$,
    then for any $N > 0$,
    there exists $f \in C_b(\cY)$ such that $\int f \,d\lambda - \int f \,d\nu \ge N$, 
    so by Birkhoff's pointwise ergodic theorem, 
    $p^*(\lambda) \ge \int f \,d\lambda - p(f) \ge N$.
\end{proof}

\subsection{The conditional exponential tightness}
To obtain the full large deviation upper bound, 
we need to show that for $X \sim \mu$ and $\nu$-a.e. $y \in \cY$,
$(M_t(X,y))_{t \in \T}$ 
is exponentially tight, i.e., for all $a > 0$, 
there exists a compact subset $\K \subset \M_1(\cX \times \cY)$ such that
$$\limsup_{t \to \infty} \frac{1}{t}\log
\mu\paran{M_t(X,y) \in \K^c} \le -a.$$ 
The following generic lemma allows us to push exponential tightness 
from $\M_1(\cX)$ to $\M_1(\cX \times \cY).$
\begin{lemma}\label{lem:exptight}
    Suppose $X \sim \mu$, 
    $ (M_t(X))_{t \in \T}$ 
    is an exponentially tight family of $\M_1(\cX)$-valued random variables. 
    Then for $\nu$-a.e. $y \in \cY$, 
    $ (M_t(X,y))_{t \in \T}$ 
    is an exponentially tight family of $\M_1(\cX \times \cY)$-valued random variables.
\end{lemma}
\begin{proof}
    Note that by Theorem \ref{thm:ergodic}, for $\nu$-a.e. $y \in \cY$, 
    $M_t(y) \convergeto \nu$, 
    so $\set{M_t(y)}_t$ has compact closure.
    Let $y \in \cY$ be such one and $ \K_\cY$ be the closure of $\set{M_t(y)}_t$, 
    which is a compact subset of $\M_1(\cY)$.
    Fix $a > 0$. By exponential tightness of 
    $(M_t(X))_t$, 
    there exists a compact subset $\K_\cX \subset \cX$ such that
    $$\limsup_{t \to \infty} \frac{1}{t}\log
    \mu\paran{M_t(X) \in \K^c_\cX } \le -a.$$
    Consider the set 
    $\K = \set{\lambda \in \M_1(\cX \times \cY) \,: 
    \, \lambda_\cX \in \K_\cX,\, \lambda_\cY \in K_\cY}$, 
    where $\lambda_\cX$ denotes the $\cX$-marginal of $\lambda$, 
    and similar for $\lambda_\cY$. 
    It is easy to check $\K \subset \M_1(\cX \times \cY)$ is 
    closed and tight, hence compact by Prokhorov's Theorem. 
%   \ref{thm:prok}. 
    Then we obtain
    \begin{align*}
        & \limsup_{t \to \infty}\frac{1}{t}\log
            \mu\paran{M_t(X,y) \in \K^c} 
        \\ & \quad \le \limsup_{t \to \infty}\frac{1}{t}\log\paran{
            \mu\paran{M_t(X) \in \K_\cX^c} 
            + \delta_y\paran{M_t(y)
            \in \K_\cY^c}} \\ 
        & \quad = \limsup_{t \to \infty}\frac{1}{t}\log\paran{
            \mu\paran{M_t(X)
            \in \K_\cX^c}} \le -a.
    \end{align*}
\end{proof}

\section{Deferred proofs} \label{append}

\subsection{The variational characterization of posterior consistency} \label{append:var}
To be self-contained, we reproduce the short proof 
of Theorem \ref{thm:posterior} from \cite[Theorem 2]{mcgoff2019gibbs}.
\begin{proof}[Proof of Theorem \ref{thm:posterior}]
    Let $U$ be an open neighborhood of $\Theta_{\min}$.
    Then $E = \Theta \setminus U$ is closed and thus a compact subset of $\Theta$.
    If $\pi_0(E) = 0$, then $\pi_t(E \,|\, y) = 0$ for any $t \ge 0$ and any $y \in \cY$.
    Otherwise, consider the conditional prior probability
    measure $\pi_E = \pi_0(\cdot \,|\, E)$ which is supported on $E$.
    Let $V^* = \inf_{\theta \in \Theta} V(\theta)$.
    Since $E$ is disjoint from $U$, there exists $\epsilon > 0$
    such that $\inf_{\theta \in E} V(\theta) \ge V^* + \epsilon$.
    We apply \eqref{eq:variational} on $\pi_0$ and $\pi_E$.
    Then for $\nu$-a.e. $y \in \cY$, there exists $T > 0$ large enough such that
    for all $t \ge T$, we have
     \[
        -\frac{1}{t} \log \partition^{\pi_0}_t(y) \le V^* + \epsilon/3,
    \]
    \[
        -\frac{1}{t} \log \partition_t^{\pi_E}(y) \ge \inf_{\theta \in E} V(\theta) - \epsilon/3
        \ge V^* + 2\epsilon/3.
    \]
    As a result, for all $t \ge T$, we have
    \begin{align*}
        \pi_t(E \,|\, y) & = \frac{1}{\partition_t^{\pi_0}(y)} 
            \int_{E \times \cX} \exp\left(-L^t_{\theta}(x,y)\right)\,d\mu_\theta(x)\,d\pi_0(\theta) \\
        & = \frac{\pi_0(E)\partition_t^{\pi_E}(y)}{\partition_t^{\pi_0}(y)} \\
        & \le \exp(-V^* t - (2\epsilon/3)t + V^* t + (\epsilon / 3)t) \\
        & \le \exp(-(\epsilon/3)t),
    \end{align*}
    which tends to zero as $t \to \infty$.
\end{proof}

\subsection{The exponentially continuous family} \label{append:exp-cont}
For the proof of Lemma \ref{lem:exp-cont-v}, we adapt the argument  
from \cite[Proposition 1.12]{budhiraja2019analysis}.
\begin{proof}[Proof of Lemma \ref{lem:exp-cont-v}]
    We write 
    \[
        V(t, \theta, y) := - \frac{1}{t} \log \int_\cX \exp(-L_{\theta}^t(x,y))\,d\mu_\theta(x).
    \]
    Then for $\nu$-a.e. $y \in \cY$, $\lim_{t \to \infty} V(t, \theta, y) =: V(\theta)$, 
    and $\lim_{t \to \infty} V(t, \theta_t, y) = V(\theta)$
    for any $\theta_t \convergeto \theta$ by Definition \ref{def:exp-cont}.
    We assume for contradiction that $V$ is not continuous. 
	Then there exists a sequence $(\theta_n)_{n \in \nat}$ and $\theta$ such that 
	$\theta_n \convergeto \theta$ but $\abs{V(\theta_n) - V(\theta)} > \epsilon$ 
	for some $\epsilon > 0$. Take $y \in \cY$ from a set of full $\nu$-measure 
	such that $\lim_{m \to \infty}V(m, \theta_n, y) = V(\theta_n)$ 
	for all $n$ and $\lim_{n \to \infty} V(n, \ttheta_n, y) = V(\theta)$, 
	for any $\ttheta_n \convergeto \theta$. Then for $n \in \nat$, 
	there exists $m_n \ge n$ such that 
	$\abs{V(m_n, \theta_n, y) - V(\theta_n)} < \epsilon / 2$, 
	which implies $\abs{V(m_n, \theta_n, y) - V(\theta)} > \epsilon / 2$ 
	and a contradiction to the fact that we can show $\lim_{n \to \infty}V(m_n,\theta_n, y) = V(\theta).$
\end{proof}

For the proof of Proposition \ref{prop:expcont},
we adapt the argument from \cite[Theorem 2.1 \& 2.2]{dinwoodie1992large}. 
One may also use the approach in \cite{wu2004large}.
\begin{proof}[Proof of Proposition \ref{prop:expcont}]
	We first show one side of the inequality. 
	Let $\theta \in \supp(\pi)$ which is a closed and thus compact subset of $\Theta$,
	and we take $y \in \cY$ from 
 	a set of full $\nu$-measure such that \eqref{eq:var-single-2} and \eqref{eq:exp-cont} hold
	for $\theta$. 
	We claim that for all $\epsilon > 0$ there exists 
	an open set $U_\theta$ containing $\theta$ and $T_\theta \in [0,\infty) $ 
	such that for every $\theta' \in U_\theta$ and $t \ge T_\theta$,
	$$\int_\cX\exp\paran{-L^t_{\theta'}(x,y)}
	d\mu_{\theta'}(x) \ge \exp(-t(V(\theta) + \epsilon)).$$
	If not, we can find a sequence of parameters $\theta_k \convergeto \theta$ 
	and a subsequence $n_k \in \nat$ such that
	$$\int_\cX\exp\paran{-L_{\theta_k}^{n_k}(x, y)}
	d\mu_{\theta_k}(x) < \exp(-n_k(V(\theta) + \epsilon)),$$
	which contradicts to \eqref{eq:exp-cont}. 
	Now since $U_\theta$ is open and $\theta \in \supp(\pi)$, 
	$\pi(U_\theta) > 0$ and so for $t \ge T_\theta$,
	\begin{align*}
		\partition_t^{\pi}(y) & = \int_{\Theta} \int_\cX \exp\paran{-L^t_{\theta'}
			(x, y)}d\mu_{\theta'}(x)\,d\pi(\theta') \\
		& \ge \int_{U_\theta}\int_\cX \exp\paran{-L^t_{\theta'}
			(x, y)}d\mu_{\theta'}(x)\,d\pi(\theta')  \\
		& \ge \int_{U_\theta} \exp(-t(V(\theta) + \epsilon)) \,d\pi(\theta') \\
		& = \pi(U_\theta)\exp(-t(V(\theta) + \epsilon)),
	\end{align*}
	which gives 
	$$ \liminf_{t\to\infty}\frac{1}{t}\log \partition_t^{\pi}(y) \ge -V(\theta) - \epsilon.$$
	Taking $\epsilon$ from a countable sequence that goes to zero and 
	choosing $\theta$ to be a minimizer of $V$ on $\supp(\pi)$, 
	we obtain the desired inequality.

	Next, we show the other side of inequality. 
	Again, let $\theta \in \supp(\pi)$ and we take $y \in \cY$ from 
	a set of full $\nu$-measure such that \eqref{eq:exp-cont} holds. 
	Let $\epsilon > 0$. By the same argument, there exists 
	an open set $U_\theta$ containing $\theta$ and 
	$T_\theta \in [0, \infty)$ such that for every $\theta' \in U_\theta$ 
	and $t \ge T_\theta$,
	$$\int_\cX\exp\paran{-L_{\theta'}^t(x, y)}
	d\mu_{\theta'}(x) \le \exp(-t(V(\theta) - \epsilon)).$$ 
	Since $\supp(\pi)$ is compact, it can be covered by 
	a finite cover $\displaystyle \set{U_{\theta_k}}_{1 \le k \le m}.$ 
	In this case, we take $y \in \cY$ from a set of full $\nu$-measure 
	such that \eqref{eq:var-single-2} and \eqref{eq:exp-cont} 
	hold for $\theta_k$, $1 \le k \le m$. 
	Thus, for $t \ge \max\set{T_{\theta_1}, \ldots, T_{\theta_m}}$,
	\begin{align*}
		\partition_t^{\pi}(y) & = \int_{\Theta} \int_\cX \exp\paran{-L_{\theta'}^t(x, y)}
		d\mu_{\theta'}(x)\,d\pi(\theta') \\
		& \le \sum_{k=1}^m \int_{U_{\theta_k}}\int_\cX \exp\paran{-L_{\theta'}^t(x, y)}
		d\mu_{\theta'}(x)\,d\pi(\theta')  \\
		& \le \sum_{k=1}^m \int_{U_{\theta_k}} \exp(-t(V(\theta_k) - \epsilon))\,
			d\pi(\theta') \\
		& = \sum_{k=1}^m \pi(U_{\theta_k}) \exp(-t(V(\theta_k) - \epsilon)),
	\end{align*}
	which implies
	\begin{align*}
		\limsup_{t\to\infty}\frac{1}{t}\log \partition_t^{\pi}(y) 
			\le \max_{1\le k\le m} -V(\theta_k) + \epsilon 
			\le -\inf_{\theta \in \supp(\pi)} V(\theta) + \epsilon.
	\end{align*}
	Again, taking $\epsilon$ from a countable sequence that goes to zero 
	we obtain the desired inequality.
\end{proof}

\subsection{Technical lemmas for hypermixing processes} \label{append:hm}
For the proof of Lemma \ref{lem:pest}, 
we adapt the arguments in \cite[Lemma 5.4.13]{deuschel2001large}.
\begin{proof}[Proof of Lemma \ref{lem:pest}] 
	Let $\ell > \ell_0$ be given and $r' = \ell + r$. 
	The goal is to show
	\begin{multline*}
	    \int_\cY \log \paran{\int_\cX \exp \paran{f^{nr'}(x,y)} d\mu(x)} d\nu(y) \\
		\le \frac{n}{\alpha(\ell)} \int_\cY \log \paran{\int_\cX \exp 
			\paran{r'\alpha(\ell) f(x, y)} d\mu(x)} d\nu(y).
	\end{multline*}
	Indeed, by dividing both sides by $nr'$ and taking $n \to \infty$, 
	we obtain the desired result.
	We first show the inequality for a Riemann sum of $f^{nr'}$
	and then pass to the integral. Write $s_j^m = \frac{jr'}{m}.$
	Then by generalized H\"older's inequality for the product of
	multiple functions and (\ref{cond:H-1}),
	\begin{align*}
	    \int_\cX & \exp \paran{\sum_{k=0}^{n-1} \frac{r'}{m}
			\sum_{j=0}^{m-1} f \circ (\xmap \times \ymap)^{s_j^m + kr'}(x,y)}  d\mu(x) \\
		& = \int_\cX \prod_{j=0}^{m-1} \exp \paran{\frac{r'}{m}
			\sum_{k=0}^{n-1}f \circ (\xmap \times \ymap)^{s_j^m + kr'}(x,y)} d\mu(x) \\
		& \qquad \le \prod_{j=0}^{m-1} \paran{\int_\cX \exp \paran{r'
			\sum_{k=0}^{n-1}f \circ (\xmap \times \ymap)^{s_j^m + kr'}(x,y)} d\mu(x)}^{1/m} \\
		& \qquad \qquad \le \prod_{j=0}^{m-1} \prod_{k=0}^{n-1} \paran{\int_\cX \exp 
			\paran{r'\alpha(\ell) f(x, \ymap^{s_j^m + kr'}y)} d\mu(x)}^{1/m\alpha(\ell)}.
	\end{align*} 
	By taking the logarithm and integrating with respect to $\nu$ on both sides,
	\begin{align*}
		\int_\cY & \log \paran{\int_\cX \exp \paran{\sum_{k=0}^{n-1} \frac{r'}{m}
			\sum_{j=0}^{m-1} f \circ (\xmap \times \ymap)^{s_j^m + kr'}(x, y)} 
			d\mu(x)} d\nu(y) \\
		& \le \sum_{j=0}^{m-1} \sum_{k=0}^{n-1} \frac{1}{m\alpha(\ell)} 
			\int_\cY \log \paran{\int_\cX \exp 
			\paran{r'\alpha(\ell) f(x, \ymap^{s_j^m + kr'}y)} d\mu(x)}
			d\nu(y) \\
		& = \frac{n}{\alpha(\ell)} \int_\cY \log \paran{\int_\cX \exp 
			\paran{r'\alpha(\ell) f(x, y)} d\mu(x)}
			d\nu(y).
	\end{align*}
	Note that the right hand side does not depend on $m$. For every $(x, y) \in \cX \times \cY$, 
	\[
		\lim_{m\to\infty}\frac{r'}{m}
			\sum_{j=0}^{m-1} f \circ (\xmap \times \ymap)^{s_j^m + kr'}(x,y) 
			= \int_0^{r'}f \circ (\xmap \times \ymap)^{s + kr'} (x,y)\,ds, 
	\]
	and clearly, 
	\[
		\sum_{k=0}^{n-1} \int_0^{r'}f \circ (\xmap \times \ymap)^{s + kr'} (x,y)\,ds
		= \int_0^{nr'} f(\xmap^s x, \ymap^s y)\,ds = f^{nr'}(x,y).
	\]
	Now as $m\to\infty$, by Lebesgue dominated convergence, we obtain
	\begin{multline*}
	    \int_\cY \log \paran{\int_\cX \exp \paran{f^{nr'}(x,y)} d\mu(x)} d\nu(y) \\
		\le \frac{n}{\alpha(\ell)} \int_\cY \log \paran{\int_\cX \exp 
			\paran{r'\alpha(\ell) f(x, y)} d\mu(x)} d\nu(y).
	\end{multline*}
\end{proof}

% \newpage
\begin{proof}[Proof of Lemma \ref{lem:hm-subadd}]
    Note that by the monotone convergnce of relative entropy (Theorem \ref{rel:mono}),
    $F_t(y)$ is left continuous in $t$, so the map $t \mapsto F_t(y)$
    is determined by the values on those $t \in \rat \cap [0, \infty).$
    
	Let $n \in \nat$, $t_1, \ldots, t_n \in [0, \infty)$ and
	$s_k = (k-1)\ell + \sum_{j=1}^{k-1}t_j$.
	Let $f_k : \cX \to \real$ be a bounded $\F_{t_k}$-measurable function.
	Then by the variational property of relative entropy (Theorem \ref{rel:var})
	and the hypermixing property \eqref{cond:H-1},
	\begin{align*}
		& -F_{(n-1)\ell + \sum_{k=1}^{n}t_k}(y) \\
	    \ge & \int_\cX \left(\sum_{k=1}^{n} f_k \circ \xmap^{s_k}\right)
		\,d\lambda_y
		- \log \int_\cX \exp\left(
		\sum_{k=1}^{n} f_k \circ \xmap^{s_k} \right)\,d\mu 
		\\
		\ge & \sum_{k=1}^{n} \left( \int_\cX f_k \circ \xmap^{s_k}
		\,d\lambda_y
		- \frac{1}{\alpha(\ell)} \log \int_\cX \exp\left(
		\alpha(\ell) f_k \circ \xmap^{s_k} \right)\,d\mu \right) 
		\\
		= & 
		\sum_{k=1}^{n} \left( 
		\int_\cX f_k  \,d\lambda_{\ymap^{s_k}(y)}
		- \frac{1}{\alpha(\ell)} \log 
		\int_\cX \exp\left( \alpha(\ell) f_k  \right)\,d\mu 
		\right) 
		\\
		= & \frac{1}{\alpha(\ell)} \sum_{k=1}^{n} \left( \int_\cX \alpha(\ell) f_k 
		\,d\lambda_{\ymap^{s_k}(y)}
		-  \log \int_\cX \exp\left(
		\alpha(\ell) f_k  \right)\,d\mu \right),
	\end{align*}
	where we used, by Proposition \ref{prop:inv-join}, 
	that for $\nu$-a.e. $y \in \cY$, 
	$\lambda_y \circ \xmap^{-t} = \lambda_{\ymap^t y}$ 
	for any $t \in \rat \cap [0, \infty)$, 
	and we used the invariance of $\mu$.
	By taking the supremum of each $f_k$, we obtain 
	\[
		- F_{(n-1)\ell + \sum_{k=1}^{n}t_k}(y) \ge -\frac{1}{\alpha(\ell)} 
	\sum_{k=1}^{n} F_{t_k}(\ymap^{s_k} y).
	\]
\end{proof}

% \newpage
For the proof of \eqref{eq:exp-approx-2} in Lemma \ref{lem:hm-exp-cont},
we basically follows the proof of \cite[Theorem 1.17]{budhiraja2019analysis}.
\begin{proof}[Proof of \eqref{eq:exp-approx-2} in Lemma \ref{lem:hm-exp-cont}]
    Take $y \in \cY$ from the set of full $\nu$-measure such that \eqref{eq:var-single-2} and
	\eqref{eq:exp-approx} hold.
	For every $\epsilon > 0$, 
	\begin{align*}
		\limsup_{t \to \infty} & \, \frac{1}{t}\log
			\expect{\exp\paran{f^t(X^{t}, y)}} \\
		&= \limsup_{t \to \infty} \frac{1}{t}\log
			\left(\expect{\exp\paran{f^t(X^{t}, y)}\one_{\set{d_t(X^{t}, X, y) 
			\le \epsilon}}}\right. \\
		& \hspace{3cm} + \left. \expect{\exp\paran{f^t( X^{t}, y)}
			\one_{\set{d_t(X^t, X, y) > \epsilon}}}\right) \\
		&\le \limsup_{t \to \infty}\frac{1}{t}\log
			\left(\expect{\exp\paran{f^t(X, y)+ t \epsilon}} \right. \\
		& \hspace{3cm} + \left. \exp\paran{t\|f\|}\prob{d_t
			\paran{X^t, X, y} > \epsilon} \right). 
	\end{align*}
	By using the fact \cite[Lemma 1.2.15]{dembo2010large} that if $a_t, b_t \ge 0$, then
	$$\limsup_{t \to \infty} \frac{1}{t}\log{(a_t + b_t)} = 
	\max\set{\limsup_{t \to \infty} \frac{1}{t}\log{a_t},\, \limsup_{t \to \infty} \frac{1}{t}\log{b_t}},$$ we have

	\begin{align*}
		 \limsup_{t \to \infty}&\,\frac{1}{t}\log
			\left(\expect{\exp\paran{f^t(X, y)+ t \epsilon}} \right. \\
		& \hspace{3cm} + \left. \exp\paran{t\|f\|}\prob{d_t
			\paran{X^t, X, y} > \epsilon} \right) \\
		&= \max\left\{\limsup_{t \to \infty}\frac{1}{t}\log
			\expect{\exp\paran{f^t(X, y)+ t \epsilon}}, \right. \\
		& \hspace{3cm} \left. \|f\| + \limsup_{t \to \infty}\frac{1}{t}\log
			\prob{d_t\paran{X^t, X, y} > \epsilon} \right\} \\ 
		&= \limsup_{t \to \infty} \, \frac{1}{t}\log{
			\expect{\exp\paran{f^t(X, y)}}} + \epsilon \\
		& = -V(\theta) + \epsilon.
	\end{align*}
	Similarly,
	\begin{align*}
		\liminf_{t \to \infty} & \, \frac{1}{t}\log
			\expect{\exp\paran{f^t(X^{t}, y)}} \\
		& \ge \liminf_{t \to \infty}\frac{1}{t}\log
			\expect{\exp\paran{f^t(X^{t}, y)}
			\one_{\set{d_t(X^t, X, y) \le \epsilon}}} \\
		& \ge \liminf_{t \to \infty}\frac{1}{t}\log
			\expect{\exp\paran{f^t(X, y) - 
			t\epsilon}\one_{\set{d_t(X^t, X, y) \le \epsilon}}} \\
		& \ge \liminf_{t \to \infty}\frac{1}{t}\log
			\left(\expect{\exp\paran{f^t(X, y) - t\epsilon}} \right. \\
		& \hspace{3cm} - \left. 
			\exp\paran{t(\|f\| - \epsilon)}
			\prob{d_t\paran{X^t, X, y} > \epsilon} \right)^+  \\
		& = \liminf_{t \to \infty}\frac{1}{t}\log{
			\expect{\exp\paran{f^t(X, y) }}} - \epsilon \\
		& = - V(\theta) - \epsilon,
	\end{align*}
	where $a^+ := \max\set{a, 0}$, and the last equality follows from the fact that 
	if $a_t, b_t \ge 0$, 
	$\displaystyle \lim_{t\to\infty} \frac{1}{t}\log{a_t} = a$ in $\real$,  
	and $\displaystyle\limsup_{t\to\infty}\frac{1}{t}\log{b_t} = -\infty$, 
	then $\displaystyle\liminf_{t \to \infty}\frac{1}{t}\log{(a_t - b_t)^+} = a$.
\end{proof}

\section{An extentison of the Subadditive Ergodic Theorem}
\label{sec:subadd}
This section includes auxiliary 
subaddive ergodic results that are adapted to hypermixing processes.
Let $(\cY, \ymap, \nu)$ be a dynamical system
such that $\nu$ is a $\ymap$-invariant measure.
We consider a sequence of measurable functions
$F_t: \cY \to [-\infty, 0]$ that are subadditive in the following sense:
there exists $\ell_0 \ge 0$ and non-increasing $\alpha: [\ell_0, \infty) \to [1,\infty)$
such that $$\lim_{\ell \to \infty} \alpha(\ell) = 1,$$ and 
for all $t_1,t_2, \ldots, t_{n} \in \T$ and $\ell > \ell_0$,
\begin{equation} \label{eq:subadd}
    F_{(n-1)\ell + \sum_{k=1}^{n}t_k}(y) \le \frac{1}{\alpha(\ell)} 
    \sum_{k=1}^{n} F_{t_k}(\ymap^{(k-1)\ell + \sum_{j=1}^{k-1}t_j} y).
\end{equation}
When $\ell = 0$ and $\alpha(\ell) = 1$, $\eqref{eq:subadd}$ 
is the exactly the subadditive property.
One can relate \eqref{eq:subadd} to decomposing the interval of length 
$(n-1)\ell + \sum_{k=1}^{n}t_k$ into $n$ subintervals of
length $t_1$, $t_2$, $\ldots \,$, $t_n$ that are $\ell$ separated from each other.

First, we consider the discrete time case $\T = \nat$, and
$(\cY, \ymap, \nu)$ is a discrete-time dynamical system.
We assume that we already showed the limit 
\[
    \lim_{k \to \infty} \frac{1}{k} \int_\cY F_k(y)\,d\nu(y)
\]
exists. The goal is to show
\[
    \lim_{k \to \infty} \frac{1}{k} \int_\cY F_k(y)\,d\nu(y)
        =  \int_\cY \liminf_{n \to \infty} \frac{F_n(y)}{n}\,d\nu(y)
        =  \int_\cY \limsup_{n \to \infty} \frac{F_n(y)}{n}\,d\nu(y),
\]
which in particular implies that $\lim_{n \to \infty} \frac{F_n(y)}{n}$
exists for $\nu$-a.e. $y \in \cY$ and
\[
     \lim_{k \to \infty} \frac{1}{k} \int_\cY F_k(y)\,d\nu(y)
        =  \int_\cY \lim_{n \to \infty} \frac{F_n(y)}{n}\,d\nu(y).
\]
We adapt the argument in \cite{avila2009subadditive, steele1989kingman} to prove 
the following results. The idea is that $F_t$ can be treated as subadditive
on those $t$ that are $\ell$ apart from each other for $\ell > \ell_0$, and
this extra factor of $\ell$ is negligible if we
consider arbitrarily large time scale $k$. 
\begin{proposition} \label{prop:liminf}
    \[
        \lim_{n \to \infty} \frac{1}{n} \int_\cY F_n(y)\,d\nu(y)
        \le \int_\cY \liminf_{n \to \infty} \frac{F_n(y)}{n}\,d\nu(y).
    \]
\end{proposition}
\begin{proof}
    Let $F_* := \liminf_{n \to \infty} \frac{F_n}{n}$, which is a
    $\ymap$-invariant function. Suppose $F_* \ge -C$ for some constant $C > 0$.
    Fix $\epsilon > 0$, $\ell > \ell_0$ and $1 \le k \le N \le n$.
    Define the set 
        $$E_{k,N} = \set{y \in \cY \,:\, \frac{F_{j}(y)}{j} \le F_*(y) + \epsilon
        \text{ for some } k \le j \le N}.$$
    Fix $y \in \cY$ such that $F_*(y) = F_*(\ymap^t y)$ for all $t \in \nat$.
    Note that $\alpha(\ell) F_{(n - 1)\ell + nk} \le  F_{(n' - 1)\ell + nk}$ if $n' \le n$.
    In a similar way as in \cite{avila2009subadditive, steele1989kingman}, 
    we decompose the interval of length $(n'-1)\ell + nk$ into classes of subintervals 
    that are $\ell$-separated from each other, where $n' \le n$ counts the
    total number of subintervals.
    One class is of the form $[\tau_j, \tau_j + l_j)$ for some $k\le l_j \le N$ 
    and $\ymap^{\tau_j}y \in E_{k,N}$,
     one class is of the form $[\sigma_i, \sigma_i + k)$ where
    $\ymap^{\sigma_i}y \in E_{k,N}^c$,
    and a remainder interval with length at most $\max\set{N, k} = N$.
    
    % Let $n_0 = m_0 = m_{-1} = 0$, 
    Let $n_0 = 0$,
    and we define inductively a sequence of integers 
    depending on $y$
    % $$0 = m_0 \le n_1 < m_1 \le n_2 < m_2 \le \cdots$$
    $$0 = n_0 \le n_1 \le n_2 \le \cdots $$
    by taking
    $$n_j = \min \set{\bar{n} \ge n_{j - 1} \,:\, \ymap^{\bar{n}(\ell + k) + (j - 1)\ell + N_{j-1} } y \in E_{k,N} }$$
    where for convenience, we write $N_0 = 0$ and $ N_j = \sum_{i=1}^j l_i$,
    and we take $l_j$ such that $k \le l_j \le N$ and
    $$F_{l_j}(\ymap^{n_j(\ell + k) + (j - 1)\ell + N_{j-1} } y) \le 
    l_j (F_*(y) + \epsilon).$$
    % where for convenience, we write
    % $$M_j = \sum_{i=0}^j n_i - m_{i-1}, \quad N_j = \sum_{i=0}^j l_i .$$ 
    % $ N_j = \sum_{i=1}^j l_i. $
    Take $r$ be the largest integer such that $kn_r + N_{r} \le nk$.
    Then by \eqref{eq:subadd}, the invariance of $F_*$ on $y$, and $F_t \le 0$,
    \begin{align*}
        F_{(n - 1)\ell + nk}(y)
            & \le \frac{1}{\alpha(\ell)} F_{(n' - 1)\ell + nk}(y) \\
            & \le \frac{1}{\alpha(\ell)^2}
            \sum_{j=1}^r \left(\sum_{i = n_{ j-1}}^{n_j - 1}  F_k\left(\ymap^{i(\ell + k) + (j - 1)\ell + N_{j-1}}y \right)\right. \\ 
             & \qquad \qquad + 
          F_{l_j}\left(\ymap^{n_j(\ell + k) + (j - 1)\ell + N_{j-1}} y\right)\Bigg)
          + \text{ non-positive terms}\\
        & \le \frac{1}{\alpha(\ell)^2}
            \sum_{j=1}^r F_{l_j}\left(\ymap^{n_j(\ell + k) + (j - 1)\ell + N_{j-1}} y\right) \\
        & \le \frac{1}{\alpha(\ell)^2} (F_*(y) + \epsilon) 
            N_r \\
        & \le \frac{1}{\alpha(\ell)^2} (nk\epsilon + F_*(y) N_r), 
    \end{align*}
    where we used the fact that $N_r \le nk$. Note that
    % Note that since $l_j \ge k$, we have 
    % $n_r \le n$, so 
    \[
        nk - k\sum_{i=0}^{(n-1)\ell + nk} \one_{E_{k,N}^c}(\ymap^{i} y) - N
        \le N_r.
    \]
    % for appropriate choices of $s_i$. 
    By the $\ymap$-invariance of $F_*$ and $\nu$,
    \[
        \int_\cY F_*(y) N_r \, d\nu \le (nk - N) \int_\cY F_* \,d\nu  
        - ((n - 1)\ell + nk + 1)k \int_{E_{k,N}^c} F_* \,d\nu.
    \]
    Then we have 
    \[
         \alpha(\ell)^2 \lim_{n \to \infty} \frac{1}{n}\int_\cY F_{(n - 1)\ell + nk}\,d\nu
        \le k \int_\cY F_* \,d\nu  
        - (k + \ell)k \int_{E_{k,N}^c} F_* \,d\nu + k \epsilon.
    \]
    By taking $N \to \infty$, we have $\nu(E_{k,N}^c) \convergeto 0$, so
    \[
        \alpha(\ell)^2 (k + \ell) \lim_{n \to \infty} \frac{1}{n}\int_\cY F_{n}\,d\nu
        \le k \int_\cY F_* \,d\nu + k \epsilon.
    \]
    Then the result follows from dividing both sides by $k$, and taking $k \to \infty$, $\ell \to \infty$ successively.
    
    For the general $F_*$, we define $F^{(C)}_* = \max\set{-C, F_*}$ for $C > 0$. 
    The above argument holds by replacing $F_*$ with $F_*^{(C)}$. Then we 
    have by monotone convergence
    \[
        \lim_{n \to \infty} \frac{1}{n}\int_\cY F_{n}\,d\nu
        \le \lim_{C \to \infty} \int_\cY F_*^{(C)} \,d\nu 
        =  \int_\cY F_* \,d\nu.
    \]
\end{proof}

\begin{lemma} \label{lem:limsup}
    For any $k \in \nat$, $\ell > \ell_0$, for any $y \in \cY$,
    \[
        \limsup_{n \to \infty} \frac{F_{(n-1)\ell + nk}(y)}{n} \ge
        \alpha(\ell) (k + \ell) \limsup_{n \to \infty}\frac{F_n(y)}{n}.
    \]
\end{lemma}
\begin{proof}
    Fix $\ell > \ell_0$. Let $n \in \nat$. Then $n = m_n(k + \ell) - \ell + r$
    for some $m_n, r \in \nat$ such that $\ell \le r < k + 2\ell$. 
    By \eqref{eq:subadd},
    \[
        F_n(y) \le \frac{1}{\alpha(\ell)}(F_{m_n(k + \ell) - \ell}(y) 
            + F_{r - \ell}(\ymap^{m_n(k+\ell)}(y)))
        \le \frac{1}{\alpha(\ell)}F_{m_n(k + \ell) - \ell}(y).
    \]
    By taking $n \to \infty$, we have $m_n \to \infty$ 
    and $\frac{m_n}{n} \to \frac{1}{k + \ell}$, so
    \[
        \limsup_{n \to \infty }\frac{F_n(y)}{n} 
        \le \frac{1}{\alpha(\ell)} \frac{1}{k + \ell} 
        \limsup_{n \to \infty} \frac{F_{n(\ell + k) - \ell}(y)}{n}.
    \]
\end{proof}

\begin{proposition} \label{prop:limsup}
    \[
        \lim_{k \to \infty} \frac{1}{k} \int_\cY F_k(y)\,d\nu(y)
        \ge \int_\cY \limsup_{n \to \infty} \frac{F_n(y)}{n}\,d\nu(y).
    \]
\end{proposition}
\begin{proof}
    Fix $k \in \nat$ and $\ell > \ell_0$. 
    Note that from \eqref{eq:subadd}, for $n \in \nat$,
    \begin{equation} \label{eq:ineq}
        F_{(n-1)\ell + nk}(y) \le 
        \frac{1}{\alpha(\ell)} \sum_{i=0}^{n-1} F_k(\ymap^{i(k+\ell)}y).
    \end{equation}
    Based on this, for $n \in \nat$,
    we define $$G_n^k(y) =  \sum_{i=0}^{n-1} F_k(\ymap^{i(k+\ell)}y),$$
    which is the $n$-th Birkohff sum of $F_k$ with respect to $\ymap^{k+\ell}$.
    Clearly, $G^k_n$ is additive in $n$.
    By the Subadditive Ergodic Theorem,
    $$\int_\cY \limsup_{n \to \infty} \frac{G_n^k}{n}\,d\nu =  
        \lim_{n \to \infty} \int_\cY  \frac{G_n^k}{n}\,d\nu = 
         \int_\cY F_k \,d\nu.$$
    Then by \eqref{eq:ineq} and Lemma \ref{lem:limsup},
    \begin{align*}
         \limsup_{n \to \infty} \frac{G_n^k}{n} & = 
        \,\limsup_{n \to \infty}\frac{1}{n} \sum_{i=0}^{n-1} F_k \circ \ymap^{i(k+\ell)} \\
        & \ge \, \alpha(\ell) \limsup_{n \to \infty} \frac{F_{(n-1)\ell + nk}}{n} \\
        & \ge \, \alpha(\ell)^2 (k + \ell) \limsup_{n \to \infty} \frac{F_{n}}{n}. 
    \end{align*}
    By integrating on both sides, we obtain
    \[
        \frac{1}{k + \ell}\int_\cY F_k \,d\nu 
        \ge \alpha(\ell)^2 \int_\cY \limsup_{n \to \infty} \frac{F_{n}}{n} \,d\nu.
    \]
    The result follows by taking the limit as $k \to \infty$ on the left hand side and then 
    taking $\ell \to \infty$ on the right hand side.
\end{proof}

Now we want to push the subadditive ergodic result to
the continuous-time case. 
Consider now $\T = [0, \infty)$ and
$(\cY, \ymap, \nu)$ is a continuous-time dynamical system.
The argument is inspired by \cite[Theorem 4]{kingman1973subadditive}.
\begin{proposition} \label{prop:dtoc}
    For $y \in \cY$, we have the following connection between
    convergence in discrete time and continuous time:
    \[
        \liminf_{n \to \infty}\frac{F_n(y)}{n} 
         \le \liminf_{t \to \infty}\frac{F_t(y)}{t} 
         \le \limsup_{t \to \infty}\frac{F_t(y)}{t}
        \le \limsup_{n \to \infty}\frac{F_n(y)}{n},
    \]
%   \[
%       \liminf_{n \to \infty} \int_\Y \frac{F_n}{n} \,d\nu
%        \le \liminf_{t \to \infty} \int_\Y \frac{F_t}{t} \,d\nu
%        \le \limsup_{t \to \infty} \int_\Y \frac{F_t}{t} \,d\nu
%       \le \limsup_{n \to \infty} \int_\Y \frac{F_n}{n} \,d\nu.
%   \]
    In particular, $\displaystyle \lim_{t \to \infty}\frac{F_t(y)}{t} = \lim_{n \to \infty}\frac{F_n(y)}{n}$ exists
    for $\nu$-a.e. $y \in \cY$, and
    \[
        \lim_{t \to \infty} \int_\cY \frac{F_t}{t} \,d\nu
        = \int_\cY \lim_{t \to \infty}\frac{F_t}{t} \,d\nu.
    \]
\end{proposition}

\begin{proof}
    Fix $\ell \in \nat$ with $\ell > \ell_0$. Let $t$ be large enough
    so that we don't get negative times below, and let $n \in \nat$ be the
    integer part of $t$. 
    By \eqref{eq:subadd}, we have
    \[
        F_t(y) \le \frac{1}{\alpha(\ell)} (F_{n - \ell}(y) + F_{t-n}(\ymap^ny)),
    	\quad 
        F_{n + \ell + 1}(y) \le \frac{1}{\alpha(\ell)} (F_t(y) + F_{n+1-t}(\ymap^{t+\ell}y)),
    \]
    which implies, using the assumption that $F_t \le 0$,
    \[
        \alpha(\ell) F_{n + \ell + 1}(y) \le F_t(y) 
        \le \frac{1}{\alpha(\ell)} F_{n - \ell}(y).
    \]
    Now it is clear that the result follows.
\end{proof}



%%       \citet{<label>} ==> Jones et al. [21]
%%       \citep{<label>} ==> [21]
%%

\bibliographystyle{siam} 
\bibliography{main2}





% \end{thebibliography}
\end{document}